\documentclass[11pt]{amsart}
\usepackage{amsmath}
\usepackage{amsfonts}
\usepackage{amssymb}
\usepackage{graphicx}
\usepackage{epstopdf}
\usepackage{alias cnt}
\usepackage{overpic}
\usepackage{enumitem,times}
\usepackage[bookmarks=true,pdfstartview=FitH, pdfborder={0 0 0}, colorlinks=true,citecolor=blue, linkcolor=blue]{hyperref}

\newtheorem{theorem}{Theorem}[section]
\newaliascnt{conj}{theorem}
\newaliascnt{cor}{theorem}
\newaliascnt{lemma}{theorem}
\newaliascnt{fact}{theorem}
\newaliascnt{claim}{theorem}
\newaliascnt{prop}{theorem}
\newaliascnt{definition}{theorem}

\newtheorem{conj}[conj]{Conjecture}
\newtheorem{cor}[cor]{Corollary}
\newtheorem{lemma}[lemma]{Lemma}

\newtheorem{claim}[claim]{Claim}
\newtheorem{prop}[prop]{Proposition}
\newtheorem{definition}[definition]{Definition}
\newtheorem{question}[theorem]{Question}
\newtheorem{notation}[theorem]{Notation}
\newtheorem{convention}[theorem]{Convention}

\numberwithin{equation}{subsection}
\numberwithin{theorem}{subsection}
\numberwithin{figure}{subsection}

\newcommand{\be}{\begin{enumerate}}
\newcommand{\ee}{\end{enumerate}}
\newcommand{\op}{\operatorname}

\aliascntresetthe{conj} \aliascntresetthe{cor}
\aliascntresetthe{lemma} \aliascntresetthe{fact}
\aliascntresetthe{claim} \aliascntresetthe{prop}
\aliascntresetthe{definition} 

\theoremstyle{remark}
\newaliascnt{rmk}{theorem}
\newtheorem{example}[rmk]{Example}
\newtheorem{remark}[rmk]{Remark}
\aliascntresetthe{rmk}

\theoremstyle{remark}
\newaliascnt{exam}{theorem}

\aliascntresetthe{exam}

\def\sek~{\S{}}

\DeclareMathOperator{\Sk}{Sk}
\DeclareMathOperator{\std}{std}

\DeclareMathOperator{\aut}{Aut}

\DeclareMathOperator{\symp}{Symp}

\newcommand{\p}{\partial}

\newcommand{\D}{\mathbb{D}}
\newcommand{\R}{\mathbb{R}}

\newcommand{\FF}{\mathcal{F}}

\newcommand{\OO}{\mathcal{O}}

\newcommand{\s}{\vskip.1in}
\newcommand{\n}{\noindent}
\newcommand{\bdry}{\partial}
\newcommand{\ot}{\text{ot}}

\newcommand{\cb}{\color{black}}

\begin{document}

\title{Bypass attachments in higher-dimensional contact topology}

\author{Ko Honda}
\address{University of California, Los Angeles, Los Angeles, CA 90095}
\email{honda@math.ucla.edu} \urladdr{http://www.math.ucla.edu/\char126 honda}

\author{Yang Huang}
\address{Uppsala University, Box 480, 751 06 Uppsala, Sweden \newline
		\indent Aarhus University, 8000 Aarhus C, Denmark}
\email{yang.huang@math.uu.se} \urladdr{https://sites.google.com/site/yhuangmath}

\date{This version: \today}

\thanks{KH is supported by NSF Grants DMS-1406564 and DMS-1549147. YH is partially supported by the grant KAW 2016.0198 from the Knut and Alice Wallenberg Foundation and by the Center of excellence grant 'Center for Quantum Geometry of Moduli Spaces' from the Danish National Research Foundation (DNRF95).}

\begin{abstract}
We initiate a systematic study of convex hypersurface theory and generalize the bypass attachment to arbitrary dimensions.  We also introduce a new type of overtwisted object called the {\em overtwisted orange} which is middle-dimensional and contractible.
\end{abstract}

\maketitle

\tableofcontents

\section{Introduction} \label{sec:intro}

Convex surface theory, as developed by Giroux in \cite{Gir91}, is an extremely powerful tool in studying $3$-dimensional contact topology. (The notion of a {\em convex contact manifold} was introduced earlier by Eliashberg and Gromov \cite{EG91}, and a convex surface is a level set of a convex contact manifold.) The merit of convex surface theory is that it characterizes the contact germ on a surface by discrete data, i.e., an isotopy class of oriented embedded $1$-dimensional submanifolds. A well-known result of Giroux \cite{Gir91} states that a $C^\infty$-generic surface is convex, and, moreover, a generic $1$-parameter family of surfaces, after a slight modification, fails to be convex at only a discrete set of times \cite{Gir00}. The bypass attachment precisely characterizes the codimension $1$ failure of convexity; its theory was developed by the first author in \cite{Hon00} and applied to many classification problems in $3$-dimensional contact topology.

\subsection{Convex hypersurfaces and bypass attachments}

The goal of the current paper is to initiate a systematic study of convex hypersurface theory and generalize the bypass attachment to arbitrary dimensions.

One crucial difference between the theory of convex hypersurfaces in dimension $3$ and higher is that a $C^\infty$-generic hypersurface is convex in dimension $3$ but not in higher dimensions. This fact was known to the experts in the field for a long time and was written down by Mori~\cite{Mor11}.  To the best of our knowledge, the following is still open:

\begin{question}
Given a hypersurface, is there a $C^0$-small perturbation which can be made convex?
\end{question}

What justifies our study, however, is that there is an abundance of convex hypersurfaces in any given contact manifold. For example, the boundary of a standard neighborhood of a closed Legendrian submanifold is always convex; more generally, the boundary of a standard neighborhood of a Legendrian skeleton with mild (e.g., arboreal) singularities is convex. Another natural source of examples of convex hypersurfaces is the Giroux correspondence \cite{Gir02} between contact structures and compatible open book decompositions: all the pages of a compatible open book decomposition are convex hypersurfaces.

It also makes no sense to define bypass attachments as characterizing the co\-dimen\-sion-1 failure of convexity in contact manifolds of dimensions greater than $3$. Instead in this paper we define a bypass attachment as a special type of a smoothly canceling pair of contact handle attachments in the middle dimensions. (If the contact manifold has dimension $2n+1$, then the middle dimensions are $n$ and $n+1$.) This definition is consistent with the description of a $3$-dimensional bypass attachment as a canceling pair of contact handles as explained in \cite{Oz11}.   In particular, the bypass attachment produces a new convex hypersurface from a given one in a way that is described in \autoref{thm:bypass_attachment}. Moreover, we will briefly discuss a slightly more general theory of smoothly canceling pairs of contact handle attachments in \autoref{subsec:smooth_canceling_pair}, leaving details to a future work.\footnote{Kevin Sackel has informed us that he is writing up overlapping results in his thesis.}

Compared to our knowledge of convex surface theory in dimension $3$, little is known in higher dimensions.  One of the first questions that comes to mind is:

\begin{question}
When is the contact germ on a convex hypersurface overtwisted in the sense of \cite{BEM15}?
\end{question}

In this paper a contact manifold which is overtwisted in the sense of \cite{BEM15} will be called {\em overtwisted} or {\em BEM-overtwisted}.
We will say that a convex hypersurface $\Sigma$ is {\em overtwisted} if any small neighborhood of $\Sigma$ is overtwisted.

In dimension $3$, Giroux's criterion asserts that a convex surface $\Sigma$ is overtwisted if and only if the dividing set $\Gamma_{\Sigma} \subset \Sigma$ contains a homotopically trivial loop if $\Sigma \neq S^2$ and $\Gamma_{\Sigma}$ is disconnected if $\Sigma = S^2$. The analogous criterion for convex hypersurfaces in contact manifolds of dimension at least $5$ seems out of reach at this moment. In fact, to the best of the authors' knowledge, not a single example of an overtwisted convex hypersurface in dimension $\geq 5$ was known. Using technical tools such as the overtwisted orange and bypass, which will be explained momentarily, we will construct the first examples of (closed) overtwisted convex hypersurfaces in any dimension in \autoref{sec:application}.

\subsection{Overtwisted oranges and bypasses}

The $h$-principle for overtwisted contact structures was established by Eliashberg \cite{Eli89} in dimension $3$ and by Borman, Eliashberg and Murphy in the groundbreaking work \cite{BEM15} in all dimensions. Soon afterwards, a list of useful criteria for overtwistedness was obtained by Casals, Murphy and Presas in \cite{CMP15}, followed by the work of the second author in \cite{Hua17}. Roughly speaking, the overtwistedness of a contact manifold $(M,\xi)$ is caused by the existence of certain embedded ``overtwisted object''.

Historically there have been many attempts at determining what this overtwisted object should look like, especially in high dimensions.
In dimension $3$, by the work of Eliashberg \cite{Eli89} an overtwisted disk is an embedded $2$-disk $D^2$ in $(M,\xi)$ such that $\xi=TD^2$ along $\bdry D^2$. The overtwisted disk is simultaneously {\em contractible and middle-dimensional.}  The search for middle-dimensional overtwisted objects in higher dimensions started with the work of Niederkr\"uger \cite{Nie06} (also see Gromov~\cite{Gro85}), where the notion of a \emph{plastikstufe} was introduced and shown to obstruct symplectic fillability. A plastikstufe is roughly the product of an overtwisted $2$-disk $D^2$ with an arbitrary closed manifold $S$ of dimension $n-1$ if $\dim M = 2n+1$. An easy dimension count shows that a plastikstufe is indeed middle-dimensional but it is never contractible if $\dim S>0$. The notion of a plastikstufe was later generalized by Massot, Niederkr\"uger and Wendl \cite{MNW13} to more general middle-dimensional objects, called \emph{bordered Legendrian open books} (bLobs), which were also shown to obstruct symplectic fillability; none of the bLobs however are contractible. On the other hand, an overtwisted disk in the sense of \cite{BEM15} is a piecewise smooth codimension 1 disk with a special contact germ, which is middle-dimensional only when $\dim M=3$.

In this paper we introduce yet another overtwisted object which we call an \emph{overtwisted orange}. It is a (singular) middle-dimensional contractible object (or more precisely a stratified manifold) which coincides with the usual overtwisted disk in dimension $3$.  We show in \autoref{thm:orange} that the existence of an embedded overtwisted orange implies BEM-overtwistedness.  The merit of the overtwisted orange, compared to other existing overtwisted objects, is that it is often easier to find in a given contact manifold $(M,\xi)$; see \autoref{sec:application}.

We define the \emph{bypass} to be one half of an overtwisted orange by analogy with the $3$-dimensional case studied in \cite{Hon00}.  In \autoref{sec:orange_bypass} we explain how to attach a bypass to a convex hypersurface and show that such an attachment is consistent with the previously described ``bypass attachments'' defined using pairs of contact handle attachments, and therefore justifies the terminology. Comparing our theory of higher-dimensional bypasses developed in this paper with the 3-dimensional case, there are several notable omissions, e.g., the relationship between bypass attachments and stabilizations of Legendrian submanifolds, as well as the existence of bypasses in ``right-veering'' supporting open book decompositions \cite{HKM07}. These aspects of convex hypersurface and bypass theory will be carried out in future work.

\s\n
{\em Organization of the paper.}
In \autoref{sec:convex_surface} we review some background on Liouville manifolds, Lagrangian cobordisms, and convex hypersurfaces. The material in this section is mostly well-known. In \autoref{sec:contact_handles} we study contact handle attachments, which are contact analogs of Weinstein handle attachments in symplectic topology. \autoref{sec:Legendrian_sum} is devoted to the Legendrian sum and Legendrian boundary sum operations, which are analogous to Polterovich's Lagrangian sum operation \cite{Pol91}, with applications to the Kirby calculus of Legendrian handleslides. Most of the material in this section is also well-known to the experts; see in particular \cite{DG09,CM16,CM16_preprint}. In \autoref{sec:bypass_attachment} we construct the bypass attachment using the contact handle attachments from \autoref{sec:contact_handles}. Some properties and examples of bypass attachments are studied \autoref{sec:example_bypass}. In particular we introduce the notion of a trivial bypass attachment and an overtwisted bypass attachment. The terminology of overtwisted bypass attachment is justified immediately but the triviality of the trivial bypass attachment is postponed to \autoref{sec:partial_OB} after we establish a dictionary between bypass attachments and partial open book decompositions, generalizing the dictionary in dimension $3$ developed in \cite{HKM09}. The overtwisted orange and the bypass are constructed in \autoref{sec:orange_bypass}, where it is shown that an embedded overtwisted orange indeed implies BEM-overtwistedness. Combining essentially all the techniques developed in the previous sections, we finally give examples of overtwisted convex hypersurfaces in any dimension in \autoref{sec:application}.

\s\n
{\em Acknowledgements.} KH is grateful to Yi Ni and the Caltech Mathematics Department for their hospitality during his sabbatical. YH thanks UCLA and Caltech Mathematics Department for their hospitality during his numerous visits. He also thanks Fr\'ed\'eric Bourgeois and Tobias Ekholm for their interest in this work.

\section{Preliminaries on convex hypersurface theory} \label{sec:convex_surface}

\subsection{Liouville manifolds} \label{subsec:Liouville_mfds}

We review the notion of a Liouville manifold and its ideal compactification.

An {\em exact symplectic manifold} $(V,\omega)$ without boundary is a (necessarily open) manifold $V$ with symplectic form $\omega=d\lambda$. Such a 1-form $\lambda$ is called a {\em Liouville form}. The nondegeneracy of $\omega$ implies that there exists a unique vector field $X$, called the {\em Liouville vector field}, which satisfies $i_X \omega = \lambda$, or equivalently $\mathcal{L}_X \omega = \omega$.

A {\em Liouville manifold} is an exact symplectic manifold $(V,\omega,X)$ such that
	\begin{itemize}
		\item the Liouville vector field $X$ is complete, and
		\item the manifold $V$ is {\em convex} in the sense that there exists an exhaustion $V = \cup_{k=1}^{\infty} V^k$ by compact domains $V^k \subset V$ with smooth boundary along which $X$ is outward-pointing.
	\end{itemize}
A Liouville manifold $(V,\omega,X)$ is of {\em finite type} if there exists a compact domain $V^c \subset V$ with smooth boundary along which $X$ is outward pointing and such that $(V\setminus V^c,\lambda|_{V\setminus V^c})$ is diffeomorphic to a positive half-symplectization $([c,\infty)\times Y,e^s\alpha)$, where $Y=\p V^c$, $\alpha=\lambda|_Y$, and $s$ is the $[c,\infty)$-coordinate.  Each component of $[c,\infty)\times Y$ will be called a {\em cylindrical end}.

\s
{\em All Liouville manifolds are assumed to have finite type throughout the paper.}

A {\em Liouville domain} $(W,\omega,X)$ is a compact exact symplectic manifold such that $X$ is outward pointing along $\p W$. Clearly any $V^c \subset V$ is a Liouville domain, and, conversely, any Liouville domain $W$ can be obtained in this way, i.e., there exists a Liouville manifold $V$ which is obtained by attaching a cylindrical end along each connected component of $\p W$, such that $W$ is symplectomorphic to $V^c$. Such $V$ is called the {\em completion} of $W$ and is uniquely defined up to symplectomorphism.

Given a Liouville manifold $(V,\omega,X)$, we define the {\em ideal compactification} $\overline{V}$ as follows: By the previous discussion, we can write $V = V^c \cup ([0,\infty)\times Y)$, where $[0,\infty)\times Y$ is a positive half-symplectization. Then we define
	\begin{equation*}
		\overline{V} := V^c \cup ([0,\infty]\times Y).
	\end{equation*}
Note that $\overline{V}$ is not a symplectic manifold since $\omega$ does not extend to $\p \overline{V}$, but $\p \overline{V} = \{\infty\}\times Y$ is naturally a contact manifold. See \autoref{subsec:convex_surface} below for more discussions on the ideal compactification.

\begin{remark}
	The ideal compactification of a Liouville manifold is also known as the {\em Giroux domain} in the literature; see for example \cite{MNW13} and more recently \cite{Gir17}.
\end{remark}

\subsection{Legendrian isotopy and Lagrangian cylinders} \label{subsec:isotopy_to_cylinder}

In this subsection we discuss the theory of Lagrangian cobordisms induced by a Legendrian isotopy following \cite[Appendix A]{Ek08} and \cite[Section 6.1]{EHK16}. Also see \cite[Section 4.2.3]{EG98} for a different approach to relating Legendrian isotopy and Lagrangian cobordism.

Let $(M,\xi)$ be a contact manifold and choose a contact form $\alpha$ such that $\xi = \ker \alpha$. Consider the symplectization $(\R \times M, d\lambda)$, where $\lambda = e^s \alpha$ and $s$ is the $\R$-coordinate. A Lagrangian submanifold $L \subset \R \times M$ is said to have \emph{cylindrical ends} if there exist (possibly empty) Legendrian submanifolds $\Lambda_{\pm} \subset M$ and large $C \gg 0$ such that
\begin{gather} \label{eqn: ends}
L \cap ([C, \infty) \times M) = [C, \infty) \times \Lambda_+,\\
\nonumber L \cap ((-\infty, -C] \times M) = (-\infty, -C] \times \Lambda_-.
\end{gather}
We say that $L$ is a {\em Lagrangian cobordism from $\Lambda_+$ to $\Lambda_-$} if it has cylindrical ends of the form Equation~\eqref{eqn: ends} for $C\gg 0$.

Consider a smooth Legendrian isotopy $\gamma_{\tau}: \Lambda \to M, \tau \in [0,1]$, i.e., $\Lambda_{\tau} = \gamma_{\tau} (\Lambda)$ is a Legendrian submanifold for all $\tau \in [0,1]$. Assume that $\Lambda$ is compact. We ``reparametrize'' $[0,1]$ by a map $f: \R \to [0,1]$ with small nonnegative derivative such that $f(s) \equiv 0$ for $s \leq -C$ and $f(s) \equiv 1$ for $s \geq C$. Consider the trace
	\begin{equation*}
		\Gamma: \R \times \Lambda \to \R \times M, \qquad (s,x) \mapsto (s, \gamma_{f(s)} (x) )
	\end{equation*}
of the isotopy $\gamma_{f(s)}$. Clearly
	\begin{equation*}
		\Gamma( [C, \infty) \times \Lambda) = [C, \infty) \times \Lambda_1, \qquad
		\Gamma( (-\infty, -C] \times \Lambda ) = (\infty, -C] \times \Lambda_0.
	\end{equation*}

The following result is standard (cf.\ \cite[Lemma 6.1]{EHK16}, for example).

\begin{prop} \label{prop:isotopy_to_cylinder}
	For any $\epsilon>0$, there exist $\delta, K>0$ such that if $\sup_{s \in \R} f'(s) < \delta$ and $C>K$, then there is an exact Lagrangian cobordism from $\Lambda_1$ to $\Lambda_0$ which is $\epsilon$-close (in the $C^0$-metric) to the image of $\Gamma$.
\end{prop}

\begin{proof}[Sketch of proof]
	By subdividing the Legendrian isotopy $\gamma_{\tau}$ into sufficiently many short pieces, we may assume without loss of generality that $M = \R \times T^\ast \Lambda$ equipped with the standard contact form $\alpha_0 = dz - \beta$, where $z$ is the coordinate on $\R$ and $\beta$ is the tautological $1$-form on $T^\ast \Lambda$. Define $H_s(x) = \alpha_0 (\p_s \gamma_{f(s)}(x))$. Let $\pi: \R \times T^\ast \Lambda \to T^\ast \Lambda$ be the projection onto the second factor. Then one can check that the perturbed trace
		\begin{equation*}
			\widetilde{\Gamma}: \R \times \Lambda \to \R \times \R \times T^\ast \Lambda, \quad (s,x) \to \left(s, z(\gamma_{f(s)}(x))+H_s(x), \pi(\gamma_{f(s)}(x)) \right)
		\end{equation*}
	of the isotopy $\gamma_{f(s)}$ has image which is exact Lagrangian in $(\R \times M, d(e^s\alpha_0))$. Moreover, if $\delta>0$ is sufficiently small, then $\widetilde{\Gamma}$ is $C^0$-close to $\Gamma$ and is an embedding since $\Gamma$ is an embedding. Finally it is clear that $\widetilde{\Gamma}$ coincides with $\Gamma$ when $|s| \geq C$ since $H_s$ vanishes there.
\end{proof}

\begin{remark}
It is clear from the proof of \autoref{prop:isotopy_to_cylinder} that, for sufficiently small $\epsilon>0$, the resulting Lagrangian cylinder depends, up to Hamiltonian isotopy, only on the Legendrian isotopy from $\Lambda_0$ to $\Lambda_1$.
\end{remark}

The following corollary will be useful in our construction of the bypass attachment in \autoref{subsec:construct_bypass_attachment}.

\begin{cor} \label{cor:isotop_Lagrangian_disk}
Let $V$ be a Liouville manifold and $D \subset V$ be a Lagrangian plane with a cylindrical end. Let $\overline{V}$ be the ideal compactification of $V$ defined in \autoref{subsec:Liouville_mfds}, and $\Lambda \subset \p \overline{V}$ be the Legendrian sphere such that $\Lambda = \p \overline{D}$. Then for any Legendrian isotopy $\Lambda_\tau$, $\tau \in [0,1]$, in $\p \overline{V}$ with $\Lambda = \Lambda_0$, there is a Lagrangian plane $D_1 \subset V$ with a cylindrical end and $\Lambda_1 = \p \overline{D}_1$, which is unique up to compactly supported Hamiltonian isotopy.
\end{cor}

\begin{proof}
Recall from \autoref{subsec:Liouville_mfds} the decomposition $V = V^c \cup ([0,\infty) \times Y)$, where $Y$ is canonically contactomorphic to $\p \overline{V}$. Assume without loss of generality that $D \cap ([0,\infty) \times Y) = [0,\infty) \times \Lambda$ since $D$ has cylindrical end. Now we apply \autoref{prop:isotopy_to_cylinder} to the Legendrian isotopy $\Lambda_\tau$ to obtain an exact Lagrangian cylinder $L \subset [0,\infty) \times Y$ such that $L \cap ([C,\infty) \times Y) = [C,\infty) \times \Lambda_1$ for $C \gg 0$ and $L \cap ([0,1] \times Y) = [0,1] \times \Lambda$. Then we obtain the Lagrangian plane $D_1 = D^c \cup L$ with cylindrical end as desired, where $D^c = D \cap V^c$.
\end{proof}

\subsection{Convex hypersurfaces} \label{subsec:convex_surface}

In this subsection we review the theory of convex hypersurfaces in contact manifolds following Giroux \cite{Gir91}.  We use the notation $(-\epsilon,\epsilon)_t$, $\R_t$, $\R^2_{x,y}$ to indicate $(-\epsilon,\epsilon)$, $\R$, $\R^2$ with coordinates $t$, $t$, and $(x,y)$.

Let $(M,\xi)$ be a contact manifold with contact form $\alpha$. A hypersurface $\Sigma \subset M$ is {\em convex}\footnote{Unfortunately, we will use the notion of convexity both in the contact and the symplectic settings (cf.\ \autoref{subsec:Liouville_mfds}), and they are not directly related. In what follows it will be clear from the context which notion of convexity we are talking about.} if there exists a contact vector field $v$ defined in a neighborhood of $\Sigma$ which is everywhere transverse to $\Sigma$. Define the $v$-{\em dividing set}
$$\Gamma_{\Sigma,v}=\{\alpha(v)=0\} \subset \Sigma,$$
which is naturally a codimension 2 contact submanifold in $M$ with contact form $\alpha|_{\Gamma_{\Sigma,v}}$. Note that the set of transverse contact vector fields is contractible and hence the isotopy class of $\Gamma_{\Sigma,v} \subset \Sigma$ is independent of the choice of $v$.
We also have a decomposition
$$\Sigma \setminus \Gamma_{\Sigma,v} = R_+(\Sigma,v) \cup R_-(\Sigma,v)$$
into positive and negative regions, where
$$R_+(\Sigma,v)=\{\alpha(v) > 0\} \quad \mbox{and} \quad R_-(\Sigma,v)=\{\alpha(v) < 0\}.$$
We suppress $\Sigma$ (resp.\ $v$) from the notation $\Gamma_{\Sigma,v}$ and $R_\pm(\Sigma,v)$ if $\Sigma$ is understood (resp.\ the particular choice of $v$ is understood or irrelevant).

The following lemma characterizes the symplectic geometry of $R_\pm \subset \Sigma$:

\begin{lemma} \label{lem:ideal_compactification}
Let $\Sigma$ be a closed convex hypersurface and $v$ be a transverse contact vector field. Then $\overline{R}_+ = \{\alpha(v) \geq 0 \}$ and $\overline{R}_- = \{\alpha(v) \leq 0 \}$ are ideal compactifications of Liouville manifolds.
\end{lemma}

\begin{proof}
Identify a collar neighborhood of $\Sigma$ with $\Sigma \times (-\epsilon,\epsilon)_t$ such that the contact form can be written as $\alpha=fdt+\beta$, where $f \in C^\infty(\Sigma)$ and $\beta \in \Omega^1(\Sigma)$ are independent of $t$. Then $\p_t$ is a contact vector field which is everywhere transverse to $\Sigma = \Sigma \times \{0\}$. By definition the $\p_t$-dividing set $\Gamma$ is given by $\{f=0\}$ and $R_+=\{f > 0\}$ and $R_-=\{f < 0\}$. Now identify a collar neighborhood $N(\Gamma)$ of $\Gamma \subset \Sigma$ with $\Gamma \times (-\epsilon,\epsilon)_\tau$ such that $\Gamma$ is identified with $\Gamma \times \{0\}$. Without loss of generality we can arrange that:
\begin{enumerate}
\item[(N)] $f=\tau$ and $\beta=\beta_0$ on $N(\Gamma)$, where $\beta_0\in \Omega^1(\Gamma)$ is independent of $\tau$ and $t$.
\end{enumerate}

Define the Liouville forms on $R_\pm$ by $\lambda_\pm=\beta/f$, respectively. Then, on a neighborhood $\Gamma \times [0,\epsilon)$ of $\p \overline{R}_+ \subset \overline{R}_+$, the Liouville form is $\beta_0/\tau$ and the Liouville vector field is given by $-\tau\p_\tau$. This shows that $\p \overline{R}_+$ is indeed the ideal boundary of $R_+$. The same argument applies to $R_-$.
\end{proof}

\begin{cor}
The induced contact structure on each component of the dividing set $\Gamma$ is semi-fillable and in particular is tight, i.e., not overtwisted in the sense of \cite{BEM15}.
\end{cor}

\begin{proof}
The semi-fillability is immediate. A PS-overtwisted contact structure is not semi-fillable by Niederkr\"uger~\cite[Theorem 1]{Nie06} (proved in the semi-positive case) and Albers-Hofer~\cite[Remark 1.2]{AH09} (proof indicated in the general case).  The equivalence of PS-overtwisted and BEM-overtwisted was proven in \cite{CM16} and \cite{Hua17}.
\end{proof}

\begin{remark}
Note that in $3$-dimensional contact geometry, the Liouville vector fields on $R_\pm$ defined above direct the characteristic foliation: Given a convex surface $\Sigma$ in a contact $3$-manifold, we pick a positive area form $\omega$ on $\Sigma$. Then the characteristic foliation is directed by a vector field $X$ given by $i_X \omega=\alpha|_\Sigma$ and $X$ always flows from $R_+$ to $R_-$.
\end{remark}

Let $\Sigma$ be a closed convex hypersurface with a neighborhood $\Sigma\times \R_t$ such that $\Sigma=\Sigma\times\{0\}$, $v=\bdry_t$ is a transverse contact vector field, and $\Gamma$ is the dividing set of $\Sigma$. Let $fdt+\beta$ be the contact form on $\Sigma\times \R_t$, such that $f\in C^\infty(\Sigma)$ and $\beta\in \Omega^1(\Sigma)$ are independent of $t$ and (N) holds, and let $\lambda_\pm=\beta/f$ be the Liouville form on $R_\pm$.

\begin{lemma}[Flexibility] \label{lemma: flexibility}
Given another contact form $fdt+\beta'$ on $\Sigma\times \R$ such that
\be
\item[(i)] $\beta'\in \Omega^1(\Sigma)$ is independent of $t$ and agrees with $\beta$ near $\Gamma$, and
\item[(ii)] $\beta'$ is exact deformation equivalent to $\beta$,
\ee
there exists a $1$-parameter family of diffeomorphisms $\phi_s: \Sigma\times\R\stackrel\sim\to \Sigma\times \R$, $s\in[0,1]$, such that
\be
\item $\phi_0=id$;
\item $\phi_1$ takes $\ker (fdt+\beta')$ to $\ker (fdt +\beta)$;
\item $\phi_s$ commutes with $\bdry_t$ (and hence $\phi_s(\Sigma\times\{0\})$ is transverse to $\bdry_t$); and
\item $\phi_s=id$ near $\Gamma\times \R$.
\ee
\end{lemma}

\begin{proof}
Immediate consequence of Moser technique.
\end{proof}

\cb

\begin{lemma}[Legendrian realization]\label{lemma: Legendrian realization}
Given an open Lagrangian disk $D_+\subset (R_+,d\lambda_+)$ that has a cylindrical end over a Legendrian sphere $\Lambda_+\subset \Gamma=\bdry R_+$, there exists a convex surface $\Sigma'\subset \Sigma\times \R$ such that:
\begin{enumerate}
\item $\Sigma'$ is graphical over $\Sigma\times\{0\}$ and hence is transverse to $v$;
\item $\Sigma'$ agrees with $\Sigma\times\{0\}$ on a neighborhood of $\Gamma$; and
\item the lift $\widetilde D_+$ of $D_+$ to $\Sigma'$ is Legendrian.
\end{enumerate}
\end{lemma}

By a lift $\widetilde D_+$ we mean that $\widetilde D_+$ is taken to $D_+$ diffeomorphically under the projection $\pi:\Sigma\times \R\to \Sigma$.

\begin{proof}
Note that $D_+$, viewed as a submanifold of $\Sigma\times \R$, is already Legendrian on the cylindrical end.  Let $N(\Gamma)$ be collar neighborhood of $\Gamma\subset \Sigma\times\{0\}$ on which $D_+\cap N(\Gamma)$ is cylindrical.  Writing $\widetilde D_+|_{U}$ for the lift of $D_+$ over the subset $U$, we can set $\widetilde D_+|_{D_+\cap N(\Gamma)}=D_+\cap N(\Gamma)$.

Next we lift the exact Lagrangian $D_+':=D_+\cap (R_+\setminus N(\Gamma))$ with respect to $d\lambda_+$ to a Legendrian $\widetilde D_+'$ with respect to
$$dt+\beta/f=dt+ \lambda_+,$$
so that $\widetilde D_+'$ agrees with $\widetilde D_+|_{D_+\cap N(\Gamma)}$. Fix $p\in \bdry D_+'$. For any $q\in D_+'$, let $\gamma:[0,1]\to D_+'$ be a path from $p$ to $q$.  Then we define
$$\widetilde D'_+|_{\{q\}}= \{(-\textstyle\int_\gamma\lambda_+,q)\}.$$
The definition is independent of the choice of $\gamma$ since $D_+$ is Lagrangian with respect to $d\lambda_+$.

The lemma follows from observing that $\widetilde D_+$ can be extended to $\Sigma'$ satisfying (1) and (2).
\end{proof}

The following lemma, whose proof is left to the reader, is useful for adjusting boundaries of Legendrian submanifolds:

\begin{lemma}[Contact parallel transport]\label{lemma: contact parallel transport}
Let $\R^2\times M^{2n-1}$ be a contact manifold with contact form $\lambda+\beta$, where $M^{2n-1}$ is a closed contact manifold with contact form $\beta$ and $\lambda$ is a Liouville form on $\R^2$, and let $\phi_t$, $t\in \R$, be the flow on $M$ corresponding to the Reeb vector field for $\beta$.  If $\Lambda\subset M$ is a Legendrian submanifold and $\gamma:[0,1]\to \R^2$ is a path, then there exists a Legendrian embedding
$$\phi: [0,1]\times \Lambda \to \R^2\times M, \quad (s,x)\mapsto (\gamma(s), \phi_{-\int_0^s \gamma^*\lambda} (x)).$$
\end{lemma}

\cb
In higher-dimensional contact topology, it is often important to understand the sizes of neighborhoods of submanifolds in a given contact manifold; see for example \cite{NP2010,CMP15}. In the following we discuss the neighborhood size of convex hypersurfaces, as well as the dividing sets. Both results will be crucial in our later study of bypass attachments.

\begin{lemma} \label{lemma: hypersurface_nbhd_size}
Let $\Sigma = \Sigma \times \{0\}$ be a convex hypersurface in  $\Sigma \times \R$ equipped with an $\R$-invariant contact structure. Then, for all $a > b > 0$, there exists a contact embedding $i: \Sigma \times (-a, a) \to \Sigma \times (-b, b)$ such that $i|_{\Sigma \times \{0\}} = \text{\em id}_{\Sigma \times \{0\}}$. In other words, any convex hypersurface has an arbitrarily large invariant neighborhood.
\end{lemma}

\begin{proof}
By assumption, $\p_t$ is an $\R$-invariant contact vector field on $\Sigma \times \R_t$. Let $H: \Sigma \times \R \to \R$ be the contact Hamiltonian generating $\p_t$. Then $H$ is independent of $t$. Given $b > 0$, let $\rho = \rho(t)$ be a cut-off function supported in $(-b, b)$ and $\rho(t) \equiv 1$ for $t \in (-b/2, b/2)$. Then the new contact Hamiltonian $H' = \rho H$ generates a contact vector field $v_{H'}$ which is supported in $\Sigma \times (-b, b)$ and agrees with $\p_t$ near $\Sigma \times \{0\}$. Now the flow of $v_{H'}$ gives the desired contact embedding of $\Sigma \times (-a, a)$ into $\Sigma \times (-b, b)$ for arbitrarily large $a$.
\end{proof}

As an immediate corollary of \autoref{lemma: hypersurface_nbhd_size}, we have the following estimate on the neighborhood size of the dividing set.

\begin{cor} \label{cor:dividing_set_nbhd_size}
Let $\Sigma \subset (M,\xi)$ be a convex hypersurface with dividing set $\Gamma$. Fix a contact form $\beta$ on $\Gamma$ such that $\xi|_{\Gamma} = \ker \beta$ and consider $\Gamma \times \R^2_{x,y}$ equipped with the contact form $\alpha = \beta -ydx$. Let $D(R) \subset \R^2$ be the open disk of radius $R$. Then for any $R>0$ there exists a contact embedding $$j: (\Gamma \times D(R), \ker \alpha) \to (M, \xi)$$ such that $j|_{\Gamma \times \{0\}} = \text{\em id}_{\Gamma}$.
\end{cor}

\begin{remark}
\autoref{cor:dividing_set_nbhd_size} should be compared with the results on the neighborhood size of overtwisted contact submanifolds. For example, it is shown in \cite{CMP15} that, roughly speaking, if $(M, \xi)$ contains an overtwisted contact submanifold with a sufficiently large neighborhood, then $(M, \xi)$ is itself overtwisted.
\end{remark}

\section{Contact handle attachments} \label{sec:contact_handles}

Our construction of a bypass attachment consists of two smoothly canceling contact handle attachments in the middle dimensions. We first examine the contact handles in detail.

\subsection{Contact $n$-handle attachment} \label{subsec:n_handle}

A {\em smooth $n$-handle} $H_n$ of dimension $2n+1$ is a bi-disk $H_n=\left\{ |x|^2 \leq 1 \right\} \times \left\{ z^2 + |y|^2 \leq 1 \right\} \subset \R^{2n+1}$, where $x=(x_1,\dots,x_n), y=(y_1,\dots,y_n)$, and $z$ are coordinates on $\R^{2n+1}$ and $|{\cdot}|$ denotes the Euclidean norm. The $n$-disk $D=\left\{ z=y=0 \right\} \subset H_n$ is called the {\em core disk}, and the $(n+1)$-disk $D'=\left\{ x=0 \right\} \subset H_n$ is called the {\em cocore disk}.

A {\em contact $n$-handle} is a triple $(H_n, \alpha_n, v_n)$, where $H_n$ is a smooth $n$-handle, $\alpha_n = dz - 2y \cdot dx - x \cdot dy$ is a contact form, and $v_n = -x \cdot \p_x + 2y \cdot \p_y + z \p_z$ is a contact vector field. (In fact we have $\mathcal{L}_{v_n} \alpha_n = \alpha_n$.) Note that $v_n$ is gradient-like for the Morse function $z^2+|y|^2-|x|^2$ with a unique critical point of index $n$ at the origin.  The core disk $D$ of $H_n$ is Legendrian and there is a contact embedding $H_n\hookrightarrow J^1(D^n)$, where $J^1(D^n)$ is the standard $1$-jet space of $D^n$ and $D$ is taken to the zero section.

We decompose the boundary $\p H_n = \p_1 H_n \cup \p_2 H_n$, where
$$\p_1 H_n = \left\{ |x| = 1 \right\} \times \left\{ z^2 + |y|^2 \leq 1 \right\},$$
$$\p_2 H_n = \left\{ |x| \leq 1 \right\} \times \left\{ z^2 + |y|^2 = 1 \right\}.$$
Then both $\p_1 H_n$ and $\p_2 H_n$ are convex with respect to $v_n$, with $v_n$ pointing into $H_n$ along $\p_1 H_n$ and out of $H_n$ along $\p_2 H_n$.

We first describe $\p_1 H_n$ in some detail. Noting that $\alpha_n(v_n)=0$ if and only if $z=0$, the $v_n$-dividing set is
\begin{equation*}
	\Gamma_{\p_1 H_n} = \left\{ z=0 \right\} \times \left\{ |x| =1 \right\} \times \left\{ |y| \leq 1 \right\},
\end{equation*}
equipped with the induced contact form $\alpha_n|_{\Gamma_{\p_1 H_n}} = -2y \cdot dx - x \cdot dy$. Consider $J^1(S^{n-1})$ with the standard contact form $du - p \cdot dq$, where $u \in \R, q,p \in \R^n$ are coordinates such that $|q|=1$ and $p \cdot q = 0$. The following is immediate:

\begin{claim} \label{claim: contact embedding}
There exists a contact embedding
\begin{gather}\label{eqn:attaching_region_n}
	\phi: (\Gamma_{\p_1 H_n}, \alpha_n|_{\Gamma_{\p_1 H_n}}) \hookrightarrow J^1(S^{n-1}),\\
\nonumber (x,y)\mapsto (u,q,p)=(-x\cdot y,x, y-(x\cdot y)x),
\end{gather}
which sends the boundary $\bdry D$ of the core disk to the $0$-section of $J^1(S^{n-1})$.
\end{claim}

\begin{claim} \label{claim: contact embedding 2}
If we identify $J^1(S^{n-1})$ with the ideal boundary of a positive half-symplectization $[c,\infty)\times J^1(S^{n-1})$, then $\overline{R_+ (\p_1 H_n)}$ is identified with a (half) tubular neighborhood of $S^{n-1} \subset  \left\{ \infty \right\}\times J^1(S^{n-1}) \subset [c,\infty]\times J^1(S^{n-1})$, where the first inclusion is given by the $0$-section.  A similar identification holds for $\overline{R_- (\p_1 H_n)}$.
\end{claim}

Next let us turn to $\p_2 H_n$. The $v_n$-dividing set is
\begin{equation*}
	\Gamma_{\p_2 H_n} = \left\{ z=0 \right\} \times \left\{ |x| \leq 1 \right\} \times \left\{ |y| = 1 \right\},
\end{equation*}
and $R_+ (\p_2 H_n) = \left\{ |x|<1 \right\} \times \left\{ z^2 + |y|^2 = 1, z > 0 \right\}$.

\begin{claim} \label{claim: def equiv}
$(R_+ (\p_2 H_n), \alpha_n|_{R_+ (\p_2 H_n)})$ and $(R_- (\p_2 H_n), \alpha_n|_{R_- (\p_2 H_n)})$ are exact deformation equivalent (relative to the boundary) to the Weinstein $n$-handle $(T^\ast D^n, -2p \cdot dq - q \cdot dp)$.
\end{claim}

\begin{proof}
Let us use coordinates $(x,y)$ on $R_+(\p_2 H_n)$.  View $z$ as a function $z(x,y)$. Then
$\alpha_n|_{R_+ (\p_2 H_n)}=dz-2y\cdot dx - x\cdot dy$.  We can then interpolate from $dz-2y\cdot dx - x\cdot dy$ to $-2y\cdot dx - x\cdot dy$ by taking $d(tz) -2y\cdot dx-x\cdot dy$, $t\in[0,1]$.
\end{proof}

We now describe the attachment of a contact $n$-handle to a convex hypersurface. \cb

\begin{prop} \label{prop:n_handle}
Let $(M,\xi)$ be a $(2n+1)$-dimensional contact manifold with convex boundary $\Sigma = \p M$. Fix an outward-pointing contact vector field $v$ along $\Sigma$, let $\Gamma$ be the $v$-dividing set, viewed as a codimension 2 contact submanifold, and let $\Lambda$ be an $(n-1)$-dimensional Legendrian sphere in $\Gamma$. Let $\Sigma \setminus \Gamma = R_+ \cup R_-$ be the decomposition into positive and negative regions as in \autoref{subsec:convex_surface}. Then a contact $n$-handle can be attached to $(M, \xi)$ along $\Lambda$ to yield a new contact manifold $(M', \xi')$ with convex boundary $\Sigma' = \p M'$ such that if we write $\Sigma' \setminus \Gamma' = R'_+ \cup R'_-$, then:
		\begin{enumerate}
			\item as a Liouville manifold, $R'_+$ is the completion of a Liouville domain obtained by attaching a Weinstein $n$-handle to $R^c_+$ along (a parallel copy of) the Legendrian sphere $\Lambda \subset \p R^c_+$;
		
			\item as a Liouville manifold, $R'_-$ is the completion of a Liouville domain obtained by attaching a Weinstein $n$-handle to $R^c_-$ along (a parallel copy of) the Legendrian sphere $\Lambda \subset \p R^c_-$;
		
			\item as a contact manifold, $\Gamma'$ is obtained from $\Gamma$ by performing a contact $(-1)$-surgery along $\Lambda \subset \Gamma$; and

			\item as a smooth manifold, $\Sigma'$ is diffeomorphic to the union of $\overline{R'_+}$ and $\overline{R'_-}$, glued along their boundaries by the identity map.
		\end{enumerate}
\end{prop}

\begin{proof}
Fix an identification $i:\bdry D\stackrel\sim\to\Lambda$. By \autoref{claim: contact embedding}, $i$ can be extended to a contactomorphism $i:\Gamma_{\p_1 H_n}\stackrel\sim\to N_\Gamma(\Lambda)$, where $N_\Gamma(\Lambda)$ is a standard tubular neighborhood $N_\Gamma(\Lambda)$ of $\Lambda$ in $\Gamma$. Next, by \autoref{lem:ideal_compactification} and \autoref{claim: contact embedding 2}, $i$ extends to an identification of contact germs on $\p_1 H_n$ and a tubular neighborhood of $N_\Gamma(\Lambda)$ in $\Sigma$. Finally, using $i$ we attach $H_n$ to $(M,\xi)$ along $\Lambda \subset \Sigma$ and round the corners to obtain $(M',\xi')$ with convex boundary.  We may need to slightly adjust the contact vector field $v_n$ on $H_n$ so that it agrees with $v$ on the overlap.

(1), (2) follow from \autoref{claim: def equiv} which identifies $R_+ (\p_2 H_n)$ with (the completion of) a Weinstein handle. Since $H_n$ is attached along $\Gamma$, $R'_+ = R_+ (\Sigma')$ is the completion of a Liouville domain obtained by attaching a Weinstein handle to $R^c_+$ along $\Lambda$, viewed as a subset of $\p R^c_+$. (3) and (4) are immediate.
\end{proof}

\begin{remark} \label{rmk:corner_rounding}
The rounding of contact handles is completely analogous to rounding Weinstein handles in the symplectic case; see \autoref{subsec:(n+1)_handle} for more details.
\end{remark}

\subsection{Contact $(n+1)$-handle attachment} \label{subsec:(n+1)_handle}

It is well-known that a gradient-like Liouville vector field for some Morse function cannot have critical points of index greater than half of the dimension of the symplectic manifold. In the contact case, however, there is no such restriction. Namely, a gradient-like contact vector field for some Morse function can have critical points of arbitrary index.

A quick way to define a contact $(n+1)$-handle is by viewing it as an upside down contact $n$-handle. More precisely, take a contact $n$-handle $(H_n, \alpha_n, v_n)$ and consider the triple $(H_n, \alpha_n, -v_n)$. Now observe that $-v_n$ is a gradient-like contact vector field for the Morse function $|x|^2-|y|^2-z^2$, which has a unique critical point of index $n+1$ at the origin.

This point of view, however, has the drawback that the contact germ on the attaching region of the $(n+1)$-handle is harder to characterize. Instead we consider the standard contact neighborhood of an $(n+1)$-dimensional disk foliated by $n$-dimensional Legendrian disks. This approach, as we will see below, is more complicated but will be useful when we construct the bypass attachment.

\subsubsection{$\Theta$-disks}

Consider the unit disk $\Theta = \left\{ z^2 + |x|^2 \leq 1 \right\} \subset \R^{n+1}_{z,x}$ with $1$-form
\begin{equation} \label{eqn:unknot_foliation}
	\beta = (z^2 - |x|^2 +1) dz + 2zx \cdot dx= (z^2-|x|^2+1) dz + 2z d(|x|^2).
\end{equation}
Then $\FF = \ker \beta$ defines a singular, radially (in the $|x|$ variable) symmetric foliation on $\Theta$ such that $\p \Theta$ is a closed leaf, all other leaves are disks, and the singular locus of $\FF$ is precisely $\left\{ z=0, |x|=1 \right\} \subset \p \Theta$; see \autoref{fig:std_foliation}.

\begin{figure}[ht]
	\begin{overpic}[scale=.23]{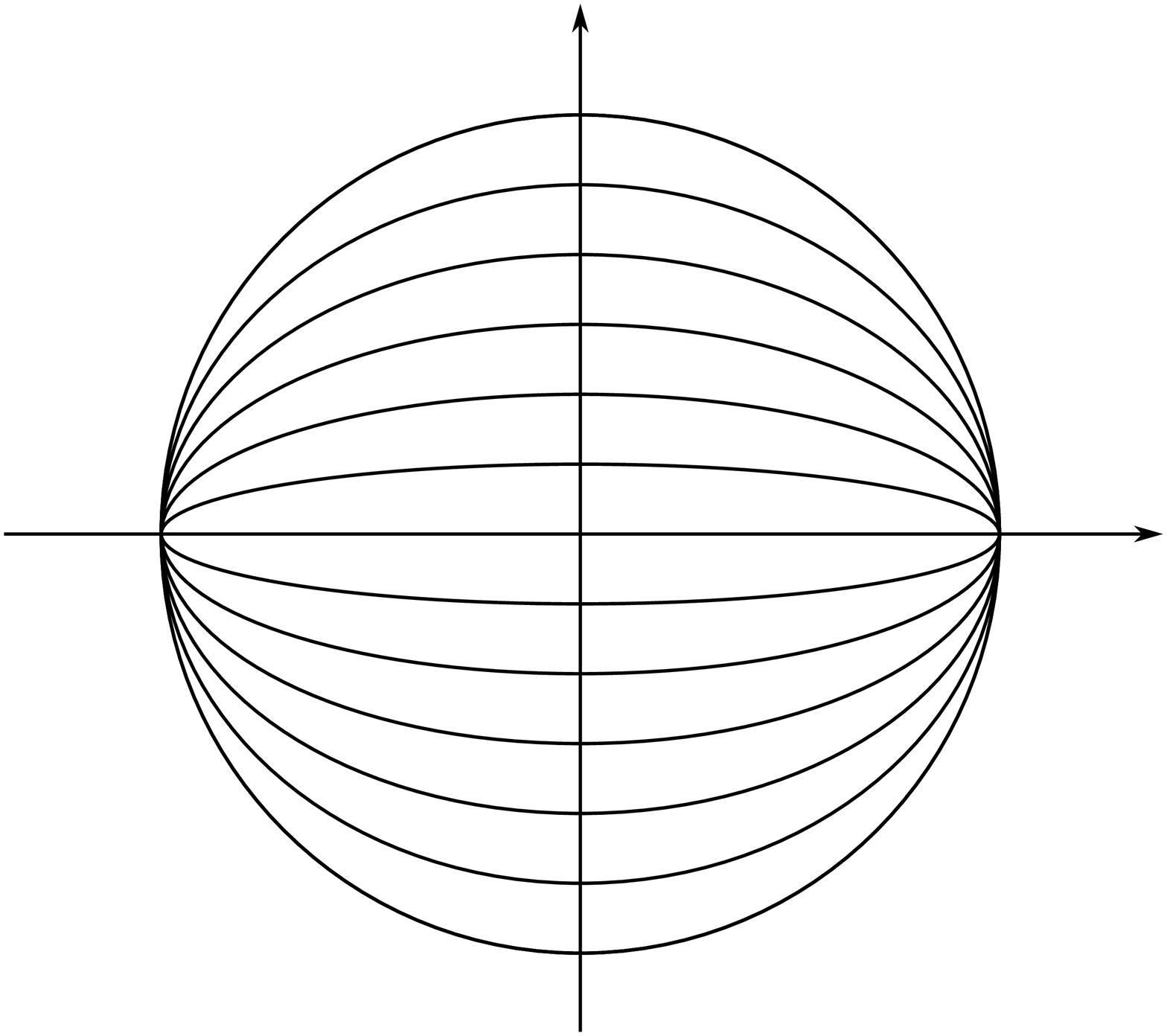}
		\put(44.5,88){$z$}
		\put(98,38){$x$}
	\end{overpic}
	\caption{A $\Theta$-disk.}
	\label{fig:std_foliation}
\end{figure}

According to \cite[Lemma~3.8]{Hua15}, there exists a unique contact germ on $\Theta \subset \Theta \times \R^{n}_{y}$, embedded as the $0$-section, such that all leaves of $\FF$ are Legendrian. \cb

\begin{definition}
An {\em $(n+1)$-dimensional $\Theta$-disk} is an embedded $(n+1)$-dimensional disk $D$ in a contact $(2n+1)$-manifold $(M,\xi)$ such that $\xi\cap TD$ is diffeomorphic to $\FF$ on the model disk $\Theta$, given above.
\end{definition}

\subsubsection{Definition of the $(n+1)$-handle}

Define the contact form
	\begin{equation*}
		\alpha_{n+1} = (z^2 - |x|^2 +1) dz + 2zx \cdot dx + x \cdot dy
	\end{equation*}
on $\Theta \times \R^n$. Indeed $\alpha_{n+1}$ is contact by the following calculation:
	\begin{equation*}
		\alpha_{n+1} \wedge (d\alpha_{n+1})^n = (z^2 + 3|x|^2 + 1) dz \wedge (dx \wedge dy)^n > 0.
	\end{equation*}
Observe $\alpha_{n+1}|_{\Theta} = \beta$, which implies that all leaves of $\FF$ are Legendrian as desired. Moreover since each $\p_{y_i}$ is a contact vector field, it follows from the proof of \autoref{lemma: hypersurface_nbhd_size} that we can assume $\Theta \subset \Theta \times \R^n$ has an arbitrarily large neighborhood with respect to the contact form $\alpha_{n+1}$.

Next we construct a contact vector field $v_{n+1}$ on $(\Theta \times \R^n, \alpha_{n+1})$ which is gradient-like for some Morse function on $\Theta \times \R^n$ with a unique critical point of index $n+1$ at the origin.  Let $v_{n+1}$ be the contact vector field associated to the contact Hamiltonian function $f(x,y,z) = -2z + x \cdot y$ and contact form $\alpha_{n+1}$ (i.e., $\alpha_{n+1}(v_{n+1})=f$).  One can verify that
\begin{align*}
	(z^2 + 3|x|^2 +1) v_{n+1}  = & -(2 - 2|x|^2) z \p_z - (z^2 - |x|^2 +3) x \cdot \p_x\\
 & \qquad + \left( (z^2 + 3|x|^2 +1) y - 4zx \right) \cdot \p_y.
\end{align*}
Observe that $v_{n+1}$ has a unique zero at the origin. Moreover, in a small neighborhood of the origin, $v_{n+1}$ is approximated, up to first order and rescaling, by
	\begin{equation*}
		\widetilde{v}_{n+1} = -2z \p_z -3x \cdot \p_x + y \cdot \p_y.
	\end{equation*}
Since $\widetilde{v}_{n+1}$ has a nondegenerate zero at the origin of index $n+1$, so does $v_{n+1}$.

Now we define the contact $(n+1)$-handle to be the triple $(H_{n+1}, \alpha_{n+1}, v_{n+1})$, where $\alpha_{n+1}, v_{n+1}$ are as defined above and
$$H_{n+1} = \left\{ z^2 + |x|^2 \leq 1 \right\} \times \left\{ |y| \leq 2 \right\} \subset \Theta \times \R^n.$$
As in \autoref{subsec:n_handle}, let us decompose the boundary $\p H_{n+1} = \p_1 H_{n+1} \cup \p_2 H_{n+1}$ such that
\begin{gather*}
\p_1 H_{n+1} = \left\{ z^2 + |x|^2 = 1 \right\} \times \left\{ |y| \leq 2 \right\},\\
\p_2 H_{n+1} = \left\{ z^2 + |x|^2 \leq 1 \right\} \times \left\{ |y| = 2 \right\}.
\end{gather*}

\begin{claim} \label{claim: transverse}
$\p_1 H_{n+1}$ and $\p_2 H_{n+1}$ are $v_{n+1}$-convex, with $v_{n+1}$ pointing into $H_{n+1}$ along $\p_1 H_{n+1}$ and out of $H_{n+1}$ along $\p_2 H_{n+1}$.
\end{claim}

\begin{proof}
The fact that $v_{n+1}$ is negatively transverse to $\p_1 H_{n+1}$ follows from
	\begin{equation*}
		\left( z^2 + 3|x|^2 +1 \right) v_{n+1} \left( z^2 + |x|^2 \right) = -4 \left( 1 - |x|^2 \right) z^2 - 2 \left( z^2 - |x|^2 +3 \right) |x|^2 < 0
	\end{equation*}
near $\left\{ z^2 + |x|^2 =1 \right\}$. Similarly, the fact that $v_{n+1}$ is positively transverse to $\p_2 H_{n+1}$ follows from
	\begin{align*}
		\left( z^2 + 3|x|^2 +1 \right) v_{n+1} (|y|^2) &= 2\left(\left( z^2 + 3|x|^2 +1 \right) |y|^2 - 4zx \cdot y \right)\\
										  &\geq 2\left( (z^2 + 3|x|^2 +1) |y|^2 - 4 |z| |x| |y|\right) \\
										  &\geq 2\left( (z^2 + 3|x|^2 +1) |y|^2 - 2\left( z^2 + |x|^2 \right) |y|\right) >0
	\end{align*}
near $\left\{ |y|=2 \right\}$.
\end{proof}

It remains to describe the convex hypersurfaces $\p_1 H_{n+1}$ and $\p_2 H_{n+1}$ in more detail.

\subsubsection{Description of $\p_1 H_{n+1}$}

Since $\alpha_{n+1}(v_{n+1})=f$, the $v_{n+1}$-dividing set of $\p_1 H_{n+1}$ is given by:
$$\Gamma_{\p_1 H_{n+1}} = \left\{ 2z = x \cdot y \right\} \subset \p_1 H_{n+1}.$$

\begin{claim} \label{claim: Legendrian ruling}
$R_+ (\p_1 H_{n+1})$ and $R_-(\p_1 H_{n+1})$ are exact deformation equivalent (relative to the boundary) to $T^\ast D^n$ with the standard Liouville form $p \cdot dq$ and Liouville vector field $p\cdot \bdry_p$.  Moreover, $\overline{{R}_{\pm} (\p_1 H_{n+1})}$ is the ideal compactification of $R_{\pm} (\p_1 H_{n+1})$, obtained by compactifying each fiber to a closed disk by adding the sphere at infinity.
\end{claim}

\begin{proof}
By construction $\p_1 H_{n+1}$ is foliated by Legendrian spheres
$$S^{n}_{y=b}=\{z^2+|x|^2=1\}\times\{y=b\}$$
such that each $S^{n}_{y=b}$ intersects $\Gamma_{\p_1 H_{n+1}}$ in an equatorial sphere $S^{n-1}_{y=b}$ which is Legendrian in $\Gamma_{\p_1 H_{n+1}}$. The claim is a consequence of the following general lemma.
\end{proof}

\begin{lemma}\label{lemma: normalization of foliation}
Let $\alpha$ be a Liouville form on $\R^n\times D^n$ with coordinates $(x,y)$ such that the pullback of $\alpha$ to each $y=b$, $b\in D^n$, is zero and the Liouville vector field is positively transverse to each $S^{n-1}_{|x|=r}\times D^n$ for $r\geq 1$. Then after applying a fiberwise diffeomorphism of $\R^n\times D^n$ (i.e., takes each $\R^n\times\{y=b\}$ to itself) $\alpha$ is exact deformation equivalent to the Liouville form $x \cdot dy$.
\end{lemma}

\begin{proof}
This is the relative version of the Arnold-Liouville theorem.  Observe that $\alpha=\sum_i f_i(x,y) dy_i$ since the pullback of $\alpha$ to each $y=b$ is zero.  Since
\begin{equation}\label{d alpha}
d\alpha=\sum_i d_x f_i \wedge dy_i + \sum_i d_y f_i\wedge dy_i
\end{equation}
is symplectic, it follows that $d_xf_1\wedge \dots \wedge d_xf_n$ is nowhere vanishing, where $d_x$ and $d_y$ refer to the exterior derivatives in the $x$- and $y$-directions.  Hence $F: \R^n \to \R^n$, $(x,y)\mapsto (f_1(x,y),\dots,f_n(x,y),y)$, is a local diffeomorphism.

Next we normalize the Liouville vector field.  Writing it as $X+Y$, where $X$ and $Y$ are components in the $\bdry_x$ and $\bdry_y$-directions, we claim that $Y=0$.  Indeed, $$i_Xd\alpha=\sum_i X(f_i) dy_i, \quad i_Yd\alpha= -\sum_i Y(x_i) d_x f_i+\dots,$$
where the dots refer to terms in $dy_i$.  Then $i_{X+Y}d\alpha=\alpha$ implies $Y=0$, keeping in mind $d_x f_1\wedge\dots\wedge d_x f_n\not=0$. Hence, after applying a fiberwise diffeomorphism of $\R^n\times D^n$, we may assume that $X=\sum_i x_i\bdry_{x_i}$. Integrating along $X$ we obtain a radial coordinate $s$ and write $\alpha= e^s (g_1dy_1+\dots +g_n dy_n)$, where $(g_1,\dots,g_n)$ are functions of coordinates on $S^{n-1}$.

Finally, the contact condition for $\alpha$ implies that the map $G=(g_1,\dots,g_n): S^{n-1}\to \R^n$ avoids the origin and descends to a local diffeomorphism $\overline G: S^{n-1} \to S^{n-1}=(\R^n-\{0\})/ \R^+$.  If $n>2$, then this implies that $\deg\overline G=\pm 1$ (say $1$), which in turn implies that $F$ is a diffeomorphism.  Outside of a compact region, we can therefore arrange $\alpha$ to be of the form $x\cdot dy$.  Inside the compact region, we can interpolate between $\alpha$ and $x\cdot dy$, which gives the exact deformation of Liouville forms.
\end{proof}

\begin{remark}
The foliation $\{S^n_{y=b}\}_{b\in D^n}$ is the higher-dimensional analog of the {\em Legendrian ruling} in dimension three~\cite{Hon00}.
\end{remark}

\subsubsection{Legendrian foliation of a neighborhood of $0$-section of $J^1(S^{n-1})$} \label{subsubsection: 1 jet space}

We now explain how to foliate a (specific) neighborhood $N$ of the $0$-section of the $1$-jet space $(J^1(S^{n-1}),\ker (du+p\cdot dq))$ by an $S^{n-1}$-family of $n$-dimensional $\Theta$-disks $\Theta_q$.

Observe that $X= u \bdry_u +p\cdot \bdry_p$ is a contact vector field.  Writing $r=|p|$, $X$ can be written as $u\bdry_u + r\bdry_r$ and the $1$-form as $du + r\beta$, where $\beta$ is a contact $1$-form on $\Gamma_{u_0}:=\{r=R, u=u_0\}$ that does not depend on $u_0$.  We now use the contact form $\alpha= {1\over r} du +\beta$ for $J^1(S^{n-1})$.

Fix $R>0$. The idea is to start with the Legendrian foliation
$$F_{u_0,q_0}:=\{u=u_0, q=q_0, r\leq R-\epsilon\}, \quad |u_0|\leq \epsilon, \quad q_0\in S^{n-1},$$
perturb it, and to ``connect" the boundary components of $F_{u_0,q_0}$ for each $q_0$: First define a $1$-parameter of pairwise disjoint embedded arcs $\gamma_{u_0}\subset \R^2_{r,u}$, $u_0\in[-\epsilon,\epsilon]$, that connect from $(R-\epsilon,u_0)$ to $(R,0)$. Consider the Legendrian $$\Lambda_{q_0}=\{u=0,q=q_0, r=R\}\subset \Gamma_0.$$
We then apply Lemma~\ref{lemma: contact parallel transport} to construct a family of Legendrian annuli $L_{u_0,q_0}$ that lie over $\gamma_{u_0}$ and end at $\Lambda_{q_0}\subset \Gamma_0$.

It remains to connect up $F_{u_0,q_0}$ and $L_{u_0,q_0}$.  We leave it to the reader to verify using the technique of Lemma~\ref{lemma: Legendrian realization} that there exists a Legendrian foliation $F'_{u_0,q_0}$ such that $F'_{u_0,q_0}$ is close to $F_{u_0,q_0}$ if $\epsilon>0$ is small and $\bdry F'_{u_0,q_0}= \bdry L_{u_0,q_0}$.

We then set
\begin{gather*}
N= \sqcup_{q_0\in S^{n-1}} \Theta_{q_0}, \quad \Theta_{q_0}=\cup_{|u_0|\leq \epsilon} (F'_{u_0,q_0}\cup L_{u_0,q_0}).
\end{gather*}

\subsubsection{Description of $\p_2 H_{n+1}$}

Next let us turn to $\p_2 H_{n+1}$, which is an $S^{n-1}$-family of $(n+1)$-dimensional $\Theta$-disks $\Theta_y$, $y\in S^{n-1}$. The $v_{n+1}$-dividing set is
$$\Gamma_{\p_2 H_{n+1}} = \left\{ 2z = x \cdot y \right\} \subset \p_2 H_{n+1},$$
and the restriction $\Theta'_y:=\Gamma_{\p_2 H_{n+1}}\cap \Theta_y$ is an $n$-dimensional $\Theta$-disk.  By \cite[Lemma 3.8]{Hua15}, $\Gamma_{\p_2 H_{n+1}}$ is contactomorphic to $N$.  Next, the domain $R_+(\bdry_2 H_{n+1})^c$ is a bundle over $S^{n-1}$ whose fiber is half of the $(n+1)$-dimensional $\Theta$-disk $\Theta_y$ which we call $\Theta_y^+$, and $L_y:=\bdry \Theta_y^+ \cap R_+(\bdry_2 H_{n+1})^c$ is the standard Lagrangian disk that bounds the Legendrian unknot $\Lambda_y:=\bdry\Theta'_y$.  By standardizing the Legendrian unknots $\Lambda_y$ and the bounding Lagrangian disks $L_y$ together with the Lagrangian foliations, it is not hard to see that there is a Liouville vector field that:
\be
\item points out of $\Gamma_{\p_2 H_{n+1}}$, viewed as a subset of $\bdry (R_+(\bdry_2 H_{n+1})^c)$;
\item into $\bdry (R_+(\bdry_2 H_{n+1})^c)- \Gamma_{\p_2 H_{n+1}}$; and
\item is foliated by intervals.
\ee
The case of $R_-(\bdry_2 H_{n+1})^c$ is analogous.

\cb

\subsubsection{Description of $(n+1)$-handle}

The analog of \autoref{prop:n_handle} for contact $(n+1)$-handles is the following:

\begin{prop} \label{prop:(n+1)_handle}
Let $(M,\xi)$ be a $(2n+1)$-dimensional contact manifold with convex boundary $\Sigma = \p M$. Fix an outward-pointing contact vector field $v$ along $\Sigma$, let $\Gamma$ be the $v$-dividing set, viewed as a codimension 2
contact submanifold, and let $D_\pm \subset R_\pm$ be open Lagrangian disks (where $R_\pm$ are regarded as Liouville manifolds as in \autoref{lem:ideal_compactification}) such that $D_\pm$ have cylindrical ends that limit to the same Legendrian sphere $\Lambda \subset \Gamma = \p \overline{R}_\pm$. Let $\Sigma \setminus \Gamma = R_+ \cup R_-$ be the decomposition into positive and negative regions as in \autoref{subsec:convex_surface}. Then a contact $(n+1)$-handle can be attached to $(M, \xi)$ along $D_+ \cup D_-$ after applying Legendrian realization (\autoref{lemma: Legendrian realization}) to yield a new contact manifold $(M', \xi')$ with convex boundary $\Sigma' = \p M'$ such that if we write $\Sigma' \setminus \Gamma' = R'_+ \cup R'_-$, then:
		\begin{enumerate}
			\item as a Liouville manifold, $R'_+$ is the completion of a Liouville domain obtained by removing a standard neighborhood of the Lagrangian disk $D^c_+ := D_+ \cap R^c_+$ from $R^c_+$;
		
			\item as a Liouville manifold, $R'_-$ is the completion of a Liouville domain obtained by removing a standard neighborhood of the Lagrangian disk $D^c_- := D_- \cap R^c_-$ from $R^c_-$;
		
			\item as a contact manifold, $\Gamma'$ is obtained from $\Gamma$ by performing a contact $(+1)$-surgery along $\Lambda \subset \Gamma$; and
		
			\item as a smooth manifold, $\Sigma'$ is diffeomorphic to the union of $\overline{R'_+}$ and $\overline{R'_-}$, glued along their boundaries by the identity map.
		\end{enumerate}
\end{prop}

\begin{proof}
By applying the Flexibility Lemma (Lemma~\ref{lemma: flexibility}), we may assume that $D_+\cup D_-$ has a Legendrian ruled neighborhood $N(D_+\cup D_-)$ in $\Sigma$.  By Claim~\ref{claim: Legendrian ruling} (and also possibly applying the Flexibility Lemma) the contact germ on $\p_1 H_{n+1}$ can be identified with the restriction of the contact germ on $\Sigma$ to $N(D_+\cup D_-)$.

(1), (2).  The analysis on the contact germ on $\p_2 H_{n+1}$ from above shows that $R'_+ = R_+ (\Sigma')$ is obtained from $R_+$ by first removing a standard neighborhood of the Lagrangian disk $D_+ \subset R_+$ and then (partially) completing the Liouville domain. The $R'_-$ case is analogous.  (3) and (4) are immediate.
\end{proof}

\begin{remark}
We say that a convex hypersurface $\Sigma$ is {\em Weinstein} if both $R_{\pm} (\Sigma)$ are Weinstein manifolds.  The Weinstein property is preserved by arbitrary contact $n$-handle attachments by \autoref{prop:n_handle}. On the other hand, the Weinstein property is not necessarily preserved by contact $(n+1)$-handle attachments since we are removing neighborhoods of Lagrangian disks by \autoref{prop:(n+1)_handle}(1) and (2). This problem is closely related to the regularity problem of Lagrangian submanifolds in Weinstein manifolds in the sense of Eliashberg-Ganatra-Lazarev \cite[Problem 2.5]{EGL15}.
\end{remark}

\section{The Legendrian sum and Legendrian handleslides} \label{sec:Legendrian_sum}

In this section we review the Legendrian sum operation and the description of Legendrian handleslides in terms of Legendrian sums. The material in this section is well-known to the experts. In particular the theory of Legendrian handleslides and their front presentations is worked out in some detail by Casals-Murphy \cite{CM16,CM16_preprint}. For our later purposes, we will also consider the Legendrian boundary sum operation in \autoref{subsec:boundary_sum}.

\subsection{Resolution of Legendrian intersections and the Legendrian sum} \label{subsec:resolution}

Consider two Legendrian submanifolds $\Lambda_1$ and $\Lambda_2$ in a contact manifold $(M,\xi)$ which intersect at one point $p$ and such that the intersection is {\em $\xi$-transverse}, i.e., $\xi_p=T_p \Lambda_1\oplus T_p \Lambda_2$. This nongeneric condition is usually achieved by starting with the shortest Reeb chord for some contact form $\alpha$ for $\xi$ from $\Lambda_2$ to $\Lambda_1$ and flowing $\Lambda_2$ using the time-$t$ Reeb flow $\phi_t$ for the Reeb vector field $R_\alpha$ until $\phi_t(\Lambda_2)$ intersects $\Lambda_1$.

The goal of this subsection is to construct a new Legendrian submanifold $\Lambda_1 \uplus \Lambda_2$ by resolving the $\xi$-transverse intersection at $p$. We will call the resulting Legendrian $\Lambda_1 \uplus \Lambda_2$ the {\em Legendrian sum} of $\Lambda_1$ and $\Lambda_2$ at $p$, and call the triple $(\Lambda_1, \Lambda_2, \Lambda_1 \uplus \Lambda_2)$ a {\em Legendrian triangle}.

\begin{remark}
Our construction here is the contact analog of Polterovich's \cite{Pol91} Lagrangian connected sum for transverse Lagrangian intersections in symplectic manifolds.
\end{remark}

\begin{remark}
We write $\Lambda_1 \uplus \Lambda_2$ instead of $\Lambda_1 \# \Lambda_2$ since the latter is commonly used to denote the Legendrian connected sum, {\em which is a different operation,} although smoothly $\Lambda_1 \uplus \Lambda_2$ is diffeomorphic to the connected sum of $\Lambda_1$ and $\Lambda_2$.
\end{remark}

\subsubsection{The construction} \label{subsec:legendrian_sum}

Since our construction is local, consider a Darboux ball $B(p)$ around $p$ in $(M,\xi)$ which is contactomorphic to the unit open ball in $(\R^{2n+1},\xi_0)$, where $\xi_0=\ker\alpha_0$ and $\alpha_0 = dz - 2y \cdot dx - x \cdot dy$. Here $x = (x_1, \dots, x_n), y=(y_1, \dots, y_n)$, and $z$ are the usual coordinates on $\R^{2n+1}$. Without loss of generality, we may further assume $\Lambda_1 \cap B(p)$ is identified with $D_1 = \left\{ z=y=0, |x|<1 \right\}$ and $\Lambda_2 \cap B(p)$ is identified with $D_2 = \left\{ z=x=0, |y|<1 \right\}$.

Given $\epsilon > 0$, consider the smooth embedding
\begin{gather*}
\phi_\epsilon: S^{n-1} \times \R \to \R^{2n+1}_{x,y,z},\\
(a=(a_1,\dots,a_n),t)\mapsto (\epsilon e^t a,  \epsilon e^{-2t} a,0),
\end{gather*}
where $ S^{n-1}$ is regarded as the unit sphere in $\R^n$.

\begin{lemma}
The image $T:=\phi(S^{n-1}\times\R)$ of $\phi$ is a Legendrian submanifold in $(\R^{2n+1},\xi_0)$.
\end{lemma}

\begin{proof}
This follows from the following calculation:
	\begin{align*}
		\alpha_0|_T &= -\epsilon^2 \left(2e^{-2t} a \cdot d(e^t a) + e^t a \cdot d(e^{-2t} a) \right) \\
				  &= -\epsilon^2 \left(2e^{-2t} d(e^t) +e^t d(e^{-2t}) \right) = 0,
	\end{align*}
where we used the identity $a \cdot da=0$ since $|a|^2= 1$.
\end{proof}

To actually construct $\Lambda_1 \uplus \Lambda_2$, we need some cutoff operation which we explain now: Let
\begin{gather*}
A_1(r) = \left\{ r < |x| <1, z=y=0 \right\} \subset D_1,\\
A_2(r) = \left\{ r < |y| <1, z=x=0 \right\} \subset D_2.
\end{gather*}
Choosing $0 < \epsilon \ll r \ll 1$, we can assume, by the Legendrian neighborhood theorem, that $T \cap \left\{ |x|>r \right\}$ is given as the 1-jet $J^1 (f)$ of a $C^1$-small smooth function $f: A_1(r) \to \R$. Now choose a cutoff function $\psi: (r,1) \to \R$ such that $\psi(s)=1$ for $s \in (r, 2r)$ and $\psi(s)=0$ for $s \in (3r, 1)$. We approximate $J^1(f)$ by $J^1(\widetilde f)$, where $\widetilde f: A_1(r) \to \R$ is given by $\widetilde f(x)=f(x) \psi(|x|)$. Similarly we approximate $T \cap \left\{ |y|>r \right\}$ by the 1-jet $J^1(\widetilde g)$ of a $C^1$-small function $\widetilde g: A_2(r) \to \R$.
By construction $J^1 (\widetilde f)$ and $J^1 (\widetilde g)$ agree with $\Lambda_1$ and $\Lambda_2$ outside of a ball of radius $3r$, respectively.

We finally define $\Lambda_1 \uplus \Lambda_2$ by
\begin{equation} \label{eqn:legsurgery}
	\Lambda_1 \uplus \Lambda_2=
		\begin{cases}
			\Lambda_1 \cup \Lambda_2 & \text{ on } M \setminus B(p), \\
			\left( T \cap \left\{ |x| < 2r, |y| < 2r, z=0 \right\} \right) \cup J^1(\widetilde f) \cup J^1(\widetilde g) & \text{ on } B(p).
		\end{cases}	
\end{equation}
The Legendrian isotopy class of $\Lambda_1 \uplus \Lambda_2$ is independent of the choice of the Darboux ball, sufficiently small constants $\epsilon, r$, and the cutoff functions.

\subsubsection{Front and Lagrangian projections} \label{subsec:front_vs_lagrangian}

We now describe the Legendrian sum in the front and Lagrangian projections. Given the local nature of Legendrian sums, it suffices to consider the standard contact space $(\R^{2n+1},\xi_{\std}= \ker(dz- y \cdot dx))$.

\s\n
{\em Front projection.}
Suppose $\Lambda_1$ and $\Lambda_2$ intersect transversely at the origin such that the front projection is as shown in \autoref{fig:front_resolution}(a). Then the front projections of $\Lambda_1 \uplus \Lambda_2$ and $\Lambda_2 \uplus \Lambda_1$ are shown in \autoref{fig:front_resolution}(b) and \autoref{fig:front_resolution}(c). It is clear from this point of view that $\Lambda_1 \uplus \Lambda_2$ is in general not Legendrian isotopic to $\Lambda_2 \uplus \Lambda_1$.

\begin{figure}[ht]
\centerline{	\begin{overpic}[scale=.3]{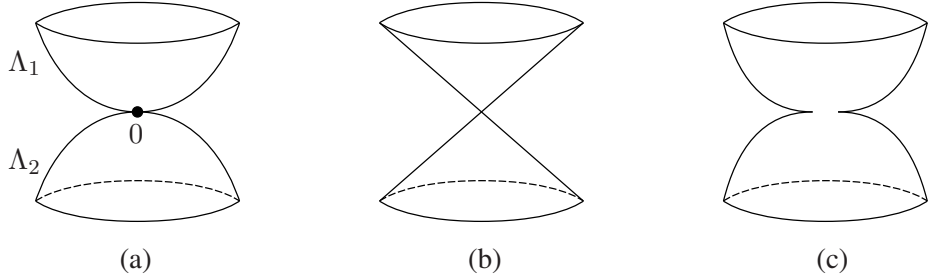}
		\put(10.5,9){$0$}
		\put(-3,17){$\Lambda_1$}
		\put(-3,6){$\Lambda_2$}
		\put(9.5,-5){(a)}
		\put(48.5,-5){(b)}
		\put(87.5,-5){(c)}
	\end{overpic}}
	\vspace{5mm}
	\caption{(a) The Legendrians $\Lambda_1$ and $\Lambda_2$ intersect $\xi$-transversely at 0. (b) The front projection of $\Lambda_1 \uplus \Lambda_2$. (c) The front projection of $\Lambda_2 \uplus \Lambda_1$.}
	\label{fig:front_resolution}
\end{figure}

\s\n
{\em Lagrangian projection.} Now suppose that $\Lambda_1$ and $\Lambda_2$ are diffeomorphic to $S^n$. Consider the Milnor fiber $A_2$ obtained by plumbing two copies of $T^\ast S^n$, and let $\lambda_2$ be its standard Liouville form. Let $L_1, L_2 \subset A_2$ be the $0$-sections of the two copies of $T^\ast S^n$. Then $\Sk(A_2) := L_1 \cup L_2$ is a {\em Lagrangian skeleton} of $A_2$.

Consider the contact manifold $(\R_t \times A_2, \xi_2)$ with the contact form $\alpha_2 = dt + \lambda_2$. Then the natural embedding $\Sk(A_2) \subset \left\{ 0 \right\} \times A_2 \subset \R \times A_2$ defines an immersed Legendrian submanifold with one $\xi_2$-transverse double point. Let $N_\epsilon \left( \Sk(A_2) \right) \subset \R \times A_2$ be an $\epsilon$-neighborhood of $\Sk(A_2)$. Then by the standard Legendrian neighborhood theorem, there exists a contact embedding
	\begin{equation*}
		\iota: \left( N_\epsilon \left( \Sk(A_2) \right), \xi_2|_{N_\epsilon \left( \Sk(A_2) \right)} \right) \hookrightarrow (M,\xi),
	\end{equation*}
such that $\iota(L_1) = \Lambda_1$ and $\iota(L_2) = \Lambda_2$.

Let $\pi: \R \times A_2 \to A_2$ be the projection map. Then the Lagrangian projection $\pi \circ \iota^{-1} (\Lambda_1 \uplus \Lambda_2)$ is exact\footnote{Here ``exact'' is automatic since we are dealing with spheres.} Lagrangian isotopic to $\tau_{L_1}^{-1}(L_2)=\tau_{L_2} (L_1)$, where $\tau_{L_1}$ is the positive Dehn twist along $L_1$ (cf.\ Seidel \cite{Sei03}). Similarly, $\pi \circ \iota^{-1} (\Lambda_2 \uplus \Lambda_1)$ is exact Lagrangian isotopic to $\tau_{L_2}^{-1} (L_1)=\tau_{L_1}(L_2)$. Conversely, using the contact embedding $\iota$, we may define $\Lambda_1 \uplus \Lambda_2$ to be the Legendrian lift of the exact Lagrangian sphere $\tau_{L_2} (L_1)$.

\begin{remark}
$\tau_{L_1}^{-1}(L_2)=\tau_{L_2} (L_1)$ is the same as the Lagrangian connected sum $L_1\# L_2$, where we are using the conventions of \cite{Au14}.
\end{remark}

From now on we assume the following convention:

\begin{convention} \label{convention for triangle}
When we write $(\Lambda_1, \Lambda_2, \Lambda_1 \uplus \Lambda_2)$ for a Legendrian triangle, we always assume that $\Lambda_1\uplus \Lambda_2\subset \iota( N_\epsilon(\Sk(A_2)))$ is a representative of its isotopy class such that:
\be
\item $\Lambda_2$ and $\Lambda_1\uplus \Lambda_2$ and $\xi$-transversely intersect at a point;
\item there are no $\iota_*\alpha_2$-Reeb chords between $\Lambda_2$ and $\Lambda_1\uplus \Lambda_2$ inside $\iota( N_\epsilon(\Sk(A_2)))$ besides the $\xi$-transverse intersection; and
\item there is a single $\iota_*\alpha_2$-Reeb chord from $\Lambda_1$ to $\Lambda_1\uplus \Lambda_2$ inside $\iota( N_\epsilon(\Sk(A_2)))$.
\ee
\end{convention}

\subsection{Basic properties of the Legendrian sum} \label{subsec:Legendrian_sum_property}

In this section we study some basic properties of the Legendrian sum operation which will be used in our analysis of bypass attachments.

Let $U \in (M, \xi)$ be the {\em standard Legendrian unknot}, i.e., i.e., it is Legendrian isotopic to the Legendrian unknot given by \autoref{fig:unknot} in a Darboux chart.
\begin{figure}[ht]
	\begin{overpic}[scale=.4]{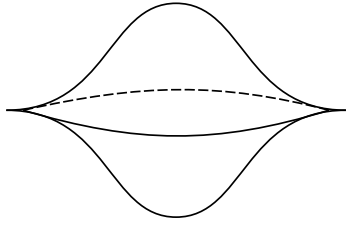}
	\end{overpic}
	\caption{The front projection of the standard Legendrian unknot.}
	\label{fig:unknot}
\end{figure}

\begin{claim}
A Legendrian sphere is standard if and only if it bounds a $\Theta$-disk.
\end{claim}


Let us first examine the properties of $\Lambda \uplus U$, where $\Lambda$ is a Legendrian submanifold that intersects $U$ $\xi$-transversely at one point. To do so, we need to specify the relative positions of $\Lambda$ and $U$. Here we are interested in two cases:
	\begin{enumerate}
		\item We say $U$ is {\em above} $\Lambda$ if $U$ bounds a $\Theta$-disk $B$ such that $\Lambda \cap B^\epsilon = \varnothing$ for any small $\epsilon>0$. Here $B^\epsilon = \phi_\epsilon (B)$ and $\phi_t: (M, \xi)\stackrel\sim \to (M, \xi)$ denotes the time-$t$ flow of the Reeb vector field associated to some contact form.

		\item We say $U$ is {\em below} $\Lambda$ if $U$ bounds a $\Theta$-disk $B$ such that $\Lambda \cap B^{-\epsilon} = \varnothing$ for any small $\epsilon > 0$.
	\end{enumerate}

We then have the following:

\begin{lemma} \label{lem:plus_unknot}
	If $U$ is above $\Lambda$, then $U \uplus \Lambda$ is Legendrian isotopic to $\Lambda$ and $\Lambda \uplus U$ is a stabilization of $\Lambda$. If $U$ is below $\Lambda$, then $U \uplus \Lambda$ is a stabilization of $\Lambda$ and $L \uplus U$ is Legendrian isotopic to $\Lambda$.
\end{lemma}

\begin{proof}
	The assertions of the lemma are most easily verified using the front projection picture of the Legendrian sum (cf.\ \autoref{fig:front_resolution}). If $U$ is above $\Lambda$, then, by Figure~\ref{fig:front_resolution}(b), $U \uplus \Lambda$ is obtained from $\Lambda$ by performing a generalized Reidemeister I move and is therefore Legendrian isotopic to $\Lambda$; by Figure~\ref{fig:front_resolution}(c), $\Lambda\uplus U$ is a stabilization of $\Lambda$. The case of $U$ below $\Lambda$ is analogous.
\end{proof}

The following lemma asserts that the Legendrian sum operation is associative and justifies the notation $\Lambda_1\uplus \Lambda_2\uplus \Lambda_3$.

\begin{lemma}[Associativity] \label{lem:associativity}
If $\Lambda_i, i=1,2,3$, are Legendrian submanifolds such that $\Lambda_1$ $\xi$-transversely intersects $\Lambda_2$ at one point, $\Lambda_2$ $\xi$-transversely intersects $\Lambda_3$ at one point, and $\Lambda_1 \cap \Lambda_3 = \emptyset$, then $(\Lambda_1 \uplus \Lambda_2) \uplus \Lambda_3$ is Legendrian isotopic to $\Lambda_1 \uplus (\Lambda_2 \uplus \Lambda_3)$.
\end{lemma}

Strictly speaking, $(\Lambda_1\uplus\Lambda_2)\uplus \Lambda_3$ means the following: First take $\Lambda_1\uplus \Lambda_2$.  By Convention~\ref{convention for triangle}, $\Lambda_1\uplus \Lambda_2$ does not intersect $\Lambda_3$ but there is a short Reeb chord from $\Lambda_3$ to $\Lambda_1\uplus \Lambda_2$.  We then apply a small isotopy to $\Lambda_3$ so it $\xi$-transversely intersects $\Lambda_1\uplus \Lambda_2$ and then apply the Legendrian sum.

\begin{proof}
We use the description of the Legendrian sum from \autoref{subsec:legendrian_sum}. The associativity follows from the fact that there exists disjoint Darboux balls around $\Lambda_1 \cap \Lambda_2$ and $\Lambda_2 \cap \Lambda_3$.
\end{proof}

\begin{lemma} \label{lem:conjugation}
If $\Lambda_1$ and $\Lambda_2$ are Legendrian spheres which intersect $\xi$-transversely at one point, then $(\Lambda_2 \uplus \Lambda_1) \uplus \Lambda_2$ is Legendrian isotopic to $\Lambda_1$.
\end{lemma}

\begin{proof}
By the discussion of the Lagrangian projection interpretation of the Legendrian sum from \autoref{subsec:front_vs_lagrangian}, $(\Lambda_2 \uplus \Lambda_1) \uplus \Lambda_2$ is the Legendrian lift of $\tau_{L_2} (\tau^{-1}_{L_2} (L_1)) = L_1$, where $L_i$ is the Lagrangian projection of $\Lambda_i$ for $i=1,2$.
\end{proof}

\begin{remark}
On the other hand, $(\Lambda_1 \uplus \Lambda_2) \uplus \Lambda_2$ is the Legendrian lift of the Lagrangian sphere $\tau^2_{L_2} (L_1)$. Hence $(\Lambda_1 \uplus \Lambda_2) \uplus \Lambda_2$ is in general not Legendrian isotopic to $\Lambda_1$.
\end{remark}

\subsection{Legendrian boundary sum} \label{subsec:boundary_sum}

In this subsection we generalize the construction from \autoref{subsec:resolution} to the case of Legendrian submanifolds with boundary.

Let $(M,\xi)$ be a contact manifold and $\Gamma \subset M$ be a compact codimension 2 contact submanifold with trivial normal bundle, e.g., $\Gamma$ is the dividing set of a closed convex hypersurface.
Consider a tubular neighborhood $N(\Gamma) \subset (M,\xi)$ of $\Gamma$ contactomorphic to $(\Gamma\times D_\delta,\zeta= \ker (\lambda +r^2d\theta))$, where $\lambda$ is a contact form on $\Gamma$ defining $\xi|_{\Gamma}$ and $D_\delta \subset \R^2$ is the disk of radius $\delta>0$ with polar coordinates $(r,\theta)$. (We assume that $r\geq 0$.)

Let $\Lambda_1, \Lambda_2$ be compact Legendrian submanifolds in $M$ with nonempty boundary such that the following conditions hold:
\begin{itemize}
\item[(A)]({\em Disjoint interior}) The interiors of $\Lambda_1$ and $\Lambda_2$ are disjoint;
\item[(B)]({\em Legendrian boundary}) $\p \Lambda_1, \p \Lambda_2 \subset (\Gamma, \xi|_{\Gamma})$ are Legendrian submanifolds that intersect $\xi|_\Gamma$-transversely at one point $p$;
\item[(C)]({\em Cylindrical end}) $\Lambda_1 \cap N(\Gamma) = \p \Lambda_1 \times \{\theta=\theta_1\}$, $\Lambda_2 \cap N(\Gamma) = \p \Lambda_2 \times \{\theta=\theta_2\}$, and $0<\theta_1<\theta_2<\pi$.
\end{itemize}

The {\em Legendrian boundary sum} $\Lambda_1 \uplus_b \Lambda_2$ is a Legendrian submanifold of $M$ with boundary $\p (\Lambda_1 \uplus_b \Lambda_2) = (\p \Lambda_1) \uplus (\p \Lambda_2) \subset \Gamma$, which is smoothly a boundary connected sum of $\Lambda_1$ and $\Lambda_2$.
We present two equivalent definitions of $\Lambda_1 \uplus_b \Lambda_2$ --- one in the Lagrangian projection and one in the front projection --- and leave the verification of the equivalence to the reader.

\s\n
{\em Lagrangian projection.} We first generalize Polterovich's Lagrangian surgery operation to Lagrangian submanifolds with nonempty boundary, where the transverse intersection point lies on the boundary. We only give a local model of the construction, leaving the globalization/smoothing operation to the reader.

Consider the standard symplectic vector space $(\R^{2n}, \omega_0)$, where $\omega_0 = dx \wedge dy$ and $x=(x_1, \dots, x_n), y=(y_1, \dots, y_n)$ are the coordinates on $\R^{2n}$. Let $\R^{2n-2} = \{ x_1=y_1=0 \} \subset \R^{2n}$ be a codimension 2 symplectic subspace. Consider Lagrangian submanifolds $\widetilde{L}_1 = \{y=0\} \cap \R^{2n-2}$ and $\widetilde{L}_2 = \{x=0\} \cap \R^{2n-2}$ in $\R^{2n-2}$, which transversely intersect at $0$. Define Lagrangian submanifolds $L_1 = \R_{x_1 \geq 0} \times \widetilde{L}_1$ and $L_2 = \R_{y_1 \geq 0} \times \widetilde{L}_2$. Then $L_1$ and $L_2$ transversely intersect at $\{0\} = \p L_1 \cap \p L_2$ and the local model is:
\begin{equation*}
L_1 \#_b L_2 = \left\{ x_i = e^t a_i, y_i = e^{-t} a_i ~|~ \sum\nolimits_{i=1}^{n} a_i^2 = 1, a_1 \geq 0, t \in \R \right\}.
\end{equation*}

Now given two Legendrian submanifolds $\Lambda_i, i=1,2$, with nonempty boundary satisfying Conditions (A), (B), and (C), we can choose a Darboux chart near the $\xi$-transverse intersection $p \in \p \Lambda_1 \cap \p \Lambda_2$ such that the Lagrangian projection of $\Lambda_i$ coincides with the $L_i$ constructed above for $i=1,2$. Then we define $\Lambda_1 \uplus_b \Lambda_2$ to be the Legendrian lift of $L_1 \#_b L_2$.

\s\n
{\em Front projection.} Let $B(p) \subset M$ be a Darboux ball around the intersection point $p\in \p \Lambda_1 \cap \p \Lambda_2$. We may assume that $B(p)$ is contactomorphic to (a neighborhood of the origin of) $(\R^{2n+1}, \xi_{\std}=dz-y\cdot x)$ and that $\Gamma \cap B(p)$ is identified with (a neighborhood of the origin of) $\R^{2n-1} = \{ x_1 = y_1 =0 \}$. By Condition (C), we may assume that the front projections of $\Lambda_1 \cap B(p)$ and $\Lambda_2 \cap B(p)$ are as in \autoref{fig:Legendrian_bdry_sum_front}(a).
Then the front projection of $\Lambda_1 \uplus_b \Lambda_2$ is given by \autoref{fig:Legendrian_bdry_sum_front}(b), where a disk family of cusps is formed.

\begin{figure}[ht]
	\centerline{
	\begin{overpic}[scale=.3]{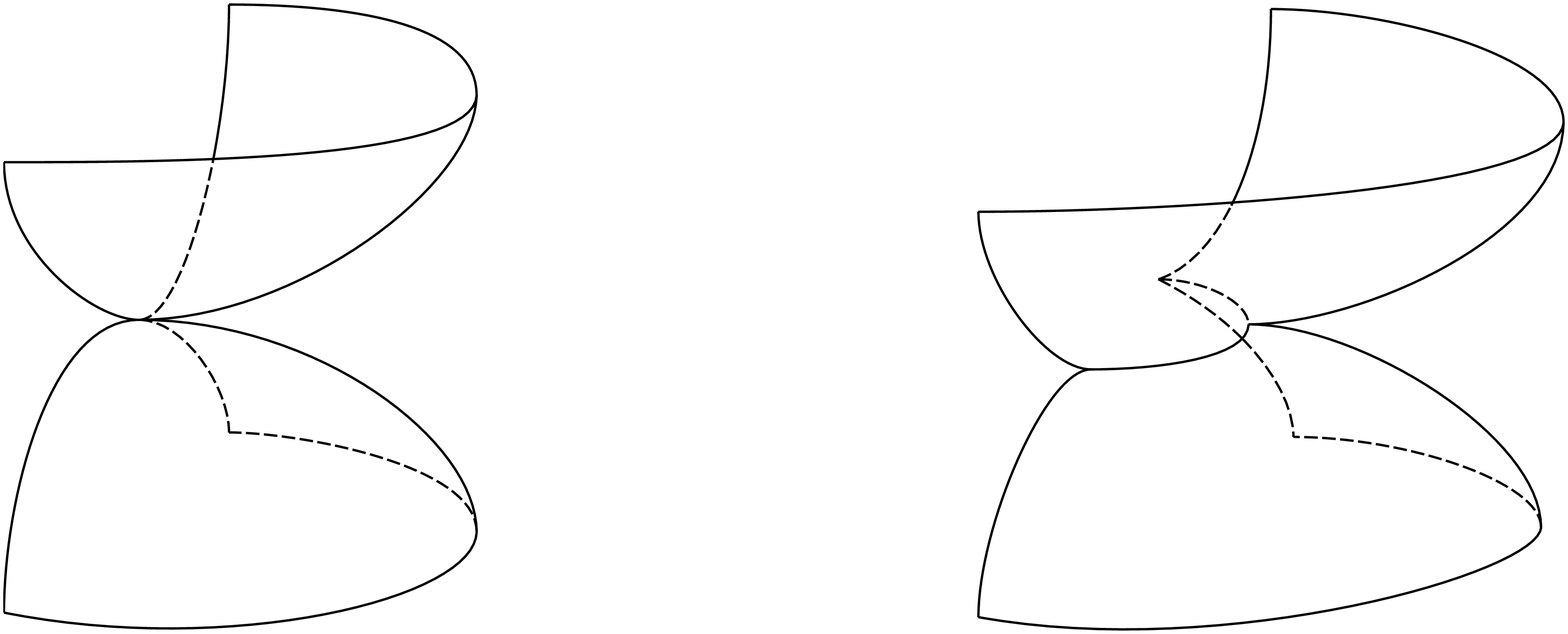}
		\put(12,-5){(a)}
		\put(78,-5){(b)}
		\put(32,6){$\Lambda_1$}
		\put(32,33){$\Lambda_2$}
	\end{overpic}
	}
	\vspace{5mm}
	\caption{(a) The Legendrians $\Lambda_1$ and $\Lambda_2$ intersect $\xi$-transversely at a point $p \in \p \Lambda_1 \cap \p \Lambda_2$. (b) The front projection of $\Lambda_1 \uplus_b \Lambda_2$.}
	\label{fig:Legendrian_bdry_sum_front}
\end{figure}

\s
We now prove a technical result which will be used in analyzing bypass attachments.
Let $\Lambda_k, k=1,2,3$, be Legendrian submanifolds in $(M, \xi)$ such that:
\begin{enumerate}
\item[(i)] $\p \Lambda_k \subset \Gamma$ and $\bdry \Lambda_k\simeq S^{n-1}$;
\item[(ii)] Conditions (A), (B), and (C) are satisfied for the pairs $(\Lambda_1,\Lambda_2)$, $(\Lambda_2,\Lambda_3)$, and $(\Lambda_3,\phi_\epsilon(\Lambda_1))$, where $\phi_\epsilon$ is a time-$\epsilon$ Reeb pushoff for $\epsilon>0$ small (the Reeb vector field is defined on $M$ and restricts to a Reeb vector field on $\Gamma$); and
\item[(iii)] the triple $( \p \Lambda_1, \p \Lambda_2, \p \Lambda_3)$ forms a Legendrian triangle in $\Gamma$.
\end{enumerate}
Since $\p \Lambda_3=\p \Lambda_1 \uplus \p \Lambda_2$ by (iii) and Convention~\ref{convention for triangle}, we can form a closed Legendrian $\Lambda_3 \cup (\Lambda_1 \uplus_b \Lambda_2)$ by gluing along the boundary.
By (iii), $(\bdry\Lambda_2,\bdry\Lambda_3,\bdry \phi_\epsilon(\Lambda_1))$ and $(\bdry\Lambda_3, \bdry \phi_\epsilon(\Lambda_1),\bdry\phi_\epsilon (\Lambda_2))$ also form Legendrian triangles, and we can form $\phi_\epsilon(\Lambda_1)\cup (\Lambda_2\uplus_b\Lambda_3)$ and $\phi_\epsilon(\Lambda_2)\cup (\Lambda_3\uplus_b \phi_\epsilon(\Lambda_1))$. {\em By abuse of notation we will often omit $\phi_\epsilon$ and refer to them as $\Lambda_1 \cup (\Lambda_2 \uplus_b \Lambda_3)$ and $\Lambda_2 \cup (\Lambda_3 \uplus_b \Lambda_1)$.}

\begin{lemma} \label{lem: Legendrian_wheel}
The Legendrians $\Lambda_3 \cup (\Lambda_1 \uplus_b \Lambda_2)$, $\Lambda_1 \cup (\Lambda_2 \uplus_b \Lambda_3)$, and $\Lambda_2 \cup (\Lambda_3 \uplus_b \Lambda_1)$ are all Legendrian isotopic.
\end{lemma}

\begin{proof}
We start by describing $\Lambda_3 \cup (\Lambda_1 \uplus_b \Lambda_2)$ in more detail.  Assume that
$$\Lambda_k\cap N(\Gamma)=\bdry\Lambda_k \times\{\theta=\tfrac{2\pi k}{3}\}, \mbox{ for } k=1,2,3.$$
By the front description, we may take $\Lambda_1\uplus_b\Lambda_2$ so that
$$(\Lambda_1\uplus_b\Lambda_2)\cap (\Gamma \times D_{\delta'})=\bdry (\Lambda_1\uplus_b \Lambda_2)\times \{\theta=\pi\}=\bdry \Lambda_3\times\{\theta=\pi\}$$
for some $0<\delta'\ll \delta$.  This glues smoothly with $\Lambda_3\cap N(\Gamma)=\bdry \Lambda_3\times\{\theta=0\}$.

Next we give an alternate description of $\Lambda_3 \cup (\Lambda_1 \uplus_b \Lambda_2)$. Legendrian isotop $\Lambda_3$ to $\Lambda_3'$ using Lemma~\ref{lemma: contact parallel transport} so that $\Lambda_3=\Lambda_3'$ outside of $\Gamma\times D_{2\delta'}$ and  $\Lambda_3'\cap (\Gamma\times D_{\delta'})= \phi_{\epsilon'}(\bdry \Lambda_3) \times\{\theta={3\pi\over 2}\}$, where $\epsilon'>0$ is small.  We remove small collar neighborhoods of the boundary from $\Lambda_3'$ and $\Lambda_1\uplus_b\Lambda_2$ to obtain $\mbox{sh}(\Lambda_3')$ and $\mbox{sh}(\Lambda_1\uplus_b\Lambda_2)$ (here $\mbox{sh}$ stands for ``shrinking") and add a Legendrian annulus $A$ with boundary $\bdry (\mbox{sh}(\Lambda_3'))\sqcup \bdry (\mbox{sh}(\Lambda_1\uplus_b\Lambda_2))$ of the form $\bdry \Lambda_3$ times a standard cusp.  Then $\Lambda_3 \cup (\Lambda_1 \uplus_b \Lambda_2)$ is Legendrian isotopic to $\mbox{sh}(\Lambda_3')\cup \mbox{sh}(\Lambda_1\uplus_b\Lambda_2)\cup A$ (verification left to the reader).

Now we claim that $\Lambda_3 \cup (\Lambda_1 \uplus_b \Lambda_2)$ is Legendrian isotopic to $((\Lambda_1 \uplus_b \Lambda_2)\uplus_b \Lambda_3  ) \cup K$, where $K$ is half of a standard Legendrian unknot as given in Figure~\ref{fig:cap}.  This can be seen by decomposing $A$ into two disk families of cusps $A_1$ and $A_2$.  Attaching $A_1$ to $\mbox{sh}(\Lambda_3')\cup \mbox{sh}(\Lambda_1\uplus_b\Lambda_2)$ is equivalent to $(\Lambda_1\uplus_b\Lambda_2)\uplus_b \Lambda_3$ and we then set $K=A_2$.
	\begin{figure}[ht]
		\centerline{	
			\begin{overpic}[scale=.3]{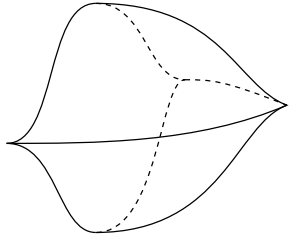}
			\end{overpic}
		}
		\caption{The front projection of a half Legendrian unknot $K$.}
		\label{fig:cap}
\end{figure}

Finally, $((\Lambda_1 \uplus_b \Lambda_2)\uplus_b \Lambda_3  ) \cup K= (\Lambda_1\uplus_b (\Lambda_2\uplus_b\Lambda_3))\cup K$, since the two handle attachments of the type given by Figure~\ref{fig:Legendrian_bdry_sum_front} commute.  Hence $\Lambda_3\cup(\Lambda_1\uplus_b \Lambda_2)$ is Legendrian isotopic to $\Lambda_1\cup (\Lambda_2\uplus_b \Lambda_3)$ and the lemma follows by cyclically rotating the Legendrians $\Lambda_k$, $k=1,2,3$.
\end{proof}

\subsection{Legendrian handleslides} \label{subsec:handleslide}

The goal of this subsection is give a quick review of the theory of Legendrian handleslides. Since a contact handle attachment changes $R_\pm$ by either attaching or removing Weinstein handles as we saw in \autoref{sec:contact_handles}, it is useful to understand how Legendrians slide over Weinstein handles. This is worked out by Ding-Geiges \cite{DG09} in dimension 4 and Casals-Murphy \cite{CM16_preprint} in general.

Let $(W^{2n},d\lambda)$ be a Liouville domain. Then $M=\bdry W$ is a contact manifold with contact form $\lambda|_M$. Let $\Lambda_i, i=1,2$, be Legendrian spheres in $M$ that intersect $\xi$-transversely at a single point $p$.

\begin{notation}
We use the notation $\Lambda^{t}_i$ to denote the pushoff of $\Lambda_i$ in the direction of the Reeb flow by time $t$.
\end{notation}

A {\em Legendrian handleslide} is a Legendrian isotopy whose trace intersects the cocore of a Weinstein handle once in the case of a handle attachment and whose trace intersects the core of the Weinstein handle once in the case of a handle removal.

Let $\epsilon > 0$ be a small positive number. We consider the following four scenarios for Legendrian handleslides:
\begin{enumerate}
	\item[(UA)] Slide $\Lambda_1^{-\epsilon}$ up over a handle attachment along $\Lambda_2$.
	
	\item[(DA)] Slide $\Lambda_1^{\epsilon}$ down over a handle attachment along $\Lambda_2$.
	
	\item[(UR)] Slide $\Lambda_1^{-\epsilon}$ up over a handle removal along $\Lambda_2$.
	
	\item[(DR)] Slide $\Lambda_1^{\epsilon}$ down over a handle removal along $\Lambda_2$.
\end{enumerate}
In the case of a handle removal, we are assuming that $\Lambda_2$ bounds a Lagrangian disk in $(W, d\lambda)$.

The result of the Legendrian handleslides is summarized in the following lemma:

\begin{lemma} \label{lemma: handle_slide}
With the above conventions, after the Legendrian handleslide the resulting Legendrian is:
		\begin{itemize}
			\item $(\Lambda_1 \uplus \Lambda_2)^\epsilon$ in Case {\em (UA)};

			\item $(\Lambda_2 \uplus \Lambda_1)^{-\epsilon}$ in Case {\em (DA)};

			\item $(\Lambda_2 \uplus \Lambda_1)^\epsilon$ in Case {\em (UR)};

			\item $(\Lambda_1 \uplus \Lambda_2)^{-\epsilon}$ in Case {\em (DR)}.
		\end{itemize}
\end{lemma}

\begin{proof}[Sketch of proof]
We treat Case (UA); the other cases can be handled in a similar manner.

We first recall the description of the Weinstein handle. It is a triple $(H, \omega, X)$, where
$$H = \{ |x| \leq 1 \} \times \{ |y| \leq 1 \} \subset \R^{2n}_{x,y},$$
$\omega = dx \wedge dy$ is the symplectic form, and $X = -x \cdot \p_x + 2y \cdot \p_y$ is the Liouville vector field. The boundary $\p H$ admits a decomposition  $\p_1 H \cup \p_2 H$ such that
$$\p_1 H = \{ |x| = 1 \} \times \{ |y| \leq 1 \}, \quad \p_2 H = \{ |x| \leq 1 \} \times \{ |y| = 1 \},$$
and $X$ points transversely into $\p_1 H$ and out of $\p_2 H$. 
Observe that:
\be
\item[(i)] $\p_1 H$ and $\p_2 H$ are contactomorphic to tubular neighborhoods of the $0$-section of the $1$-jet space $J^1 (S^{n-1})$ with the standard contact structure;
\item[(ii)] $\p (\p_1 H) = \p (\p_2 H)$ is a convex hypersurface;
\item[(iii)] if we write $\p (\p_1 H) = R_+ \cup_\Gamma R_-$ as usual, then both $R_\pm$ are symplectomorphic to $T^\ast S^{n-1}$ and $\p (\p_1 H)$ is obtained by gluing two copies of the ideal compactification of $T^\ast S^{n-1}$ along their boundaries by the identity map.
\ee
The handle $H$ is attached to $(W, d\lambda)$ along the Legendrian sphere $\Lambda_2 \subset M$, by identifying a tubular neighborhood of $\Lambda_2$ in $M$ with $\p_1 H$ via a contactomorphism which sends $\Lambda_2$ to the $0$-section in $J^1 (S^{n-1})$.

Fix $p\in S^{n-1}$. Let $D_\pm$ be the fiber $T_p^\ast S^{n-1}\subset R_\pm$ and let $\overline D_\pm$ be its compactification in $\overline{R}_\pm$. Here we are assuming that $m:=\overline D_+ \cup \overline D_- \subset \p (\p_1 H)$
\begin{itemize}
\item is Legendrian isotopic to the standard Legendrian unknot in $\p_1 H$ and
\item bounds an $n$-disk which is foliated by Legendrian disks as in \autoref{fig:std_foliation}.
\end{itemize}
 Then $m$ intersects the zero section $\ell\subset R_+$ $\xi$-transversely in one point (i.e., at a copy of $p$); see \autoref{fig:contact_framing}(a). One can verify that $m\uplus \ell$
\begin{itemize}
\item is Legendrian isotopic to the standard Legendrian unknot in $\p_2 H$ and
\item bounds an $n$-disk which is foliated by Legendrian disks.
\end{itemize}
Finally, the procedure of sliding $\Lambda_1^{-\epsilon}$ up across $H$ is done by first isotoping $\Lambda_1$ so that $\Lambda_1\cap \bdry_2 H=\overline D_-$, and then replacing $\Lambda_1$ by $(\Lambda_1-\overline D_-) \cup (\overline D_+\uplus \ell)$, which is Legendrian isotopic to $(\Lambda_1 \uplus \Lambda_2)^\epsilon$.

This finishes the proof of the lemma in Case (UA).
\end{proof}

\begin{figure}[ht]
\s
\centerline{\begin{overpic}[scale=.33]{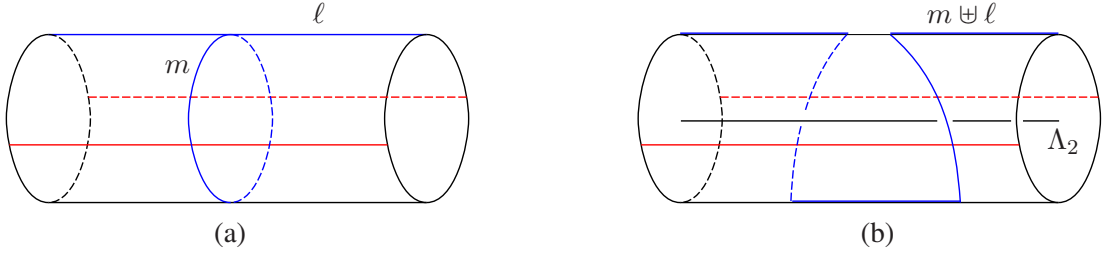}
				\put(14.5,12){$m$}
				\put(28,16.5){$\ell$}
				\put(84,16.5){$m \uplus \ell$}
				\put(95,5){$\Lambda_2$}
				\put(19,-3.5){(a)}
				\put(78,-3.5){(b)}
			\end{overpic}}
			\vspace{3mm}
			\caption{Legendrians $m$, $\ell$, and $m\uplus \ell$ drawn in blue on the convex boundary $\p (\p_1 H)$. The dividing set is drawn in red.}
			\label{fig:contact_framing}
\end{figure}

\begin{remark}
\autoref{lemma: handle_slide} and the results from \autoref{subsec:Legendrian_sum_property} allow us to do the higher-dimensional analog of Kirby calculus for Legendrians.
\end{remark}

\section{Bypass attachment} \label{sec:bypass_attachment}

The goal of this section is to construct the bypass attachment and study its properties using the techniques developed in Sections~\ref{sec:contact_handles} and \ref{sec:Legendrian_sum}.

\subsection{Construction of the bypass attachment} \label{subsec:construct_bypass_attachment}

Let $\Sigma$ be a convex hypersurface with contact germ $\xi$ and dividing set $\Gamma$.  Recall we have the decomposition $\Sigma \setminus \Gamma = R_+ \cup R_-$, where $R_\pm$ can be equipped with the structure of Liouville manifolds and $\Sigma$ can be obtained by gluing the ideal compactifications $\overline{R}_\pm$ along the boundary $\Gamma$ by the identity map. In the following we will continue to use the terminology from \autoref{sec:convex_surface}.

Let us first introduce the initial data for a bypass attachment:

\begin{definition} \label{defn:bypass_data}
	The quadruple $(\Lambda_+, \Lambda_-; D_+, D_-)$ is called a {\em bypass attachment data} on $\Sigma$ if $\Lambda_\pm \subset \Gamma$ are Legendrian spheres that $\xi|_{\Gamma}$-transversely intersect at one point and $D_\pm \subset (R_\pm,d\lambda_\pm)$ are Lagrangian disks with cylindrical ends asymptotic to $\Lambda_\pm$, respectively. Here $\lambda_{\pm}$ are the Liouville forms on $R_{\pm}$ specified by \autoref{lem:ideal_compactification}.
\end{definition}

Given $ (M, \xi=\ker\alpha)$, let  $\phi_t: M \stackrel\sim\to M$ be the time-$t$ flow of the Reeb vector field for $\alpha$.  If $\Lambda \subset (M, \xi)$ is a Legendrian submanifold, then we write $\Lambda^t := \phi_t (\Lambda)$.   Next let $(W,\lambda)$ be a Liouville domain such that $\bdry W=M$ and $\lambda|_{M}= \alpha$.  There exists a collar neighborhood $N(M)=[-\epsilon,0]\times M$ of $M=\{0\}\times M$ with Liouville form $d(e^s\alpha)$.  Using a Hamiltonian function $H=H(s)$ with support on $N(M)$, we can extend $\phi_t$ to a symplectic isotopy $\widetilde{\phi}_t: W\stackrel\sim\to W$ with support on $N(M)$. If $\Lambda$ bounds a Lagrangian disk $D \subset W$, then we set $D^t := \widetilde\phi_t (D)$.

\begin{notation}
Sometimes $\Lambda$ (resp.\ $D$) can be viewed as a subset of two different contact manifolds $M, M'$ (resp.\ $W,W'$).  In that case we write $\Lambda^{t,M}$ or $\Lambda^{t,M'}$ (resp.\ $D^{t,W}$ or $D^{t,W'}$) to distinguish the contact manifold (resp.\ Liouville domain).
\end{notation}

Given a bypass attachment data, the following theorem constructs/defines the bypass attachment:

\begin{theorem}[Bypass attachment] \label{thm:bypass_attachment}
Let $\Sigma$ be a convex hypersurface with contact germ $\xi$ and let $(\Lambda_+, \Lambda_-; D_+, D_-)$ be a bypass attachment data on $\Sigma$. Then there exists an extension of $\xi$ to $\Sigma\times[0,1]$ and a contact vector field $v$ on $\Sigma \times [0,1]$ such that:
\begin{enumerate}
\item $v$ is gradient-like for a Morse function $f: \Sigma \times [0,1] \to \R$ such that $\Sigma \times \{0,1\}$ are regular level sets of $f$ and $f$ has exactly two critical points --- one of index $n$ and one of index $n+1$ --- which are connected by a unique gradient trajectory.			
\end{enumerate}
For $i=0,1$ let $\Gamma^i$ be the $v$-dividing set of $\Sigma^i=\Sigma\times\{i\}$ and $\Sigma^i\setminus \Gamma^i=R_+^i\cup R_-^i$. Let $\epsilon>0$ be a small constant. Then:
\begin{enumerate}
\item[(2)] $R^1_\pm$ is the completion of the Liouville domain $(R^1_\pm)^c$ which is obtained from $(R^0_\pm)^c$ by attaching a Weinstein handle along an $\epsilon/4$-neighborhood of $\Lambda_- \uplus \Lambda_+$ and removing a standard (at most) $\epsilon/4$-neighborhood of the Lagrangian disk $D^{\mp\epsilon}_\pm \cap (R^0_\pm)^c$, made Legendrian after applying Legendrian realization. Here $D^{\mp\epsilon}_\pm= D^{\mp\epsilon, R^0_\pm}_\pm$.
\item[(3)] $\Gamma^1$ is obtained from $\Gamma^0$ by performing a contact $(-1)$-surgery along $\Lambda_- \uplus \Lambda_+$ and a contact $(+1)$-surgery along $\Lambda^{\mp\epsilon}_\pm=\Lambda^{\mp\epsilon, \Gamma^0}_\pm$.  Let $(\Lambda^{\mp\epsilon})'\subset \Gamma^1$ be the Legendrian corresponding to $\Lambda^{\mp\epsilon}\subset \Gamma^0$.
\item[(4)] $\Sigma^1$ is obtained by gluing $\overline{R^1_+}$ and $\overline{R^1_-}$ along the boundary using a contactomorphism $\psi: \p \overline{R^1_+} \stackrel\sim\to \p \overline{R^1_-}$ induced by sliding $(\Lambda^{-\epsilon}_+)'$ up over the contact $(-1)$-surgery along $\Lambda_- \uplus \Lambda_+$ to $(\Lambda^{\epsilon}_-)'$.
\end{enumerate}
\end{theorem}

The contact structure $\xi$ on $\Sigma \times [0,1]$ is called the {\em bypass attachment} along $(\Lambda_+, \Lambda_-; D_+, D_-)$.

\begin{proof}
The proof consists of three steps.

\s\n
\textsc{Step 1.} \emph{Attaching a pair of contact handles.}
		
\s
Consider the Legendrian sphere $\Lambda_- \uplus \Lambda_+ \subset \Gamma^0 \subset \Sigma^0$. By \autoref{prop:n_handle}, a contact $n$-handle $H_n$ can be attached to $\Sigma^0$ along $\Lambda_- \uplus \Lambda_+$ to yield a convex hypersurface $S=(\Sigma^0\setminus \bdry_1 H_n) \cup \bdry_2 H_n$. Let $S \setminus \Gamma_S = R_+ (S) \cup R_- (S)$ be the usual decomposition. Then $(R_\pm(S))^c$ is obtained from $(R_\pm(\Sigma))^c$ by a Weinstein handle attachment along (copies of) $\Lambda_- \uplus \Lambda_+$, respectively, and $\Gamma_S$ is obtained from $\Gamma_{\Sigma^0}$ by a contact $(-1)$-surgery along $\Lambda_- \uplus \Lambda_+$.

Next we describe the attaching locus of the contact $(n+1)$-handle on $S$. By \autoref{lemma: handle_slide} and \autoref{lem:conjugation}, $\Lambda_-^\epsilon$ can be obtained by handlesliding $\Lambda_+^{-\epsilon}$ up across the contact $(-1)$-surgery along $\Lambda_- \uplus \Lambda_+$. Note that we have chosen the surgery region to be disjoint from $\Lambda_{\pm}^{\mp \epsilon}$. At this point we apply \autoref{cor:isotop_Lagrangian_disk} to $D_+^{-\epsilon}$ and the Legendrian handleslide from $\Lambda_+^{-\epsilon}$ to $\Lambda_-^\epsilon$ to construct a Lagrangian disk $D'_+ \subset R_+(S)$ with cylindrical end such that $\p \overline{D'_+} = \Lambda^{\epsilon}_- = \p \overline{D^{\epsilon}_-}$. Then view $D'_+$ as Legendrian disk using Legendrian realization (Lemma~\ref{lemma: Legendrian realization}).  By \autoref{prop:(n+1)_handle}, one can attach a contact $(n+1)$-handle $H_{n+1}$ along the Legendrian sphere $\overline{D'_+} \cup_{\Lambda^{\epsilon}_-} \overline{D^{\epsilon}_-} \subset S$.

Note that, instead of sliding from $\Lambda^{-\epsilon}_+$ to $\Lambda^{\epsilon}_-$, one can also slide from $\Lambda^{\epsilon}_-$ to $\Lambda^{-\epsilon}_+$ which, by applying \autoref{cor:isotop_Lagrangian_disk} as above, yields a Lagrangian disk $D'_- \subset R_-(S)$ such that $\p \overline{D'_-} = \Lambda^{-\epsilon}_+$. Then a contact $(n+1)$-handle can be attached along $\overline{D^{-\epsilon}_+} \cup_{\Lambda^{-\epsilon}_+} \overline{D'_-} \subset S$. The reader can verify that these two contact $(n+1)$-handle attachments are contactomorphic.  For the rest of the paper we use the first description.

\s\n
\textsc{Step 2.} \emph{The contact handles form a smoothly canceling pair.}

\s		
It suffices to show that up to smooth (but not necessarily contact) isotopy, the belt sphere of $H_n$ geometrically intersects the attaching locus of $H_{n+1}$ in one point. Observe that, up to smooth isotopy, the attaching sphere of $H_{n+1}$ can be obtained by taking the union of $D_+^{-\epsilon}$, $D_-^{\epsilon}$ and the trace of the handleslide from $\Lambda_+^{-\epsilon}$ to $\Lambda_-^{\epsilon}$ in $\Gamma_{\p_2 H_n} \subset \Gamma_S$, where $\p_2 H_n$ is defined in \autoref{subsec:n_handle}. Clearly $D_+^{-\epsilon} \cup D_-^{\epsilon}$ is disjoint from the belt sphere of $H_n$. Then the assertion follows from the observation that the trace of the handleslide intersects the belt sphere of $H_n$ in precisely one point.
	
\s\n
\textsc{Step 3.} \emph{Verify Properties (2)--(4).}
	
\s	
We show (2) for $R_-^1$, leaving the rest to the reader. By \autoref{prop:n_handle}, the contact $n$-handle attachment changes $R_-^0$ by a Weinstein handle attachment along $\Lambda_- \uplus \Lambda_+$. Now using the first description of the attaching locus of the $(n+1)$-handle from Step 1 and \autoref{prop:(n+1)_handle}, we see that $R^1_-$ is precisely the completion of the Liouville domain obtained from $(R^0_-)^c$ by attaching a Weinstein handle along $\Lambda_- \uplus \Lambda_+$ and removing a standard neighborhood of the Lagrangian disk $D^{\epsilon}_- \cap (R^0_-)^c$, as claimed in (2).
\end{proof}

The characterization of the bypass attachment from \autoref{thm:bypass} can be slightly strengthened as follows: Fix a parametrization $\Sigma \cong \Sigma^0$. The product structure on $\Sigma \times [0,1]$ then gives a canonical identification $\Sigma\cong\Sigma^0\cong \Sigma^1$. We briefly explain how $\overline{R^1_+}$ is embedded in $\Sigma$ as a smooth submanifold; the $\overline{R^1_-}$ case is similar. To this end, let $N_{\epsilon/4} ({D}_+) \subset {R}_+^0$ be an $\epsilon/4$-neighborhood of $D_+$ and consider
$$T = R_+^0 \setminus N_{\epsilon/4} (D_+) \subset \Sigma.$$
The Legendrian sphere $\Lambda^{\epsilon}_+=\Lambda^{\epsilon,\Gamma^0}_+\subset \p T$ bounds a Legendrian disk $D^{\vee}_+$ on (some parallel copy of) $\Sigma \setminus T$ with $(-1)$-framing, i.e., any pushoff of $\Lambda^\epsilon_+$ in $\p T$ with the contact $(-1)$-framing bounds a disk parallel to $D^\vee_+$ in $\Sigma\setminus T$. In fact $D^\vee_+$ is essentially the cocore disk of $N_{\epsilon/4} (D_+)$ (i.e., dual to the core $D_+$). Now $\Lambda^{\epsilon}_- \subset \p T$ bounds a Legendrian disk $D^\epsilon_-$ on (some parallel copy of) $\Sigma\setminus T$ with contact $(+1)$-framing. Moreover we may assume that $D^\epsilon_-$ and $D^\vee_+$ satisfy (A), (B), (C) from \autoref{subsec:boundary_sum}, so one can construct the Legendrian disk $D^\epsilon_- \uplus_b D^\vee_+$ with $\p (D^\epsilon_- \uplus_b D^\vee_+) = (\Lambda_- \uplus \Lambda_+)^\epsilon \subset \p T$. It turns out that $D^\epsilon_- \uplus_b D^\vee_+$ has $(-1)$-framing in the sense explained above, so its neighborhood can be equipped with the structure of a Weinstein handle.

\begin{lemma}
$R^1_+ \subset \Sigma$, up to completion, is obtained by attaching the Weinstein handle with core $D^\epsilon_- \uplus_b D^\vee_+$ to $(R_+^0 \setminus N_{\epsilon/4} (D_+))^c $ along $(\Lambda_- \uplus \Lambda_+)^\epsilon$ and $\overline{R^1_+}$ is the ideal compactification of $R^1_+$.
\end{lemma}


\subsection{Some lemmas}

\subsubsection{Piecewise smooth descriptions of Lagrangians} \label{subsection: piecewise smooth descriptions}

The goal of this subsection is to give piecewise smooth descriptions of Lagrangians in the disk cotangent bundle $\D T^* D^n$ of the unit ball $D^n\subset \R^n$ and Lagrangian and Lagrangian boundary sums. Let $x=(x_1,\dots,x_n)$ be the coordinates on $\R^n$ and $y=(y_1,\dots,y_n)$ be the dual coordinates on $T^* \R^n$.

\s\n
A. Let $L= \D T^*_0 D^n$ be the fiber over $0$ in $\D T^* D^n$.  Let
$$L^0=int(D^n)\quad \mbox{and} \quad L^1=\bdry D^n.$$
Consider the positive and negative conormals
\begin{align*}
C_+(L^1)=& \{\omega\in \D T^*_{L^1} D^n~|~ \omega(\bdry_r) >0 \},\\
C_-(L^1)=& \{\omega\in \D T^*_{L^1} D^n~|~ \omega(\bdry_r) <0 \}\subset ~~ \D T^*_{L^1} D^n \cong L^1 \times [-1,1],
\end{align*}
where $r=|x|$.
Then define
$$L^\square_+:= L^0 \cup C_+(L^1) \quad \mbox{ and } \quad L^\square_-:=L^0\cup C_-(L^1),$$
where we are viewing $L^0$ as a portion of the zero section. Then $L$ is Lagrangian isotopic to smoothings of $L^\square_+$ and $L^\square_-$ through Lagrangian disks with Legendrian boundary on the unit cotangent bundle $\mathbb{S}T^*D^n$.  This is a straightforward limit of graphs of the form $df$ where $f:D^n\to \R$ satisfies $f(x)=f(r)$.

\s\n
B. Let $L$ be  the standard Lagrangian disk in $\D T^* D^n$ that bounds the standard Legendrian unknot $\Lambda$ in the unit contangent bundle $\mathbb{S}T^* D^n$.   Let
$$L^0=int(D^n),\quad L^1=\bdry D^n\setminus \{x_1=0\}, \quad L^2= \bdry D^n \cap \{x_1=0\}.$$
We take
\begin{gather*}
C(L^1)= C_+(L^1\cap \{x_1>0\}) \cup C_-(L^1\cap \{x_1<0\}),\\
C(L^2)=\{\omega\in \D T^*_{L^2} D^n~|~ \omega(\bdry_{x_1})\geq 0\}.
\end{gather*}
Then define $L^\square:= L^0\cup C(L^1)\cup C(L^2)$.  The fact that $L$ is isotopic to a smoothing of $L^\square$ through Lagrangian disks with Legendrian boundary on $\mathbb{S}T^*D^n$ is not obvious but left to the reader.

\s\n
C. (Lagrangian sum) Given Lagrangian submanifolds $L_1, L_2$ which intersect transversely at one point $p$, we may assume that in a Darboux chart $C=\D T^\ast D^n$ around $p=0$, $L_1, L_2$ are the $0$-section and the fiber over the origin, respectively. We will give a piecewise smooth description of the Lagrangian sum $L_1\# L_2$.

We first deform $L_2\cap C$ into piecewise smooth Lagrangians $(L_2\cap C)^\square_\pm$ as in (A) above. Although $(L_2\cap C)^\square_\pm$ does not agree with $L_2\cap C$ on $\mathbb{S}T^* D^n$, outside of a small neighborhood of the $0$-section it is a graphical Lagrangian (i.e., defined as the graph of an exact 1-form) over $L_2$. Hence using a cutoff function we may assume that $(L_2\cap C)^\square_\pm$ agrees with $L_2\cap C$ outside of a small neighborhood of the $0$-section of $\D T^\ast D^n$.  We then define $(L_2)^\square_\pm= (L_2\setminus C) \cup (L_2\cap C)^\square_\pm$.

\begin{lemma} \label{lem:alternative_Lagrangian_sum}
	Up to smoothing, $L_1 \# L_2$ is Lagrangian isotopic rel boundary to $( L_1 \cup (L_2)^\square_+ ) \setminus int(D^n)$ and $L_2 \# L_1$ is Lagrangian isotopic rel boundary to $( L_1 \cup (L_2)^\square_-) \setminus int(D^n)$.
\end{lemma}

The proof of the lemma is straightforward and is left to the reader.

\s\n
D. (Lagrangian boundary sum $L_1 \#_b L_2$)  Assume in a Darboux chart $C=\D T^\ast D^n$ around $p=0$,
$$L_1\cap C = \{ x_1 \geq 0, y=0 \} \subset D^n\times\{0\}, \quad L_2 \cap C= \{ y_1 \geq 0,x=0 \} \subset \D T^\ast_0 D^n.$$
Let $D^n_{x_1\geq 0} = D^n \cap \{ x_1 \geq 0 \}$.  We then define
$$(L_2)^\square_+= (L_2\setminus C)\cup ((L_2\cap C)^\square_+ \cap \{x_1\geq 0\}),$$
where we are assuming $(L_2\cap C)^\square_+ \cap \{x_1\geq 0\}$ is identified with $L_2\cap C$ outside of a small neighborhood of the $0$-section of $\D T^* D^n$.

\begin{lemma} \label{lem:alternative_boundary_sum}
	Up to smoothing, $L_1 \#_b L_2$ is Lagrangian isotopic rel boundary to $( L_1 \cup (L_2)^\square_+) \setminus int(D^n_{x_1\geq 0})$.
\end{lemma}

We also omit the proof of this lemma.

\subsubsection{Legendrian handleslide lemma}

The goal of this subsection is to prove an important technical lemma called the {\em Legendrian handleslide lemma} and which is used frequently in the rest of the paper.  We continue to use the notation from the proof of \autoref{thm:bypass_attachment}.

Let $K$ be the core Legendrian disk of the contact $n$-handle $H_n$ with $\p K = \Lambda_- \uplus \Lambda_+ \subset \Gamma^0$.  We identify $R_+ (\p_2 H_n)$ with the disk cotangent bundle $\D T^\ast D^n$ of the unit ball $D^n\subset \R^n$ so that $K$ is identified with the $0$-section of $T^\ast D^n$. Let $x=(x_1,\dots,x_n)$ be the coordinates on $\R^n$ and $y=(y_1,\dots,y_n)$ be the dual coordinates on $T^\ast \R^n$.  Then $R_+(S)=R_+(\Sigma^0)\cup \D T^\ast D^n$. 
We will not be distinguishing Lagrangians and their Legendrian lifts.

\begin{lemma}[Legendrian handleslide lemma] \label{lemma: trace equals boundary connect sum}
	Let $(\Sigma \times [0,1], \xi)$ be a bypass attachment along $(\Lambda_+, \Lambda_-; D_+, D_-)$ and let $D'_+ \subset R_+(S)$ be the Lagrangian disk obtained from $D_+^{-\epsilon}$ by concatenating with the trace of a Legendrian handleslide from $\Lambda_+^{-\epsilon}$ to $\Lambda_-^\epsilon$. Then its Legendrian lift, also called $D_+'$, is Legendrian isotopic to $(D_+ \uplus_b K)^{\epsilon, \Gamma^0}$
	relative to the boundary in $\Sigma\times[0,1]$.
\end{lemma}

Here we are assuming that $\bdry D'_+= \bdry(D_+ \uplus_b K)^{\epsilon, \Gamma^0}$.

\begin{proof}[Sketch of proof] $\mbox{}$

\s\n
{\em Step 1.}
Let $\Lambda\subset \Gamma^0$ be a Legendrian with a single short Reeb chord $c$ from $\Lambda$ to $\Lambda_-\uplus \Lambda_+$.  Applying a Legendrian handleslide to $\Lambda$ across $\D T^* D^n$ corresponding to $c$ is equivalent to taking the Legendrian sum $\bdry \Theta \uplus \Lambda$, where $\Theta$ is a $\Theta$-disk in the unit cotangent bundle $\mathbb{S}T^* D^n$ which intersects the cocore $T^*_0 D^n$ once.

A $\Theta$-disk is isotopic rel boundary to a Lagrangian disk with Legendrian boundary (still called $\Theta$).  If $\Lambda$ bounds a Lagrangian disk $D\subset R_+(\Sigma^0)$, then applying the Hamiltonian isotopy corresponding to the above Legendrian handleslide yields $\Theta\uplus_b D\subset R_+(S)$.

\s\n
{\em Step 2.}
Next we Legendrian isotop the boundary of $(D_+ \uplus_b K)^{\epsilon, \Gamma^0}$ down through $\mathbb{U}T^* D^n$. Let us write $D'''_+$ for the result of the corresponding Hamiltonian isotopy. Writing $D''_+$ for the Lagrangian disk obtained from $D_+^\epsilon$ by concatenating with the trace of a Legendrian handleslide across $\D T^* D^n$, we see that there exists a $\Theta$-disk in $\mathbb{S}T^* D^n$ which intersects the cocore $T^*_0 D^n$ once (basically the same one as in Step 1), such that:
\begin{gather}
\nonumber D''_+= D_+^\epsilon \uplus_b \Theta,\\
\label{eqn for D'''} D'''_+= D_+^\epsilon \uplus_b \Theta \uplus_b K^{-\epsilon,\Gamma^0}.
\end{gather}

\s\n
{\em Step 3.} We claim that $\Theta \uplus_b K^{-\epsilon,\Gamma^0}$ is Legendrian isotopic (relative to the point of intersection between $D_+^\epsilon$ and $\Theta$) to a $\Theta$-disk $\Theta'$ with boundary on $\D T^* D^n|_{\bdry D^n}$ and which intersects the core $K$ once.  This will imply the lemma by Step 1 since $$D'''_+= D_+^\epsilon \uplus_b \Theta'= D_+^{-\epsilon}.$$

To prove the claim we use piecewise smooth models $L^\square$ and $L^\square_+$ for $\Theta$ and $K^{-\epsilon,\Gamma^0}$ from \autoref{subsection: piecewise smooth descriptions}(B) and (A).  Since $L^\square$ and $L^\square_+$ agree on $\{x_1>0\}$ and on $L^0=int(D^n)$, by contracting the overlapping portions of $L^\square$ and $L^\square_+$ it follows that (the smoothed version of) $L^\square\uplus_b L^\square_+$ is Lagrangian isotopic to
$$(\D T^*_{L^1\cap\{x_1<0\}} D^n) \cup C(L^2),$$
which in turn is Lagrangian isotopic to $\Theta'$.
\end{proof}

\begin{lemma} \label{lem:bypass_criterion}
	Let $(\Sigma \times [0,1], \xi)$ be a bypass attachment along $(\Lambda_+, \Lambda_-; D_+, D_-)$. Then $(D_- \uplus_b D_+) \cup K$ is Legendrian isotopic to the standard Legendrian unknot.
\end{lemma}

\begin{proof}
	By construction $D'_+ \cup D_-$ bounds the core disk of the contact $(n+1)$-handle, which is a $\Theta$-disk.  By Lemma~\ref{lemma: trace equals boundary connect sum}, $D'_+$ is Legendrian isotopic to $D_+\uplus_b K$ relative to the boundary. Hence $(D_+ \uplus_b K) \cup D_-$ is Legendrian isotopic to the standard unknot. The lemma then follows from \autoref{lem: Legendrian_wheel}.
\end{proof}

\subsection{Dual bypass attachment} \label{subsec:up_side_down_bypass}

As indicated in \autoref{subsec:(n+1)_handle}, since the direction of a contact vector field may be reversed, a bypass attachment to $\Sigma \times \{0\}$ which yields $(\Sigma \times [0,1],\xi)$ may be viewed as an upside down bypass attachment to $\Sigma \times \{1\}$. 

Continuing to use the notation from \autoref{thm:bypass_attachment}, suppose that $(\Sigma \times [0,1],\xi)$ is given by a bypass attachment along $(\Lambda_+, \Lambda_-; D_+, D_-)$. The gluing contactomorphism $\psi: \p \overline{R^1_+} \to \p \overline{R^1_-}$ is given by a Legendrian handleslide.

We define the \emph{dual bypass attachment data} $(\Lambda^{\dagger}_+, \Lambda^{\dagger}_-; D^{\dagger}_+, D^{\dagger}_-)$ on $\Sigma^1$ as follows: Let $\Lambda^{\dagger}_+ := (\Lambda_- \uplus \Lambda_+)^{-\epsilon/2} \subset  \p \overline{R^1_+}$. Then $\Lambda^{\dagger}_+$ bounds a Lagrangian disk $D^{\dagger}_+$ in $R^1_+$ which is Hamiltonian isotopic to the cocore disk of the Weinstein handle attached along $\Lambda_- \uplus \Lambda_+$ (cf.\ \autoref{thm:bypass_attachment} (2)). Next let $\Lambda^{\dagger}_- := \Lambda^{-\epsilon/2}_- \in  \p \overline{R^1_+}$. A simple Kirby calculus implies that $\psi(\Lambda^{\dagger}_-) = (\Lambda_- \uplus \Lambda_+)^{\epsilon/2} \in \p \overline{R^1_-}$: this involves moving $\Lambda^\dagger_-$ below $(\Lambda^{-\epsilon}_+)'$ to $(\Lambda_-\uplus \Lambda_+)^{-2\epsilon}$, sliding $(\Lambda^{-\epsilon}_+)'$ above the handle attached along $\Lambda_-\uplus \Lambda_+$, and then moving $(\Lambda_-\uplus \Lambda_+)^{-2\epsilon}$ above the same handle. By the same argument as for the positive region, $\psi(\Lambda^{\dagger}_-)$ bounds a Lagrangian disk $D^{\dagger}_-$ in $R^1_-$.

By reversing the orientation of the contact vector field used in defining the bypass attachment to $\Sigma \times \{0\}$ along $(\Lambda_+, \Lambda_-; D_+, D_-)$, one can verify that $(\Sigma \times [0,1],\xi)$ may be realized as a bypass attachment to $\Sigma \times \{1\}$ along $(\Lambda^{\dagger}_+, \Lambda^{\dagger}_-; D^{\dagger}_+, D^{\dagger}_-)$ from below.

\subsection{Smoothly canceling pairs (SCP) of contact handle attachments} \label{subsec:smooth_canceling_pair}

As we have seen in \autoref{subsec:construct_bypass_attachment}, the bypass attachment is a canceling pair of contact handle attachments in the middle dimensions. However, the bypass attachment is not the only possible canceling pair of contact handle attachments in the middle dimensions.

Our goal here is to describe a more general situation, called an {\em SCP attachment}, which is shorthand for a ``smoothly canceling pair of contact handle attachments"; the bypass attachment then becomes a special case of an SCP attachment.
The material in this subsection is presented for completeness and will not be used anywhere else in the paper.  We are also not trying to describe the most general canceling pair of contact handle attachments.


Let $\Sigma$ be a convex hypersurface with contact germ $\xi$ and dividing set $\Gamma$ and let $\Sigma \setminus \Gamma = (R_+,d\lambda_+) \cup (R_-,d\lambda_-)$ be the usual decomposition into Liouville manifolds. 

\begin{definition} \label{defn:SCP_data}
The quadruple $(\Lambda, \Lambda'; D_{\mp}, A_{\pm})$ is called an \emph{SCP attachment data on $\Sigma$} if $D_{\mp} \subset (R_{\mp}, d\lambda_{\mp})$ is a Lagrangian disk with cylindrical end asymptotic to $\Lambda$, and $A_{\pm} \cong  \R\times S^{n-1} \subset R_{\pm}$ is a (necessarily exact) Lagrangian annulus with cylindrical ends which are asymptotic to $\Lambda$ as $t \to \infty$ and to $\Lambda'$ as $t \to -\infty$, where $t$ is the coordinate on $\R$.
\end{definition}

The following result is analogous to \autoref{thm:bypass_attachment}. We state the result for $(\Lambda, \Lambda'; D_-, A_+)$, leaving the corresponding statement for $(\Lambda, \Lambda'; D_+, A_-)$ to the reader.

\begin{prop} \label{prop:SCP_attachment}
	Let $\Sigma$ be a convex hypersurface with contact germ $\xi$ and let $(\Lambda, \Lambda'; D_-, A_+)$ be an SCP attachment data on $\Sigma$. Then there exists an extension of $\xi$ to $\Sigma \times [0,1]$ and a contact vector field $v$ on $\Sigma \times [0,1]$ such that \autoref{thm:bypass_attachment} (1) holds. Moreover, if we write $\Sigma^i \setminus \Gamma^i = R^i_+ \cup R^i_-$ for $i=0,1$, as in \autoref{thm:bypass_attachment}, then:
		\begin{enumerate}
			\item $(R^1_+)^c$ is obtained from $(R^0_+)^c$ by first attaching a Weinstein handle $H$ along $\Lambda'$, and then removing a standard neighborhood of the Lagrangian disk given by gluing $A_+$ and the core disk of $H$ along $\Lambda'$.

			\item $(R^1_-)^c$ is obtained from $(R^0_-)^c$ by attaching a Weinstein handle along $\Lambda'$ and removing a standard neighborhood of $D_-$.

			\item As the ideal boundary of $R^1_{\pm}$, $\Gamma^1$ is obtained from $\Gamma^0$ by performing a contact $(-1)$-surgery along $\Lambda'$ and a contact $(+1)$-surgery along $\Lambda$.

			\item As a smooth manifold, $\Sigma^1$ is given by gluing $\overline{R^1_+}$ and $\overline{R^1_-}$ along the common boundary by the identity map.
		\end{enumerate}
\end{prop}

\begin{proof}[Sketch of proof]
The proof is similar to that of \autoref{thm:bypass_attachment} but simpler, since no handleslides are involved. We first attach a contact $n$-handle $H_n$ to $\Sigma^0$ along $\Lambda'$ to obtain a new convex hypersurface $S$. Then on $R_+(S)$ there is a Lagrangian disk $D_+$ obtained by gluing $A_+$ to the core disk of $R_+ (\p_2 H_n)$ along a parallel copy of $\Lambda'$. Now a contact $(n+1)$-handle $H_{n+1}$ can be attached along $D_+ \cup D_- \subset S$ to complete the SCP attachment. The fact that the attachments of $H_n$ and $H_{n+1}$ cancel each other smoothly follows from observing that the attaching locus of $H_{n+1}$ intersects the belt sphere of $H_n$ in precisely one point.
\end{proof}

In fact, the bypass attachment constructed in \autoref{thm:bypass_attachment} can be realized as a special case of an SCP attachment. In this paper we will only outline the construction and leave the details to a future work.

Let $\Sigma$ be a convex hypersurface with bypass attachment data $(\Lambda_+, \Lambda_-; D_+, D_-)$.  The key observation is that one can construct a Lagrangian annulus from $D_+$ using the \emph{saddle cobordism} studied by \cite{EHK16} in dimension $3$ and \cite{DR16} in higher dimensions. In our situation, let $(\Gamma, \eta)$ be a contact manifold and $\Lambda_{\pm} \subset \Gamma$ be Legendrian spheres that $\eta$-transversely intersect in one point. Then there is a natural way to arrange $\Lambda_- \uplus \Lambda_+$ and $\Lambda_-$ so they also $\eta$-transversely intersect in one point and $(\Lambda_- \uplus \Lambda_+) \uplus \Lambda_-$ is Legendrian isotopic to $\Lambda_+$ (cf.\ \autoref{lem:conjugation}). Then there exists a saddle Lagrangian cobordism $L\subset \R\times \Gamma$ from $(\Lambda_- \uplus \Lambda_+) \cup \Lambda_-^{\epsilon}$ to $\Lambda_+$ such that $L$ is diffeomorphic to $S^n$ with $3$ points removed.

Consider the decomposition $R_+ = R_+^c \cup ([0,\infty) \times \Gamma)$, where $[0,\infty) \times \Gamma$ is the positive half-symplectization of $\Gamma$.   The previous paragraph (after $\R$-translation and truncation) gives a Lagrangian cobordism $L' \in [0,\infty) \times \Gamma$ asymptotic to $(\Lambda_- \uplus \Lambda_+) \cup \Lambda_-^{\epsilon}$ at $+\infty$ and such that $L' \cap ([0,1] \times \Gamma) = [0,1] \times \Lambda_+$. Hence we obtain a Lagrangian annulus $A_+ = L' \cup (D_+ \cap R^c_+) \subset R_+$. Define the associated SCP attachment data $(\Lambda, \Lambda'; D'_-, A_+)$ by setting $\Lambda = \Lambda^{\epsilon}_-, \Lambda' = \Lambda_- \uplus \Lambda_+$, and $D'_- = D^{\epsilon}_-$. The proof of the following lemma will be omitted.

\begin{lemma}
	The contact structure on $\Sigma \times [0,1]$ given by a bypass attachment along $(\Lambda_+, \Lambda_-; D_+, D_-)$ is contactomorphic to the SCP attachment along $(\Lambda, \Lambda'; D'_-, A_+)$ defined above.
\end{lemma}

\section{Examples of bypass attachments} \label{sec:example_bypass}

In this section we discuss three examples: the trivial bypass attachment, the overtwisted bypass attachment, and the anti-bypass attachment. All three examples were first studied by the first author \cite{Hon00,Hon02} in dimension $3$. Our constructions in this section generalize those of the $3$-dimensional case and share many similarities with the $3$-dimensional case. However, as we will see, the higher-dimensional bypass attachments also exhibit phenomena which are not present in the $3$-dimensional case.

\subsection{Trivial bypass attachment} \label{subsec:trivial_bypass}

Let $(\R^{2n}, \lambda =\tfrac{1}{2}(x \cdot dy - y \cdot dx))$ be the standard Liouville domain and let $B^{2n} \subset \R^{2n}$ be the closed unit ball. Then $(S^{2n-1} = \p B^{2n}, \xi_{\std}=\ker \lambda|_{S^{2n-1}})$ is the standard contact sphere. The {\em standard Lagrangian disk} $\{ |x| \leq 1, y=0 \} \subset B^{2n}$ intersects $S^{2n-1}$ along the standard Legendrian unknot $U$.

Since the connected sum (resp.\ boundary connected sum) of a contact manifold $(M,\xi)$ (resp.\ Liouville domain $W$) with $(S^{2n-1}, \xi_{\std})$ (resp.\ $(B^{2n}, \lambda)$) is trivial, we define a Lagrangian disk $D\subset W$ with Legendrian boundary $\Lambda\subset M$ to be {\em standard} if it can be Lagrangian isotoped to the standard Lagrangian disk in $B^{2n}$ such that the boundary remains Legendrian in $M$ during the isotopy.  The boundary $\Lambda$ of a standard Lagrangian disk $D$ is clearly a standard Legendrian unknot.

\begin{definition} \label{defn:trivial_bypass}
A bypass attachment data $(\Lambda_+, \Lambda_-; D_+,D_-)$ on a convex hypersurface $\Sigma$ is {\em trivial} if
\begin{enumerate}
\item[(T$_\pm$)] $\Lambda_\pm$ is the standard Legendrian unknot, $D_\pm \subset R_\pm$ is (the completion of) the standard Lagrangian disk bounded by $\Lambda_\pm$, and $\Lambda_+$ is below $\Lambda_-$ (resp.\ $\Lambda_-$ is above $\Lambda_+$) in the sense of \autoref{subsec:Legendrian_sum_property}.
\end{enumerate}
A bypass attachment is {\em trivial} if the corresponding attaching data is trivial; see \autoref{fig:trivial_bypass}.
\end{definition}

\begin{figure}[ht]
	\begin{overpic}[scale=.3]{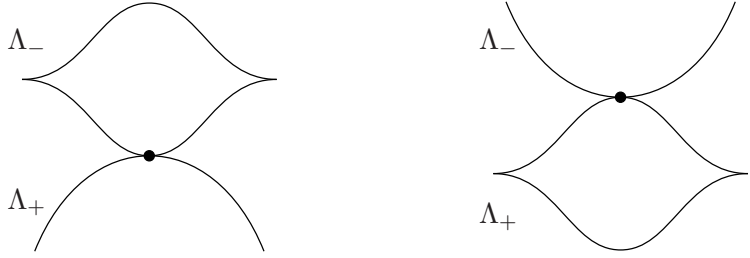}
		\put(-2,6){$\Lambda_+$}
		\put(-2,28){$\Lambda_-$}
		\put(63,4){$\Lambda_+$}
		\put(63,28){$\Lambda_-$}
	\end{overpic}
	\caption{The front projection of the trivial bypass attachment in the dividing set. (T$_-$) is to the left and (T$_+$) is to the right.}
	\label{fig:trivial_bypass}
\end{figure}

As the terminology suggests, a trivial bypass attachment on $\Sigma$ produces an $I$-invariant contact structure on $\Sigma \times [0,1]$. The proof of this fact in dimension $3$, as given in \cite{Hon02}, relies on Eliashberg's uniqueness result \cite{Eli92} for tight contact structures on the $3$-ball. Since the higher-dimensional analog of Eliashberg's result is known to be false (cf.\ Eliashberg~\cite{Eli91} and Ustilovsky~\cite{Us}), one needs a different approach. The proof that a trivial bypass attachment is trivial will be postponed to \autoref{subsec:bypass_to_OB} after we establish a dictionary between bypass attachments and (partial) open book decompositions.

For the moment we apply \autoref{thm:bypass_attachment} to show that the trivial bypass attachment does not change the germ of the contact structure on the convex hypersurface.

\begin{lemma} \label{lemma:trivial_bypass}
	Let $\Sigma$ be a convex hypersurface. If the contact structure $\zeta$ on $\Sigma \times [0,1]$ is given by a trivial bypass attachment on $\Sigma = \Sigma \times \{0\}$, then $\zeta|_{\Sigma \times \{0\}} = \zeta|_{\Sigma \times \{1\}}$ as contact germs.
\end{lemma}

\begin{proof}
Using the notation from \autoref{thm:bypass_attachment}, it suffices to show $R^0_\pm \cong R^1_\pm$ as Liouville manifolds. We consider (T$_+$) in \autoref{defn:trivial_bypass}; (T$_-$) is similar.
	
	According to \autoref{thm:bypass_attachment}, modulo completion, $R^1_+$ can be obtained from $R^0_+$ by removing a standard neighborhood of the standard Lagrangian disk $D_+$ and attaching a Weinstein handle along $(\Lambda_- \uplus \Lambda_+)^{\epsilon}$. By the homotopy theory of Liouville manifolds (see for example \cite{CE2012}), removing a neighborhood of $D_+ \subset R^0_+$ is equivalent to attaching a (subcritical) symplectic $(n-1)$-handle to $R^0_+$, which in turn is canceled by the Weinstein handle attachment along $(\Lambda_- \uplus \Lambda_+)^\epsilon$; see \autoref{fig:trivial_bypass_slide}(a).
	
	Next we turn to the negative region. Modulo completion, $R^1_-$ can be obtained from $R^0_-$ by removing a standard neighborhood of $D_-$ and attaching a Weinstein handle along $(\Lambda_- \uplus \Lambda_+)^{-\epsilon}$. These two operations cancel each other since $(\Lambda_- \uplus \Lambda_+)^{-\epsilon}$ is Legendrian isotopic to $\Lambda_-$ by a (generalized) Reidemeister I move; see \autoref{fig:trivial_bypass_slide}(b).
\end{proof}

\begin{figure}[ht]
	\begin{overpic}[scale=.3]{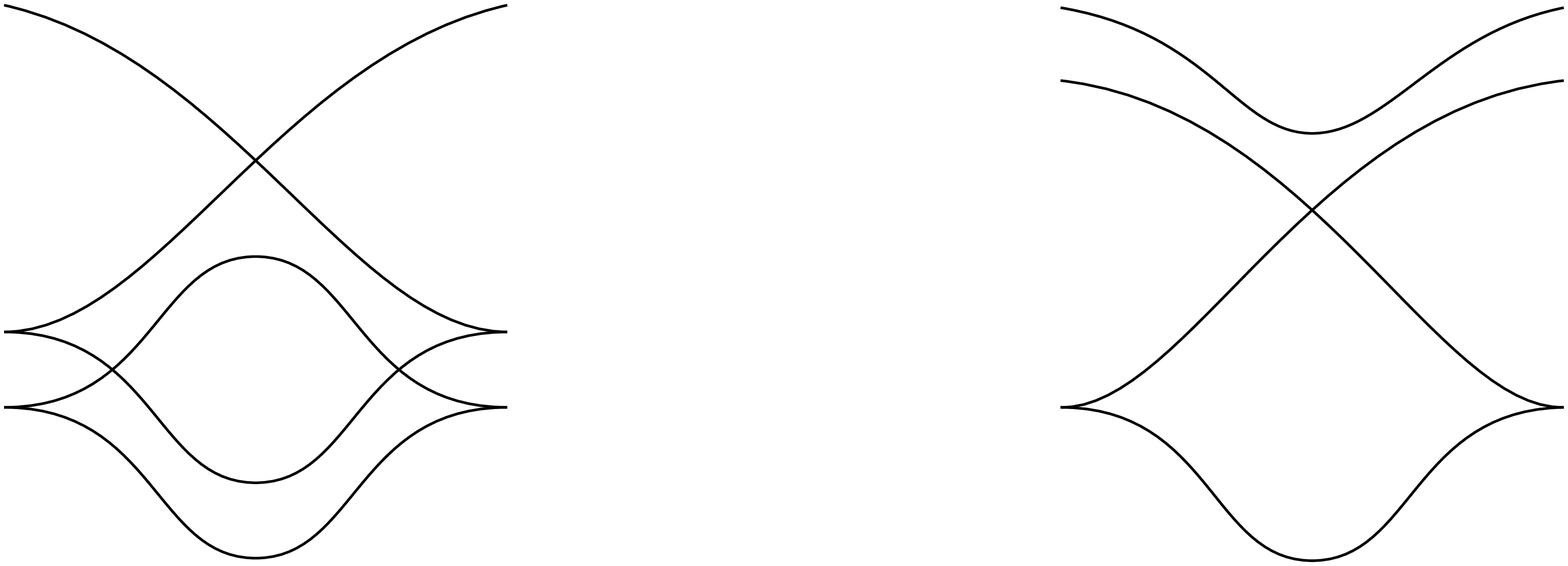}
		\put(-7,27){\small{$(\Lambda_- \uplus \Lambda_+)^\epsilon$}}
		\put(1,4){$\Lambda_+$}
		\put(53,3){\small{$(\Lambda_- \uplus \Lambda_+)^{-\epsilon}$}}
		\put(69,37){$\Lambda_-$}
		\put(28,4){\small{$(+1)$}}
		\put(18.5,37){\small{$(-1)$}}
		\put(90,37){\small{$(-1)$}}
		\put(94,4){\small{$(+1)$}}
		\put(15,-5){(a)}
		\put(82,-5){(b)}
	\end{overpic}
	\vspace{5mm}
	\caption{Legendrian surgery diagrams of $\Gamma = \p \overline{R}_+ = \p \overline{R}_-$ for the trivial bypass attachment.}
	\label{fig:trivial_bypass_slide}
\end{figure}

\subsection{Overtwisted bypass attachment} \label{subsec:OT_bypass}

Overtwisted contact structures, defined and classified in dimension $3$ by Eliashberg \cite{Eli89} and in all dimensions by Borman-Eliashberg-Murphy \cite{BEM15}, are flexible in the sense that they satisfy a Gromov-type $h$-principle. Besides the definition of overtwistedness given in \cite{BEM15} which we will refer to as ``BEM-overtwistedness'', there are now several equivalent criteria for overtwistedness by the work of Casals-Murphy-Presas \cite{CMP15}. The most relevant criterion for us is based on the theory of loose Legendrians introduced by Murphy \cite{Mur12}. However, since Murphy's theory works only in dimension greater than $3$, we assume throughout this section that the dimension of the ambient contact manifold is at least $5$; see \cite{Hon02} for the discussion of overtwisted bypass attachments in dimension $3$.

Let us recall the definition of a loose chart and loose Legendrian submanifolds in dimension at least $5$. First consider the standard contact $3$-space $(\R^3, \xi_{\std}=\ker (dz-ydx))$. Let $\gamma \subset (\R^3,\xi_{\std})$ be a stabilized Legendrian arc whose front projection is as shown in \autoref{fig:kink}. Specifically, the front projection of $\gamma$ has a unique transverse double point and a unique Reeb chord of length $a$, called the {\em action} of the stabilization. Let $B$ be an open ball in $\R^3$ containing $\gamma$ of action $a$ as defined above.

Now consider the standard Liouville manifold $(T^\ast \R^{n-1},-pdq)$ with the usual coordinates $q,p$. Let
\begin{equation*}
V_C = \{|p|< C, |q|<C\} \subset T^\ast \R^{n-1}.
\end{equation*}
Then $B \times V_C$ is an open subset of $(\R^{2n+1},\xi_{\std})$ which contains the Legendrian submanifold $\Lambda=\gamma \times \{|q|<C,p=0\}$. The pair $(B \times V_C, \Lambda)$ is called a {\em loose chart} if $a/C^2<2$. Finally, a Legendrian submanifold $L \subset (M,\xi)$ is {\em loose} if there exists a Darboux chart $U \subset M$ such that the $(U,U \cap L)$ is contactomorphic to a loose chart.

\begin{figure}[ht]
	\begin{overpic}[scale=.3]{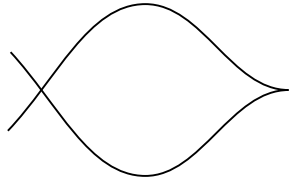}
	\end{overpic}
	\caption{The front projection of a stabilized Legendrian arc.}
	\label{fig:kink}
\end{figure}

The following overtwistedness criterion is due to Casals-Murphy-Presas~\cite{CMP15}:

\begin{theorem} \label{thm:OT_criterion}
	A contact structure is BEM-overtwisted if and only if the standard Legendrian unknot is loose.
\end{theorem}

Let $\Sigma\times[-1,0]_t$ be a ($1$-sided) collar neighborhood of $\Sigma=\Sigma\times\{0\}$ with contact vector field $v=\bdry_t$ and suppose the bypass attachment gives $(\Sigma\times[0,1],\xi)$.  We still write $\Gamma$ for the dividing set of $\Sigma$.

We now define the overtwisted bypass attachment as follows: 

\begin{definition} \label{defn:OT_bypass} $\mbox{}$
	\begin{enumerate}
		\item Given a Legendrian disk $D\subset \Sigma=\Sigma\times\{0\}$ such that $D\cap \Gamma=\bdry D$ and $\bdry D\subset \Gamma$ is Legendrian, the ``pushdown" $D_\flat$ is the smoothing of $(\bdry D\times[-\epsilon,0]) \cup (D\times \{-\epsilon\})$ for $\epsilon>0$ small.
		\item A bypass attachment data $(\Lambda_+, \Lambda_-; D_+, D_-)$ along a convex hypersurface $\Sigma$ is {\em overtwisted} if $(D_- \uplus_b D_+)_\flat$ is a loose Legendrian with a loose chart in $\Sigma\times[-1,0)$.
		\item A bypass attachment is {\em overtwisted} if the corresponding attaching data is overtwisted.
	\end{enumerate}
		
\end{definition}

\begin{remark}
	In dimension $3$ we take $D_-\uplus_b D_+$ to be to be the interval obtained by gluing $D_-$ and $D_+$ along their common intersection.
\end{remark}


\begin{example}
In particular, a bypass attachment $(\Lambda_+, \Lambda_-; D_+, D_-)$ is overtwisted if $\Lambda_-\uplus \Lambda_+$ is loose in $\Gamma$. Indeed, since the dividing set $\Gamma$ has an arbitrarily large neighborhood by \autoref{cor:dividing_set_nbhd_size}, we may assume that $D_- \uplus_b D_+$ contains a loose chart away from a small neighborhood of $\Gamma$.
\end{example}

As the terminology suggests, the contact structure defined by an overtwisted bypass attachment should be overtwisted in the sense of \cite{BEM15}. This is the content of the following lemma:

\begin{lemma} \label{lemma: OT_bypass}
Let $\Sigma$ be a convex hypersurface. If the contact structure $\zeta$ on $\Sigma \times [0,1]$ is given by an overtwisted bypass attachment on $\Sigma = \Sigma \times \{0\}$, then $\zeta$ is overtwisted.
\end{lemma}

\begin{proof}
Let $(\Lambda_+, \Lambda_-; D_+, D_-)$ be an overtwisted attachment data on $\Sigma$. If $K$ is the core Legendrian disk of the contact $n$-handle in the bypass attachment, then $(D_- \uplus_b D_+) \cup K$ is the standard Legendrian unknot by \autoref{lem:bypass_criterion}.  On the other hand, $(D_- \uplus_b D_+) \cup K$ is loose by definition.  Hence $\zeta$ is overtwisted by \autoref{thm:OT_criterion}.
\end{proof}

\begin{remark}
In dimension $3$, the overtwistedness of an overtwisted bypass attachment (called the ``disallowed" bypass attachment in \cite[Figure 6]{Hon02}) follows immediately from Giroux's criterion for determining when a convex surface is overtwisted. In dimensions greater than $3$, there is currently no analog of Giroux's criterion. Nevertheless we will see in \autoref{sec:application} that certain overtwisted bypass attachments also yield overtwisted convex hypersurfaces in any dimension.
\end{remark}

\begin{question}
Give criteria for determining (i) precisely when a convex hypersurface is tight and (ii) precisely when a bypass attached to a tight convex hypersurface $\Sigma$ yields an overtwisted contact $\Sigma\times[0,1]$.
\end{question}

\subsection{Anti-bypass attachment}

Given a convex hypersurface $\Sigma$ and a bypass attachment data $(\Lambda_+^0, \Lambda_-^0; D_+^0, D_-^0)$, a bypass can be attached to $\Sigma=\Sigma\times\{0\}$ along $(\Lambda_+^0, \Lambda_-^0; D_+^0, D_-^0)$ to produce a contact structure $\zeta$ on $\Sigma \times [0,1]$. We say the bypass is attached to $\Sigma$ ``from above''. Analogously, a bypass may be attached to $\Sigma=\Sigma\times\{0\}$ along $(\Lambda_+^0, \Lambda_-^0; D_+^0, D_-^0)$ ``from below'' to produce a contact structure $\zeta^{\vee}$ on $\Sigma \times [-1,0]$. We call $\zeta^{\vee}$ the \emph{anti-bypass attachment} of $\zeta$.

The goal of this subsection is to show that the concatenation of a bypass attachment with its anti-bypass attachment is overtwisted.

\begin{prop} \label{prop:anti_bypass}
If $(\Sigma \times [0,1], \zeta)$ is given by a bypass attachment along a convex hypersurface $\Sigma\times\{0\}$ and $(\Sigma \times [-1,0], \zeta^{\vee})$ is given by the anti-bypass attachment along $\Sigma\times\{0\}$, then the concatenation $(\Sigma \times [-1,1], \zeta^{\vee} \cup \zeta)$ is overtwisted.
\end{prop}

\begin{proof}
The idea of the proof is as follows. First we turn the anti-bypass attachment upside down using \autoref{subsec:up_side_down_bypass} so that $(\Sigma \times [-1,1], {\zeta^{\vee}} \cup \zeta)$ is given by two consecutive bypass attachments. Then we observe that the two consecutive bypass attachments can be made ``disjoint'' and hence the order of the attachments can be interchanged. Once this is done, we see that one of the bypass attachments becomes overtwisted, and hence the proposition follows from \autoref{lemma: OT_bypass}.

Here are the details. As before we write $\Sigma^i=\Sigma\times\{i\}$, $\Gamma^i=\Gamma_{\Sigma^i}$, and $\Sigma^i\setminus \Gamma^i=R^i_+\cup R^i_-$.
By \autoref{subsec:up_side_down_bypass}, $(\Sigma \times [-1,0], \zeta^{\vee})$ can be obtained by a bypass attachment to $\Sigma^{-1}$ along some $(\Lambda_+, \Lambda_-; D_+, D_-)$.  
The positive region $R^0_+$, modulo completion, is obtained from $R^{-1}_+$ by attaching a Weinstein handle along $(\Lambda_- \uplus \Lambda_+)^\epsilon$ and removing a standard neighborhood of the Lagrangian disk $D_+ \subset R^{-1}_+$; the description for $R^0_-$ is similar. We also assume that the size of the neighborhoods where these surgeries are performed is much smaller than $\epsilon$.

By the description from \autoref{subsec:up_side_down_bypass}, $(\Sigma \times [0,1], \zeta)$ is given by a bypass attachment to $\Sigma^0$ along $(\Lambda^0_+, \Lambda^0_-; D^0_+, D^0_-)$, where $\Lambda^0_+ = (\Lambda_- \uplus \Lambda_+)^{\epsilon/2}$, $\Lambda^0_- = \Lambda^{\epsilon/2}_-$, $D^0_+ \subset R^0_+$ is Hamiltonian isotopic to the cocore Lagrangian disk of the Weinstein handle attached along $(\Lambda_- \uplus \Lambda_+)^{\epsilon}$, and $D^0_-$ is Hamiltonian isotopic to the cocore Lagrangian disk of the corresponding Weinstein handle in $R^0_-$. Note that the asymmetry between $(\Lambda^0_+; D^0_+)$ and $(\Lambda^0_-; D^0_-)$ is due to the fact that $\Sigma^0$ is given by gluing $\overline{R^0_+}$ and $\overline{R^0_-}$ via a nontrivial contactomorphism of $\Gamma^0$, and we always parametrize $\Gamma^0$ as the boundary of $\overline{R^0_+}$. In particular, as a contact manifold, $\Gamma^0$ is obtained from $\Gamma^{-1}$ by a contact $(-1)$-surgery along $(\Lambda_- \uplus \Lambda_+)^{\epsilon}$ and a contact $(+1)$-surgery along $\Lambda_+$; see \autoref{fig:commuting_bypass}.

Observe that in $\Gamma^0$, one can Legendrian handleslide $\Lambda^0_+$ up over the Weinstein handle attachment along $(\Lambda_- \uplus \Lambda_+)^\epsilon$ (i.e., ``over the $(-1)$-surgery along $(\Lambda_- \uplus \Lambda_+)^\epsilon$''), relative to a neighborhood of the $\xi$-transverse intersection $\Lambda^0_+ \cap \Lambda^0_-$, to a standard Legendrian unknot $U$ which links $(\Lambda_- \uplus \Lambda_+)^\epsilon$ once as shown in \autoref{fig:commuting_bypass}. Moreover, under the exact symplectic isotopy of $R^0_+$ induced by the above handleslide of $\Lambda^0_+$, the Lagrangian disk $D^0_+$ is identified with a standard Lagrangian disk bounded by $U$. Abusing notation, let us still denote the isotoped bypass attachment data for $\zeta$ by $(\Lambda^0_+, \Lambda^0_-; D^0_+, D^0_-)$.

\s
\begin{figure}[ht]
	\begin{overpic}[scale=.3]{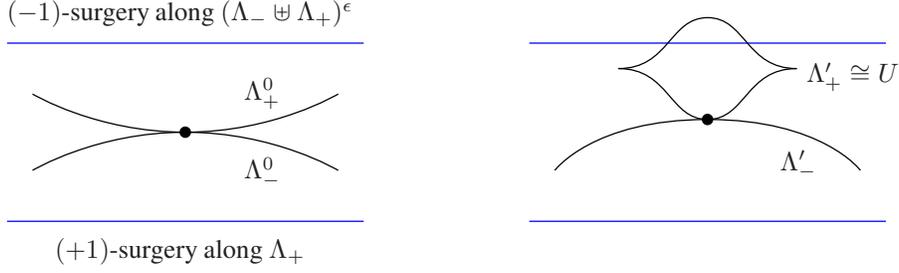}
		\put(5.5,-4){\small{$(+1)$-surgery along $\Lambda_+$}}
		\put(0,23){\small{$(-1)$-surgery along $(\Lambda_- \uplus \Lambda_+)^\epsilon$}}
		\put(27,14){\small{$\Lambda^0_+$}}
		\put(27,5){\small{$\Lambda^0_-$}}
		\put(88,6){\small{$\Lambda'_-$}}
		\put(91,16){\small{$\Lambda'_+ \cong U$}}
	\end{overpic}
	\vspace{5mm}
	\caption{A schematic picture of $\Gamma^0$, obtained by performing surgeries along $(\Lambda_-\uplus \Lambda_+)^\epsilon$ and $\Lambda_+$ (drawn as straight blue lines) inside $\Gamma^{-1}$. The vertical direction represents the Reeb direction.}
	\label{fig:commuting_bypass}
\end{figure}

We now claim that the bypass attachment data $(\Lambda^0_+, \Lambda^0_-; D^0_+, D^0_-)$ can be ``pushed down'' to a quadruple $(\Lambda'_+, \Lambda'_-; D'_+, D'_-)$ on $\Sigma^{-1}$ such that the following hold:
	
	\begin{enumerate}
		\item The quadruples $(\Lambda_+, \Lambda_-; D_+, D_-)$ and $(\Lambda'_+, \Lambda'_-; D'_+, D'_-)$ on $\Sigma^{-1}$ are disjoint.
		\item The bypass attachment data $(\Lambda'_+, \Lambda'_-; D'_+, D'_-)$ is overtwisted.
		\item After the first bypass attachment along $(\Lambda_+, \Lambda_-; D_+, D_-)$, the quadruples $(\Lambda'_+, \Lambda'_-; D'_+, D'_-)$ and $(\Lambda^0_+, \Lambda^0_-; D^0_+, D^0_-)$ are contact isotopic.
	\end{enumerate}

The bypass attachment data $(\Lambda'_+, \Lambda'_-; D'_+, D'_-)$ is defined as follows: First let $\Lambda'_+$ be a standard Legendrian unknot which is not linked with $\Lambda_+$ and links $\Lambda_-$ once, i.e., $\Lambda'_+$ bounds a $\Theta$-disk in $\Gamma$ which does not intersect $\Lambda_+$ and transversely intersects $\Lambda_-$ in a point. (Alternatively, $\Lambda'_+$ is the unknot $U$ which is now viewed inside $\Gamma^0$.) Let $D'_+ \subset R^{-1}_+$ be a standard Lagrangian disk bounded by $\Lambda'_+$ which is disjoint from $D_+$; this is possible because $\Lambda'_+$ is not linked with $\Lambda_+$. Next let $\Lambda'_-$ be a Legendrian pushoff of $\Lambda_-$ with respect to the contact $(+1)$-framing, such that $\Lambda'_- \cap \Lambda_+ = \emptyset$. We chose the framing so that $\Lambda'_-$ bounds a Lagrangian disk $D'_- \subset R^{-1}_-$ which is Hamiltonian isotopic to $D_-$ and $D'_- \cap D_- = \emptyset$. Finally $\Lambda'_+$ and $\Lambda'_-$ $\xi$-transversely intersect at a point in such a way that $\Lambda'_+$ is above $\Lambda'_-$ in the sense of \autoref{subsec:Legendrian_sum_property}.

We now verify that $(\Lambda'_+, \Lambda'_-; D'_+, D'_-)$ satisfies Properties (1)--(3): (1) follows from the construction; (2) follows from \autoref{lem:plus_unknot} applied to the Legendrian sum $\Lambda'_- \uplus \Lambda'_+$; and (3) is a consequence of the following two observations: (i) $\Lambda^0_+$ is Legendrian isotopic (via a Legendrian handleslide) to $\Lambda'_+$ viewed as a Legendrian unknot in $\Gamma^0$, and (ii) $\Lambda^0_-$ is Legendrian isotopic to $\Lambda'_-$ viewed as a Legendrian sphere in $\Gamma^0$. See \autoref{fig:commuting_bypass}.

We have now identified the contact manifold $(\Sigma \times [-1,1], \zeta^{\vee} \cup \zeta)$ with two disjoint bypass attachments to $\Sigma^{-1}$ along
$$(\Lambda_+, \Lambda_-; D_+, D_-) ~~\mbox{ and } ~~(\Lambda'_+, \Lambda'_-; D'_+, D'_-).$$
Since the latter bypass attachment data is overtwisted, the contact structure $\zeta^{\vee} \cup \zeta$ is overtwisted by \autoref{lemma: OT_bypass}.
\end{proof}

\section{Overtwisted orange and bypass} \label{sec:orange_bypass}

The goal of this section is to introduce a singular middle-dimensional contractible Legendrian foliation in a contact manifold, which we call the \emph{overtwisted orange} and whose existence implies the BEM-overtwistedness of the contact structure. Before going into details of the construction, let us briefly review the existing constructions of overtwisted objects:
\begin{enumerate}
\item An overtwisted disk in \cite{BEM15} is a piecewise smooth codimension 1 disk equipped with a special contact germ.
\item A {\em plastikstufe}, constructed in \cite{Nie06} and whose existence was shown to be equivalent to BEM-overtwistedness in \cite{CMP15,Hua17}, is a middle-dimensional noncontractible singular Legendrian foliation.
\item A {\em bordered Legendrian open book}, constructed in \cite{MNW13}, is a middle-dimensional noncontractible singular Legendrian foliation.
\end{enumerate}
In comparison, the overtwisted orange has the advantage of being simultaneously middle-dimensional and contractible, and we expect it to be more easily found in contact manifolds, especially those given by open book decompositions.

In \autoref{subsec:bypass} we will also define the \emph{bypass} to be half of an overtwisted orange, whose existence implies the existence of the bypass attachment in the sense of \autoref{subsec:construct_bypass_attachment}. This generalizes the definition of a bypass in dimension $3$ from \cite{Hon00}.

\subsection{Definition of an overtwisted orange} \label{subsec:orange}

We will construct an overtwisted orange $\OO$ in two steps: first we describe $\OO$ as a topological space which is a manifold away from a singular point, and then we define a contact germ on $\OO$ which makes it into a singular Legendrian foliation.

\s\n
{\em Topological description of $\OO$.} Consider the following (not necessarily orientable) rank $n$ vector bundle over $S^1$
	\begin{equation*}
		E =	([0,1] \times \R^n) / (0,x) \sim (1, \sigma(x)),
	\end{equation*}
where $\sigma: \R^n\stackrel\sim \to \R^n$ is given by $\sigma(x_1, x_2, \dots, x_n) = (x_1, -x_2, \dots, -x_n)$. Let $D(E) \subset E$ be the unit disk bundle
	\begin{equation*}
		D(E) = \left\{ (\tau,x) ~|~ \|x\| \leq 1 \right\} \subset E.
	\end{equation*}
It is well-defined since $\sigma$ preserves the Euclidean norm.

We then define $\OO = D(E) / \gamma$, where $\gamma$ is the loop
	\begin{equation*}
		\gamma = \left\{ (\tau, x_0) ~|~ x_0 = (1, 0, \dots, 0) \right\} \subset \p D(E).
	\end{equation*}
Let $p \in \OO$ be the equivalence class of $\gamma$. Then $\OO$ is a smooth manifold away from $p$.  When $n=1$, $\OO$ is just a $2$-dimensional disk.

\s\n
{\em Contact germ of $\OO$}. The contact germ on $\OO$ will be defined in two steps. We first construct a contact germ on a neighborhood of the singular point $p \in \OO$ and then extend the contact germ to all of $\OO$ using the methods of \cite{Hua15}.

\s\n
{\em Step 1.}  Consider the standard contact space $(\R^{2n+1}, \xi_{\std})$ with the contact form
	\begin{equation} \label{eqn:contact_form_in_orange}
		\alpha_{\std} = dz + \sum\nolimits_{i=1}^{n} r_i^2 d\theta_i.
	\end{equation}

Let $\R^{2n-1} = \{ r_n = 0 \} \subset \R^{2n+1}$. Consider the loop of isotropic subspaces
	\begin{equation} \label{eqn:Lagrangian_loop}
		\Lambda_\tau = \left\{ z=0, \theta_1 = \dots = \theta_{n-1} = -\tau\pi \text{ or } (1-\tau)\pi \right\} \subset \R^{2n-1},~~ \tau\in[0,1].
	\end{equation}
Note that $\Lambda_\tau \cap \Lambda_{\tau'} = 0$ if $\tau\not=\tau'$ and $\tau,\tau'\in[0,1)$. Let $W_\tau$ be the $n$-dimensional half-space spanned by $\Lambda_\tau$ and $\{ \theta_n = 2\tau\pi \}$. We define the Legendrian
	\begin{equation*}
		D_\tau = \left\{ r_n \geq R / \sqrt{2} \right\} \subset W_\tau,
	\end{equation*}
where $R=\sqrt{ \sum_{i=1}^{n-1} r_i^2}$; see \autoref{fig:corner}.

\begin{figure}[ht]
	\begin{overpic}[scale=.3]{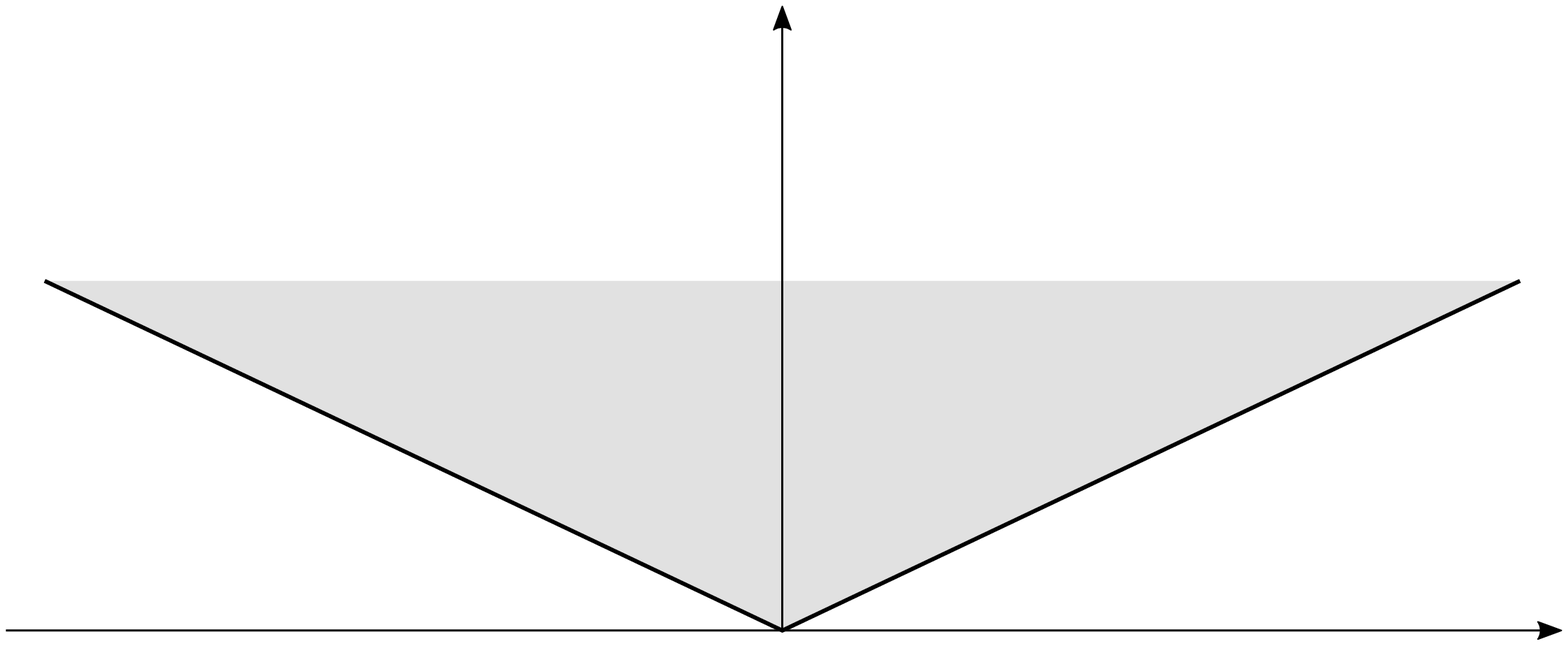}
		\put(43,40){$r_n$}
		\put(102,-1){$\Lambda_\tau$}
	\end{overpic}
	\caption{The shaded region represents the Legendrian $D_\tau$.}
	\label{fig:corner}
\end{figure}

Choose a small $\delta>0$ and let $B_\delta(0) \subset \R^{2n+1}$ be the ball centered at the origin with radius $\delta$. Consider
	\begin{equation} \label{eqn:embedding_X}
		X = \left( \cup_{\tau \in [0,1]} D_\tau \right) \cap B_\delta(0) \subset (\R^{2n+1}, \xi_{\std}).
	\end{equation}
If we write $\p D_\tau = \left\{ r_n = R / \sqrt{2} \right\}$, then the following observation is immediate.

\begin{lemma} \label{lem:orange_peel}
	The subset $\cup_{\tau \in [0,1]} \p D_\tau$ is a Legendrian submanifold away from the origin.
\end{lemma}

\begin{proof}
Since $\p D_\tau$ is isotropic for every $\tau$ by definition, it suffices to check that
\begin{equation} \label{eqn: calculation of Legendrian}
\alpha_{\std} \left( \tfrac{d}{d\tau} \left( \p D_\tau \right) \right) = -\pi \sum\nolimits_{i=1}^{n-1} r_i^2 + \tfrac{R^2}{2} \cdot 2\pi = 0.
\end{equation}
Then observe that $\cup_{\tau \in [0,1]} \p D_\tau$ is smooth away from the origin.
\end{proof}

We define the contact germ on $U(p)\subset \OO$ by choosing a homeomorphism $$U(p)\stackrel\sim\to X \subset (\R^{2n+1}, \xi_{\std})$$ which sends $p$ to $0$ and is a diffeomorphism away from $0$.

\s\n
{\em Step 2.} To extend the contact germ to all of $\OO$, observe that $\OO \setminus U(p)$ may be identified with $D(E) \setminus N(\gamma)$, where $N(\gamma) \subset D(E)$ is a (half) tubular neighborhood of $\gamma$. Following \cite[Lemma 3.8]{Hua15}, there exists a unique contact germ on $D(E) \setminus N(\gamma)$ if we require that
	\begin{enumerate}
		\item for every $\tau_0 \in [0,1]$, $\left( D(E) \setminus N(\gamma) \right) \cap \left\{ \tau=\tau_0 \right\}$ is Legendrian, and

		\item the outer boundary $\p D(E) \setminus N(\gamma)$ is Legendrian.
	\end{enumerate}
Moreover, by \cite[Lemma 3.2]{Hua15}, the contact germs on $U(p)$ and $\OO \setminus U(p)$ can be glued together to give a well-defined contact germ on $\OO$. In fact, the contact germ on $\OO$ is determined by the Legendrian foliation.  Let us write $E_\tau$ for the Legendrian leaf of $\OO$ that extends $D_\tau\cap B_\delta(0)$.

\begin{definition}[Overtwisted orange] \label{defn:orange}
An \emph{overtwisted orange} in a contact manifold $(M,\xi)$ is a topological embedding\footnote{By topological embedding, we mean an injective continuous map.} $\varphi:\OO \hookrightarrow M$ which is a smooth embedding away from a singular point $p\in \OO$ such that there exists an open neighborhood $Op(\varphi(\OO))$ of $\varphi(\OO)$ in $M$ with $\xi|_{Op(\varphi(\OO))}$ contactomorphic to the contact germ on $\OO$ constructed above. The singular point $p \in \OO$ is the \emph{navel}, $\p D(E) \setminus \gamma \subset \OO$ is the \emph{peel} and each $E_\tau \subset \OO, \tau \in [0,1]$, is a \emph{section} of $\OO$.
\end{definition}

In what follows we will not distinguish between $\OO$ and its image under $\varphi$.

\begin{definition}[Maslov index $\mu(\OO)$]
Consider the $2$-disk
	\begin{equation*}
		B = \left\{ 1-\epsilon \leq x_1 \leq 1, x_2 = \dots = x_n = 0 \right\} / \gamma \subset \OO,
	\end{equation*}
where $\epsilon>0$ is small. Observe that each section $E_\tau, \tau \in [0,1]$, intersects $\p B$ in a unique point $p_\tau = E_\tau \cap \p B$ and defines a Lagrangian subspace $T_{p_\tau} E_\tau \subset \xi_{p_\tau}$. Then $\mu(\OO)$ is the Maslov index of the loop $T_{p_\tau} E_\tau, \tau \in [0,1]$, of Lagrangian subspaces with respect to any trivialization of $\xi|_B$.
\end{definition}

A straightforward calculation shows that if the contact manifold has dimension $2n+1$, then $\mu(\OO) = 3-n$.

At this point it is instructive to compare the overtwisted orange with the plastikstufe. Recall from \cite{Nie06} that a plastikstufe is an embedding
$$\psi: P_S = D_{\ot} \times S\hookrightarrow(\R^3 \times T^\ast S, \alpha= \eta_{\ot} - pdq),$$
where $S$ is a closed $(n-1)$-dimensional manifold, $\eta_{\ot}$ is an overtwisted contact form on $\R^3$, $pdq$ is the tautological $1$-form on $T^\ast S$, and $D_{\ot}$ is mapped to a standard overtwisted disk in $(\R^3, \eta_{\ot})$ and $S$ is mapped to the zero section in $T^\ast S$. Given a small disk $B \subset D_{\ot}$ around the center and any point $p \in S$, we consider $B_p = B \times \{p\} \subset P_S$. It is straightforward to check the Maslov index $\mu(\p B_p) = 2$.

In fact there is a family of overtwisted objects interpolating between the overtwisted orange and plastikstufe defined by $\OO_k \times S_k \subset \R^{2(n-k)+1} \times T^\ast S_k$, where $\OO_k$ is the overtwisted orange in $\R^{2(n-k)+1}$ and $S_k$ is a $k$-dimensional closed manifold for any $0 \leq k \leq n-1$. When $k=n-1$ we recover the plastikstufe and when $k=0$ we obtain the overtwisted orange by taking $S_0$ to be a single point.

The terminology ``overtwisted orange'' is justified by the following theorem:

\begin{theorem} \label{thm:orange}
	A contact manifold is BEM-overtwisted if and only if it contains an embedded overtwisted orange.
\end{theorem}

The proof of \autoref{thm:orange} will be given in the next subsection after we prove some basic properties of bypasses. Combining \autoref{thm:orange}, \cite[Theorem 1.2]{Hua17}, and the $h$-principle from \cite{BEM15}, it is not hard to show that:

\begin{cor}
A contact manifold is BEM-overtwisted if and only if it contains an embedded $\OO_k \times S_k$.
\end{cor}

\begin{remark}
In (\ref{eqn:Lagrangian_loop}) $\Lambda_\tau$ rotates clockwise by $\pi$ in each $(r_i,\theta_i)$-plane as $\tau$ goes from 0 to 1. One can construct orange-like objects by defining $\Lambda_\tau$ to be a loop of isotropic subspaces that rotates clockwise by $j\pi$, $j\geq 1$, in each $(r_i,\theta_i)$-plane as $\tau$ goes from 0 to 1, and repeat the rest of the construction of $\OO$. However, the resulting orange-like object is overtwisted if and only if $j=1$.
\end{remark}

\subsection{Definition of bypass} \label{subsec:bypass}

As in dimension $3$, we define a bypass to be one-half of an overtwisted orange.  Continuing to use the notation from \autoref{subsec:orange}, let \begin{gather*}
D(E)^{\wedge} = D(E) \cap \{ 0 \leq \tau \leq \tfrac{1}{2} \}, \quad D(E)^{\vee} = D(E) \cap \{ \tfrac{1}{2} \leq \tau\leq 1 \}\\
\gamma^{\wedge} = \gamma \cap \{ 0 \leq \tau \leq \tfrac{1}{2} \} \subset \p D(E)^{\wedge}, \quad \gamma^{\vee} = \gamma \cap \{ \tfrac{1}{2} \leq \tau \leq 1 \}\subset \p D(E)^{\vee}.
\end{gather*}

\begin{definition}[Bypass and anti-bypass] $\mbox{}$
\be
\item A \emph{bypass} is the space $\Delta = D(E)^{\wedge} / \gamma^{\wedge}$ together with the germ of a contact structure given in \autoref{subsec:orange}.
\item An \emph{anti-bypass} is the space $\Delta^{\vee} = D(E)^{\vee} / \gamma^{\vee}$ together with the germ of a contact structure given in \autoref{subsec:orange}.
\ee
\end{definition}

It follows from the definition that $\Delta \cup \Delta^{\vee}$ is an overtwisted orange.

In the following we describe how to attach a bypass to a convex hypersurface. Let $\Sigma$ be a convex hypersurface and $v$ be a transverse contact vector field. Then we have the usual decomposition $\Sigma \setminus \Gamma = R_+ \cup R_-$, where $\Gamma$ is the $v$-dividing set and $R_\pm$ are (not necessarily connected) Liouville manifolds.

\begin{definition}
A bypass $\Delta = D(E)^{\wedge} / \gamma^{\wedge}$ is \emph{attached} to $\Sigma$ with attaching data $(\Lambda_+, \Lambda_-; D_+, D_-)$ (cf.\ \autoref{defn:bypass_data}) if $D(E) \cap \{ \tau=0 \}$ is identified with the closure $\overline{D}_+$, $D(E) \cap \{ \tau=\tfrac{1}{2} \}$ is identified with $\overline{D}_-$, and under these identifications $\Delta \cap \Sigma = \overline{D}_+ \cup \overline{D}_-$.
\end{definition}

A little care should be taken in attaching higher-dimensional bypasses as they are not smooth. To spell out the details, let us consider the following model of a {\em one-sided} neighborhood $Op(\Sigma)$ of the convex hypersurface $\Sigma$:
	\begin{equation*}
		Op(\Sigma) = (R_+ \times [0,\epsilon]_t) \cup ( \Gamma \times [0,\epsilon]_s \times [0,\tfrac{1}{2}]_t) \cup (R_- \times [\tfrac{1}{2}-\epsilon, \tfrac{1}{2}]_t) / \sim~,
	\end{equation*}
where $(x,0,t) \sim (x,0,t')$ for any $x \in \Gamma$, $t,t' \in [0,\tfrac{1}{2}]$, and $\Sigma$ is identified with $$\Sigma = (R_+ \times \{0\}) \cup (\Gamma \times \{0\} \times [0,\tfrac{1}{2}]) \cup (R_- \times \{\tfrac{1}{2}\}) / \sim.$$
Here $\Gamma$ is regarded as the ideal boundary of $R_\pm$ as usual and $\Gamma \times (0,\epsilon] \times \{0\}$ (resp. $\Gamma \times (0,\epsilon] \times \{\tfrac{1}{2}\}$) is the cylindrical end of $R_+ \times \{0\}$ (resp. $R_- \times \{\tfrac{1}{2}\}$).  

The contact structure $\xi$ on $Op(\Sigma)$ is defined on the three regions as follows so they agree on their common overlaps:
\begin{itemize}
\item On $\Gamma \times [0,\epsilon]_s \times [0,\tfrac{1}{2}]_t$, $\xi = \ker (sdt+\lambda)$ where $\lambda$ is a contact form on $\Gamma$ defining $\xi|_{\Gamma}$.
\item On $R_+ \times [0,\epsilon]_t$ and $R_- \times [\tfrac{1}{2}-\epsilon, \tfrac{1}{2}]_t$, $\xi = \ker(dt+\lambda_{\pm})$, respectively, where $\lambda_\pm$ are Liouville forms on $R_\pm$ (cf.\ \autoref{lem:ideal_compactification}).
\end{itemize}
It follows that $\p_t$ is a contact vector field on $Op(\Sigma)$ which is positively transverse to $R_+$ and negatively transverse to $R_-$.

Let us also denote by $\widetilde{Op(\Sigma)}$ the neighborhood before squashing, i.e., $Op(\Sigma) = \widetilde{Op(\Sigma)} / \sim$. Similarly define $\widetilde{\Sigma}$ such that $\Sigma = \widetilde{\Sigma} / \sim$; see \autoref{fig:convex_surface_model}. 

	\begin{figure}[ht]
		\centerline{
			\begin{overpic}[scale=.6]{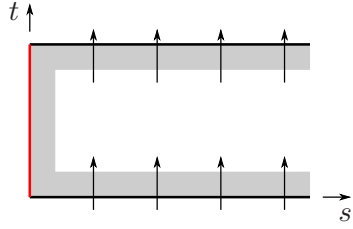}
				\put(96,-3){$s$}
				\put(-5.5,59){$t$}
			\end{overpic}
		}
		\caption{A schematic picture of $\widetilde{Op(\Sigma)}$. The ``suture'' is a concave suture $\Gamma \times \{0\} \times [0,\tfrac{1}{2}]$, depicted in red. {\em The bypass will be attached from the shaded side in \autoref{fig:convex_surface_model}.}}
		\label{fig:convex_surface_model}
	\end{figure}

The plan is to first describe the attachment of a bypass $\Delta$ onto $\widetilde{\Sigma}$ in $\widetilde{Op(\Sigma)}$, and then quotient out by the equivalence relation to get the actual bypass attachment on $\Sigma$. To this end, it suffices to focus on an open neighborhood of the singular point $p \in \Delta$, i.e., the navel. Choose a Darboux chart $\mathcal{U} \cong (\R_{r,\theta,z}^{2n-1}, \xi_{\std})$ around $p$ in $\Gamma$ with contact form $\alpha_{\std}=dz+\sum_{i=1}^{n-1} r_i^2d\theta_i$ such that the following hold:
\begin{itemize}
\item $\Lambda_+ \cap \mathcal{U}$ is identified with $\Lambda_0$ in the ideal boundary of $R_+ = R_+ \times \{0\}$ and $\Lambda_- \cap \mathcal{U}$ is identified with $\Lambda_{1/2}$ in the ideal boundary of $R_- = R_- \times \{\tfrac{1}{2}\}$, where $\Lambda_\tau, \tau \in [0,\tfrac{1}{2}]$, is defined by (\ref{eqn:Lagrangian_loop}).
\item $D_+ \cap (\mathcal{U} \times [0,\epsilon] \times \{0\}) \subset \Gamma \times [0,\epsilon] \times \{0\}$ is identified with $\Lambda_0 \times [0,\epsilon] \times \{0\}$ and $D_- \cap (\mathcal{U} \times [0,\epsilon] \times \{\tfrac{1}{2}\}) \subset \Gamma \times [0,\epsilon] \times \{\tfrac{1}{2}\}$ with $\Lambda_{1/2} \times [0,\epsilon] \times \{\tfrac{1}{2}\}$.
\end{itemize}

Let $\phi: [0,\tfrac{1}{2}] \to [0,\tfrac{1}{2}]$ be a strictly increasing function such that $\phi(0)=0, \phi(\tfrac{1}{2})=\tfrac{1}{2}, \phi'(0)=\phi'(\tfrac{1}{2})=0$, and $\phi'(\tau) > 0$ for all $\tau \in (0,\tfrac{1}{2})$. We then define the Legendrian disks
	\begin{equation} \label{eqn: Dt}
		D_\tau = \left\{ s \geq \phi'(\tau) \pi \sum\nolimits_{i=1}^{n-1} r_i^2  \right\} \subset \Lambda_{\phi(\tau)} \times [0,\epsilon] \times \{\tau\}.\footnote{Strictly speaking, given a fixed neighborhood size $\epsilon > 0$, one may need to stretch the $[0,\tfrac{1}{2}]$-direction and choose $\phi$ with sufficiently small derivative. The details are left to the reader.}
	\end{equation}
Here we recall that $\Lambda_\tau$ is $1$-periodic by definition. Then $\cup_{\tau \in [0,1/2]} \p D_\tau$ is Legendrian as in the proof of \autoref{lem:orange_peel}, where $s,t$ play the role of $r_n,\theta_n$.
Moreover by our choice of $\phi$, it is clear that $D_0 = D_+ \subset R_+, D_{1/2} = D_- \subset R_-$ and $\p D_\tau \cap \widetilde{\Sigma} = \{p\}$ for all $\tau \in (0,\tfrac{1}{2})$.

Finally, by squashing the $[0,\tfrac{1}{2}]$-factor in $\Gamma \times [0,\epsilon] \times [0,\tfrac{1}{2}] \subset \widetilde{Op(\Sigma)}$, we obtain the local model for a bypass $\Delta$ attached to $\Sigma$ along $(\Lambda_+, \Lambda_-; R_+, R_-)$ near the attaching region.

The following theorem generalizes the bypass attachment in dimension $3$. The proof in higher dimensions is different from the $3$-dimensional case since the edge-rounding lemma \cite[Lemma 3.11]{Hon00} in dimension $3$ has no known analog in higher dimensions.

\begin{theorem} \label{thm:bypass}
Let $\Sigma \subset (M, \xi)$ be a convex hypersurface. If there exists a bypass $\Delta$ attached to $\Sigma$ along $(\Lambda_+, \Lambda_-; D_+, D_-)$, then there is a contact embedding $\phi: (\Sigma \times [0,1], \zeta) \hookrightarrow (M, \xi)$ such that $\phi( \Sigma \times \{0\} ) = \Sigma$ and $(\Sigma \times [0,1], \zeta)$ is given by the bypass attachment along $(\Lambda_+, \Lambda_-; D_+, D_-)$ as constructed in \autoref{thm:bypass_attachment}.
\end{theorem}

\begin{proof}
The main idea is to resolve the navel of the bypass $\Delta$ and find the canceling pair of contact handles explicitly.


We start by resolving the navel in a local model. Using the notation from \autoref{subsec:orange}, consider the following path of isotropic subspaces
\begin{equation*}
\Lambda'_\tau = \left\{ z = \epsilon \pi \tau, \theta_1 = \dots = \theta_{n-1} = -\tau\pi \text{ or } (1-\tau)\pi \right\} \subset \R^{2n-1}, ~~ \tau \in [0, \tfrac{1}{2}].
\end{equation*}
for a fixed small $\epsilon>0$.  Inside $W'_\tau=\Lambda'_\tau\times \{ \theta_n = 2\tau\pi \}$ we construct a Legendrian
\begin{equation*}
D'_\tau =\{r_n^2\geq \tfrac{R^2-\epsilon}{2}\}=\{ r_n \geq \sqrt{\tfrac{R^2 - \epsilon}{2}},R \geq \sqrt{\epsilon} \}\cup \{ r_n \geq 0, R \leq \sqrt{\epsilon}\}
\end{equation*}
with piecewise smooth boundary, where $R=\sqrt{\sum_{i=1}^{n-1}r_i^2}$; see \autoref{fig:piecewise_smooth}.
\begin{figure}[ht]
		\centerline{
			\begin{overpic}[scale=.3]{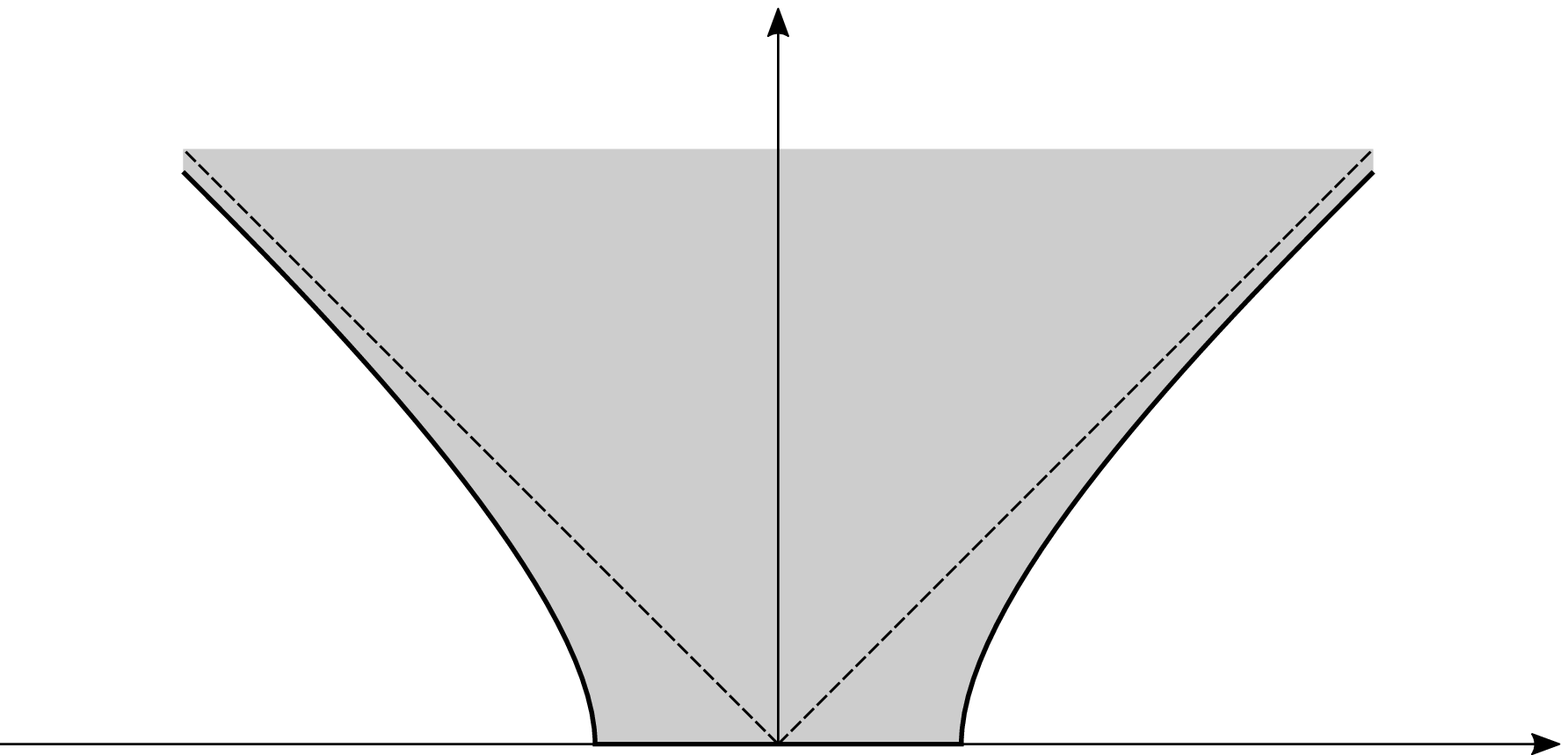}
				\put(43,40){$r_n$}
				\put(102,-1){$\Lambda'_\tau$}
			\end{overpic}
		}
			\caption{The shaded region represents the Legendrian $D'_\tau$.}
			\label{fig:piecewise_smooth}
\end{figure}

There exists a decomposition $\p D'_\tau = \p^{\text{h}} D'_\tau \cup \p^{\text{v}} D'_\tau$ such that
\begin{gather*}
\p^{\text{h}} D'_\tau = \p D'_\tau \cap \{ R \leq \sqrt{\epsilon} \}, \quad \p^{\text{v}} D'_\tau = \p D'_\tau \cap \{ R \geq \sqrt{\epsilon}\}.
\end{gather*}
Here $\cup_{\tau \in [0,1/2]} {\p^{\text{v}} D'_\tau}$ is a smooth Legendrian by the proof of \autoref{lem:orange_peel} and $\cup_{\tau \in [0,1/2]} {\p^{\text{h}} D'_\tau}$ is not Legendrian.
By choosing $\epsilon>0$ small, the Legendrians $D'_\tau$, $\tau \in [0,\tfrac{1}{2}]$, and $\cup_{\tau \in [0,1/2]} {\p^{\text{v}} D'_\tau}$ can be made arbitrarily close to the corresponding Legendrians $D_\tau$ and $\cup_{\tau \in [0,1/2]} {\p^{\text{v}} D_\tau}$. Hence using a cutoff function, we may assume that $D'_\tau = D_\tau$ and $\cup_{\tau \in [0,1/2]} {\p^{\text{v}} D'_\tau} = \cup_{\tau \in [0,1/2]} {\p D_\tau}$ outside of a small neighborhood $B_{\delta/2}(0)$ of the navel.  Let
\begin{align*}
E_\tau'&= (E_\tau\setminus (D_\tau\cap B_{\delta}(0)))\cup (D'_\tau\cap B_{\delta}(0)),\\
\Delta'&= (\Delta\setminus (\Delta\cap B_{\delta}(0)))\cup (\cup_{\tau\in[0,1/2]}D'_\tau \cap B_{\delta}(0)).
\end{align*}
We also set $\bdry^{\text{h}}E'_\tau=\bdry^{\text{h}}D'_\tau$ and $\bdry^{\text{v}}E'_\tau= \bdry E'_\tau \setminus int(\bdry^{\text{h}}D'_\tau)$.

To summarize, by resolving the navel as above, we obtain a piecewise smooth Legendrian foliation $\Delta'$, which we call a \emph{smoothed bypass}, such that:
		\begin{itemize}
			\item $\Delta'$ is foliated by pairwise disjoint Legendrian disks;

			\item $\cup_{\tau \in [0,1/2]} {\p^{\text{h}} E'_\tau} \subset \p \Delta'$ is not Legendrian, but is foliated by pairwise isotropic disks;

			\item $\p \Delta' \setminus \left( \cup_{\tau \in [0,1/2]} {\p^{\text{h}} E'_\tau} \right)$ is a Legendrian disk consisting of singularities of the foliation.
		\end{itemize}
See \autoref{fig:3d_bypass}.
		\begin{figure}[ht]
			\begin{overpic}[scale=.3]{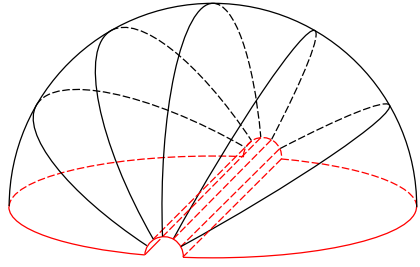}
				
			\end{overpic}
			\caption{The smoothed bypass $\Delta'$.}
			\label{fig:3d_bypass}
		\end{figure}
We have the decomposition $\p \Delta' = \p^{\text{h}} \Delta' \cup \p^{\text{v}} \Delta'$, where
\begin{equation*}
\p^{\text{h}} \Delta' = \left( \cup_{\tau \in [0,1/2]} \p^{\text{h}} E'_\tau \right) \cup E'_0 \cup E'_{1/2}, \quad\p^{\text{v}} \Delta' = \cup_{\tau \in [0,1/2]} \p^{\text{v}} E'_\tau.
\end{equation*}
	
Now the smoothed bypass $\Delta'$ can be attached to $\Sigma$ along $\p^{\text{h}} \Delta'$ such that, up to rounding corners, $E'_0$ is identified with $\overline{D}_+$, $E'_{1/2}$ is identified with $\overline{D}^{\epsilon}_-$, $\p E'_0$ is identified with $\Lambda_+$, $\p E'_{1/2}$ is identified with $\Lambda^{\epsilon}_-$, and $\cup_{\tau \in [0, 1/2]} {\p^{\text{h}} E'_\tau} \subset \Gamma$. Here the superscript $\epsilon$ denotes the $\epsilon$-pushoff in the Reeb direction. In \autoref{fig:3d_bypass} the blue portion of $\p \Delta'$ is $\cup_{\tau \in [0,1/2]} \p^{\text{h}} E'_\tau $, the red portion is $\p^{\text{v}} E'_0 \cup  \p^{\text{v}} E'_{1/2}$, and they both lie in $\Gamma$.

We claim that the Legendrian sphere
\begin{equation*}
\bdry (\bdry^{\text{v}}\Delta')=\left( \cup_{\tau \in [0,1/2]} \p \left( \p^{\text{h}} E'_\tau \right) \right) \cup \p^{\text{v}} E'_0 \cup  \p^{\text{v}} E'_{1/2} \subset \Gamma
\end{equation*}
is Legendrian isotopic to $\Lambda_- \uplus \Lambda_+$. This is most easily seen in the front projection as follows. Take a Darboux chart in $\Gamma$ containing $(\cup_{\tau \in [0,1/2]} {\p^{\text{h}} E'_\tau})\cap B_\delta(0)$, whose front projection is in blue as shown in \autoref{fig:Legendrian_sum_front}. The claim then follows from the front interpretation of the Legendrian sum.

		\begin{figure}[ht]
			\begin{overpic}[scale=.3]{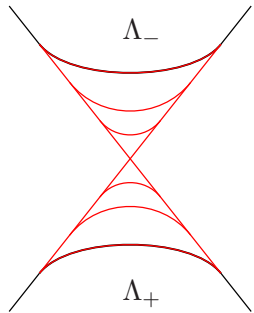}
				\put(37,89){$\Lambda_-$}
				\put(37,4){$\Lambda_+$}
			\end{overpic}
			\caption{The front projection of the Legendrian sum. The colors are in accordance with \autoref{fig:3d_bypass}.}
			\label{fig:Legendrian_sum_front}
		\end{figure}

Now a standard neighborhood $N(\p^{\text{v}} \Delta')$ of $\p^{\text{v}} \Delta'$ may be regarded as a contact $n$-handle attached to $\Sigma$. It remains to observe that the rest of the bypass $\Delta' \setminus N(\p^{\text{v}} \Delta')$ is foliated by Legendrian disks and hence is equipped with a standard contact neighborhood which may be regarded as a contact $(n+1)$-handle. Clearly the two contact handle attachments form a bypass attachment following the discussion in \autoref{subsec:construct_bypass_attachment}. This completes the proof of the theorem.
\end{proof}

As a corollary, we now prove \autoref{thm:orange}.

\begin{proof}[Proof of \autoref{thm:orange}]
If $(M, \xi)$ is BEM-overtwisted, then an embedded overtwisted orange exists by the $h$-principle of \cite{BEM15}. Conversely, given an embedded overtwisted orange $\OO$, there exists (at least locally) a convex hypersurface $\Sigma$ such that $\Sigma \cap \OO = D_0 \cup D_{1/2}$ and $\Sigma$ cuts $\OO$ into a bypass $\Delta$ and an anti-bypass $\Delta^{\vee}$, which are attached to $\Sigma$ from opposite sides. Indeed, such a $\Sigma$ may be obtained by gluing ideal completions of $T^\ast D_0$ and $T^\ast D_{1/2}$ along a Darboux ball around the navel. It then follows from \autoref{thm:bypass} that 
there simultaneously exist a bypass attachment and its anti-bypass attachment on the two sides of $\Sigma$. \autoref{thm:orange} is now a consequence of \autoref{prop:anti_bypass}, noting that \autoref{prop:anti_bypass} holds even when $\Sigma$ is not closed.
\end{proof}

\section{Bypass attachments and contact partial open books} \label{sec:partial_OB}

The Giroux correspondence \cite{Gir02} provides a dictionary between contact structures on any given closed manifold $M$ and certain open book decompositions of $M$.  In dimension $3$, there is a relative version of the Giroux correspondence due to Kazez, Mati\'c, and the first author \cite{HKM09}.  In particular, the triple $(M,\xi,\Gamma)$ consisting of a compact contact $3$-manifold $(M, \xi)$ with convex boundary and dividing set $\Gamma$ on $\bdry M$ admits a supporting partial open book decomposition. \cite{HKM09} also described how a supporting partial open book decomposition of $(M,\xi,\Gamma)$ is modified under a bypass attachment to $\p M$.

The goal of this section is to generalize the notion of a partial open book decomposition to higher dimensions and to describe how a supporting partial open book decomposition of a higher-dimensional contact manifold $(M,\xi,\Gamma)$ with convex boundary is modified under a bypass attachment. 

\subsection{Contact structures and open book decompositions} \label{subsec:OB}

We briefly review the Giroux correspondence, partly to introduce terminology; see for example \cite{Et06} for more details (in dimension $3$).

Given a smooth closed $(2n+1)$-dimensional\footnote{In the topological category, there is no need to restrict to odd-dimensional manifolds. We do so here since we are only interested in open book decompositions of contact manifolds.} manifold $M$, an \emph{open book decomposition} of $M$ consists of a pair $(S, \phi)$ and an identification $M \cong M_{(S,\phi)}$, where $S$ is a compact $2n$-dimensional manifold with nonempty boundary, $\phi \in \aut(S, \p S)$ is a diffeomorphism relative to the boundary, and
\begin{gather*}
M_{(S, \phi)} = S \times [0,1] / \sim \\
(x,t) \sim (x,t'), ~\forall x \in \p S, ~t,t' \in [0,1] \text{ and } (\phi(x),0) \sim (x,1), ~\forall x \in S.
\end{gather*}
We call $S$ the \emph{page}, $B = \p S$ the \emph{binding}, and $\phi: S \stackrel\sim\to S$ the \emph{monodromy} of the open book decomposition.


Equivalently, an open book decomposition of $M$ is a pair $(B,\pi)$, where $B \subset M$ is an embedded closed codimension $2$ submanifold with trivial normal bundle, and $\pi: M \setminus B \to S^1$ is a fibration which is consistent with the trivialization. One can recover the page of the open book decomposition by taking $S = \overline{\pi^{-1} (\theta_0)}$ for any fixed $\theta_0 \in S^1$. In the following we will not distinguish the two points of view of the open book decompositions.

Now suppose $\xi$ is a contact structure on $M$. We say an open book decomposition $(S, \phi)$ of $M$ \emph{supports} $\xi$ if there exists a contact form $\alpha$ for $\xi$ such that $(B,\xi|_B)$ is a contact submanifold, $(S, d(\alpha|_{S}))$ is the ideal compactification of a Liouville manifold, and $\phi \in \symp^c (int(S))$ is a compactly supported (exact) symplectomorphism of $int(S)$.

The existence part of the Giroux correspondence can be stated as follows.

\begin{theorem}[\cite{Gir02}] \label{thm:Giroux_correspondence}
	Any closed contact manifold admits a supporting open book decomposition.
\end{theorem}

\begin{remark}
In fact, according to \cite{Gir02}, \autoref{thm:Giroux_correspondence} can be strengthened in two ways: first, the pages of the supporting open book can be taken to be Weinstein instead of just Liouville; and second, a certain class of supporting open books of a given contact structure can be related to each other via {\em positive stabilizations} (see \autoref{defn:stabilization}).
\end{remark}

\subsection{Contact partial open book decompositions} \label{subsec:partial_OB}

We begin with a description of the (topological) partial open book, which is a straightforward generalization of the definition in dimension $3$ in \cite{HKM09}.

Let $W$ be a compact manifold with corners and $\mathcal{C}\subset \bdry W$ be a closed smooth submanifold of codimension $1$.  Then we say that $W$ {\em has a codimension $1$ corner along $\mathcal{C}$} if a neighborhood of $\mathcal{C} \subset W$ is diffeomorphic to a neighborhood of $\mathcal{C} = \mathcal{C} \times \{0\} \subset \mathcal{C} \times (\R_{\geq 0})^2$ and there are no other corners.  We write $\mathcal{C}={\frak c}(\bdry W)$. 

Given a compact manifold $S$ with nonempty boundary and a codimension 0 compact submanifold $W$ with a codimension 1 corner, we say the embedding $W \subset S$ is a \emph{cornered embedding} if, for any connected component $V$ of $\p W$:
\begin{enumerate}
\item If ${\frak c}(V)=\varnothing$, then $V \subset \p S$.
\item If ${\frak c}(V)\not=\varnothing$, then ${\frak c}(V)$ divides $V$ into finitely many connected components.  The connected components are either contained in the interior $int(S)$ of $S$ or in $\p S$, and moreover two components that are adjacent at the same corner cannot both lie in $int(S)$ or in $\p S$.
\end{enumerate}
See \autoref{fig:cornered_embedding} for an example of cornered embedding in dimension $2$.

 	\begin{figure}[ht]
 		\begin{overpic}[scale=.25]{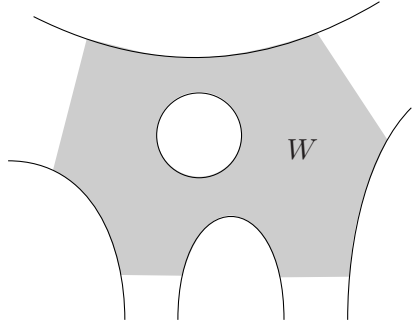}
 			\put(69,40){$W$}
 		\end{overpic}
 		\caption{A cornered embedding $W \subset S$.}
 		\label{fig:cornered_embedding}
 	\end{figure}

\begin{definition}
A \emph{partial open book} is a triple $(S, W, \phi)$, where $S$ is a compact manifold with boundary, $W \subset S$ is a cornered embedding, and $\phi: W \to S$ is a diffeomorphism onto its image such that $\phi|_{\p W \cap \p S}$ is the identity map. We call $S$ the \emph{page} and $\phi$ the \emph{monodromy} of the partial open book.
\end{definition}

Given a partial open book $(S, W, \phi)$, one can associate to it a compact manifold
\begin{gather*}
M_{(S, W, \phi)} = S \times [0,1] / \sim, \\
(x,t) \sim (x,t'), ~\forall x \in \p S, ~t,t' \in [0,1] \text{ and } (\phi(x),0) \sim (x,1), ~\forall x \in W.
\end{gather*}
It is straightforward to see that $\p M_{(S, W, \phi)} = ( S \setminus W ) \times \{1\} \cup ( S \setminus \phi(W) ) \times \{0\}$.

Next we turn to the definition of contact partial open books. Let $S$ be the ideal compactification of a Liouville manifold and let $S^c \subset S$ be the (unique up to deformation equivalence) Liouville domain whose completion has ideal compactification $S$.

A cornered embedding $W^c \subset S^c$ is called a \emph{cornered Liouville embedding} if $S^c$ is a Liouville domain with Liouville vector field $v$ that transversely points into $W^c$ along $\p W^c \cap int(S^c)$ and transversely points out of $W^c$ along $\p W^c \cap \p S^c$. Our primary example of a cornered Liouville embedding is the embedding of a standard neighborhood of a properly embedded Lagrangian disk $D \subset S^c$ with $\p D \subset \p S^c$.

By adding the cylindrical end to $S^c$ and compactifying at infinity, there exists a natural extension of $W^c$ to a cornered embedding $W\subset S$.

\begin{definition} \label{defn:contact_partial_OB}
	A \emph{contact partial open book} is a partial open book $(S, W, \phi)$ such that $S$ is the ideal compactification of a Liouville manifold, $W \subset S$ is the completion of a cornered Liouville embedding $W^c\subset S^c$, and $\phi:  W \to S$ is the completion of an exact symplectomorphism onto its image such that $\phi|_{[R,\infty]\times(\p W \cap \p S)}$ is the identity map for $R\gg 0$.
\end{definition}

Any contact partial open book gives rise to a compact contact manifold $M_{(S, W, \phi)}$ with convex boundary, after rounding corners. In particular, the dividing set on $\p M_{(S, W, \phi)}$ is given by $\p ( S \setminus W )$, up to corner rounding. This is the relative version of a higher-dimensional analog of Thurston-Winkelnkemper's theorem in \cite{TW75}.

When $W=S$, i.e., in the closed case, a positive stabilization changes the open book but does not change the resulting contact manifold up to contactomorphism; see \cite{Et06} for more details.  The following is a related operation:

\begin{definition} [Positive partial stabilization] \label{defn:stabilization}
	Given a contact partial open book $(S, W, \phi)$ and a properly embedded Lagrangian disk $L \subset S^c$ with Legendrian boundary $\p L \subset \p S^c$ which is disjoint from $W$, a \emph{positive partial stabilization of $(S, W, \phi)$} is a contact partial open book $(S', W', \phi')$ such that:
		\begin{itemize}
			\item $S'$ is the ideal compactification of the completion of the Liouville domain obtained by attaching a Weinstein handle $H$ to $S^c$ along (a parallel copy of) $\p L$. Let $D, D^{\dagger}$ be the core and cocore Lagrangian disks of $H$.

			\item $(W')^c = W^c \cup H$ and $\phi'= \tau_{L\cup D}\circ (\phi\cup \op{id}_{H})$, where $\tau_{L\cup D}$ is the positive Dehn twist along the Lagrangian sphere $L \cup D$.
		\end{itemize}
\end{definition}

The following lemma shows that a positive partial stabilization does not change the resulting contact manifold up to contactomorphism:

\begin{lemma} \label{lem:positive_stabilization}
	If $(S', W', \phi')$ is a positive partial stabilization of the contact open book $(S, W, \phi)$, then $M_{(S', W', \phi')}$ is contactomorphic to $M_{(S, W, \phi)}$.
\end{lemma}

\begin{proof}[Sketch of proof]
The proof is analogous to the closed case (cf.\ \cite[Theorem 2.17]{Et06}).
	
Consider the contact partial open book $(S_0,I_0, \tau)$, where $S_0$ is the ideal compactification of $T^\ast S^n$ with the standard Liouville structure, $I_0$ is the standard rectangular neighborhood of the ideal compactification of a Lagrangian fiber $L_0=T^*_p S^n\subset T^\ast S^n$, and $\tau$ is the positive Dehn twist along the $0$-section. The associated contact manifold $M_{(S_0, I_0, \tau)}$ is contactomorphic to the unit ball in $(\R^{2n+1}, \xi_{\std})$.
	
We now describe the ``plumbing" $(S', W', \phi')=(S, W, \phi)*(S_0,I_0, \tau)$:  Let $R_0=S_0\setminus I_0$ and let  $R$ be the standard rectangular neighborhood of $L$.  We then set $S'=S_0\cup_{R_0=R} S$ such that $L_0$ and $L$ intersect exactly once and $W'= I_0\cup W$. After plumbing, $I_0$ and $L_0$ become the $n$-handle $H$ and cocore disk $D^\dagger$. We then construct $M_{(S', W', \phi')}$ by starting from
\begin{gather*}
M^\circ:=S'\times[0,1]/\sim \\
(x,t)\sim (x,t'), \forall x\in \bdry S', t,t'\in[0,1] ~ \mbox{and} ~ (x,1)\sim ((\phi\cup \op{id}_H)(x),0), \forall x\in W',
\end{gather*}
and applying a Legendrian $(-1)$-surgery along $(L\cup D) \times\{\tfrac{3}{4}\}$. Then $M_{(S', W', \phi')}$ decomposes into $H_0\cup H_1$, where $H_0= (S\times[0,1] - R\times[{1\over 2},1])/\sim$ and $H_1$ is obtained from $(S_0\times[0,1] - R\times[0,{1\over 2}])/\sim$ by applying the $(-1)$-surgery.

Let us write $\bdry R= \bdry_1 R\cup \bdry_2 R$, where $\bdry_1R=\bdry R\cup \bdry S$ and $\bdry _2 R=\bdry R\cup \bdry S_0$.
We then consider the sphere
$$S^{2n}:= (R\times \{\tfrac{1}{2}, 1\}) \cup (\bdry_1 R\times [0,\tfrac{1}{2}]) \cup (\bdry_2 R\times [\tfrac{1}{2}, 1])/\sim.$$
Observe that
\begin{gather*}
\bdry H_0\cap \bdry H_1=(R\times \{\tfrac{1}{2}\}) \cup (\bdry_1 R\times [0,\tfrac{1}{2}]) \cup (\bdry_2 R\times [\tfrac{1}{2}, 1])/\sim
\end{gather*}
is a disk in $S^{2n}$.  In fact this disk can be viewed as the positive region of the standard convex $S^{2n}$ given as the boundary of the standard $(2n+1)$-ball.
Hence, modulo carefully rounding $S^{2n}$ (which is the reason for calling this proof only a sketch), we see that $M_{(S', W', \phi')}$ is contactomorphic to the boundary connected sum $M_{(S, W, \phi)} \# M_{(S_0,I_0, \tau)}$, which in turn is contactomorphic to $M_{(S, W, \phi)}$.
\end{proof}

The significance of contact partial open book lies in the following relative analog of the Giroux correspondence \cite{Gir02}, which we state as a conjecture:

\begin{conj}
A compact contact manifold with convex boundary is contactomorphic to $M_{(S, W, \phi)}$ for some contact partial open book $(S, W, \phi)$.
\end{conj}

\subsection{From a bypass attachment to a contact partial open book} \label{subsec:bypass_to_OB}

In this subsection we describe how a supporting partial open book decomposition of a contact $(2n+1)$-manifold $(M,\xi,\Gamma)$ with convex boundary is modified under a bypass attachment.

Let $(M, \xi,\Gamma)$ be a compact contact manifold with convex boundary. Suppose there exists a contact partial open book $(S, W, \phi)$ such that $(M,\xi)$ is contactomorphic to $M_{(S, W, \phi)}$. Let us write $\p M \setminus \Gamma = R_+ \cup R_-$. Then $R_+$ is identified with $(S \setminus W) \times \{1\}$, $R_-$ is identified with $(S \setminus \phi(W)) \times \{0\}$, and $\Gamma$ is identified with $\p (S \setminus W)$.

Given bypass attachment data $(\Lambda_+, \Lambda_-; D_+, D_-)$ on $\p M$, we will construct a new contact partial open book $(S^{\flat}, W^{\flat}, \phi^{\flat})$.  {\em To simplify notation, we will not distinguish a Liouville domain, its completion, and the ideal compactification of the completion, as well as Lagrangian submanifolds in each of those, whenever it is clear from the context what we mean.} (For example, we view $\Lambda_- \uplus \Lambda_+$ as a Legendrian on the boundary of the Liouville domain $R^c_+$ instead of the ideal boundary of $R_+$.)

\s\n
{\em Construction of $(S^{\flat}, W^{\flat}, \phi^{\flat})$.}
Let $\epsilon>0$ be small. Let $S^{\flat}$ be the Liouville domain obtained by attaching a Weinstein handle $H$ to $R_+$ along $(\Lambda_- \uplus \Lambda_+)^{2\epsilon}$ and let $W^{\flat} = W \sqcup N_{\epsilon} (D_+) \subset S^{\flat}$ be the disjoint union of $W$ and a standard $\epsilon$-neighborhood $N_{\epsilon} (D_+)$ of the Lagrangian disk $D_+$. Here we may take $N_{\epsilon} (D_+) \subset S \setminus W$ to be a cornered Liouville embedding.

The partial monodromy $\phi^{\flat}: W^{\flat} \to S^{\flat}$ is defined as follows: Let $D$ be the core Lagrangian disk of $H$ and let $\widetilde D \subset S^{\flat}$ be the Hamiltonian isotopic copy of $D$ such that $\p \widetilde D = \Lambda_- \uplus \Lambda_+$. Let $D_-^{\ast} \subset R_+\subset S\times\{1\}$ be the parallel copy of $D_-\subset R_-\subset S\times\{0\}$, obtained by translating in the $[0,1]$-direction of the partial open book; note that $D_- \cap \phi(W) =\emptyset$. 
Viewing $\Lambda_- \uplus \Lambda_+$ as $\xi$-transversely intersecting $\Lambda_-$ at a point, we consider the Legendrian boundary sum $\widetilde D \uplus_b D_-^{\ast}$.  Note that $\p (\widetilde D \uplus_b D^{\ast}_-) = (\Lambda_- \uplus \Lambda_+) \uplus \Lambda_- \cong \Lambda_+ = \p D_+$. We then define $\phi^{\flat}: W^{\flat} \to S^{\flat}$ such that $\phi^{\flat}|_W = \phi$ and $\phi^{\flat} (N_{\epsilon} (D_+)) = N_{\epsilon} (\widetilde D \uplus_b D^{\ast}_-)$ such that $\phi^{\flat} (D_+) = \widetilde D \uplus_b D^\ast_-$. Note that $\phi^{\flat}$ restricts to the identity map on $\p W^{\flat}$ by construction.


\begin{lemma} \label{lemma:bypass_to_OB}
Let $(M, \xi,\Gamma)$ be a compact contact manifold with convex boundary supported by a contact partial open book $(S, W, \phi)$. If $(M^{\flat}, \xi^{\flat},\Gamma^\flat)$ is the contact manifold obtained by attaching a bypass along $(\Lambda_+, \Lambda_-; D_+, D_-)$ on $(M, \xi,\Gamma)$, then $(M^{\flat}, \xi^{\flat},\Gamma^\flat)$ is supported by the contact partial open book $(S^{\flat}, W^{\flat}, \phi^{\flat})$.
\end{lemma}

\begin{proof}
The proof is a direct translation from the contact handle description of the bypass attachment to the partial open book description, and is completely analogous to the $3$-dimensional case studied by Kazez, Mati\'c, and the first author in \cite{HKM09}.

Recall that a bypass attachment to $\p M$ along $(\Lambda_+, \Lambda_-; D_+, D_-)$ is a smoothly canceling pair of contact $n$- and $(n+1)$-handle attachments. The contact $n$-handle is attached along $(\Lambda_- \uplus \Lambda_+)^{2\epsilon} \subset \Gamma$, where $\Gamma \subset \p M$ is the dividing set and $\epsilon>0$ is small. This produces a new contact manifold $(M', \xi',\Gamma')$ which corresponds to the contact partial open book $(S^{\flat}, W', \phi')$, where $S^{\flat}$ is obtained by attaching a Weinstein handle to $S$ along $(\Lambda_- \uplus \Lambda_+)^{2\epsilon}$, $W'=W$, and $\phi'=\phi$.
	
The contact $(n+1)$-handle is attached to $\p M'$ along the Legendrian sphere $(\widetilde{D}' \uplus_b D_-) \cup D_+$ which we describe now. Consider the usual decomposition of the convex boundary $\p M' \setminus \Gamma' = R'_+ \cup R'_-$. The region $R'_-$ is obtained by attaching a Weinstein handle $H'$ to $R_-$ along $(\Lambda_- \uplus \Lambda_+)^{2\epsilon}$. Let $D'$ be the core Lagrangian disk of $H'$. It is clear that $D'$ can be identified with the core disk $D$ of $H$ by a translation in the $[0,1]$-direction of the partial open book. Let $\widetilde{D}' \subset R'_-$ be a Hamiltonian isotopic copy of $D'$ such that $\p \widetilde{D}' = \Lambda_- \uplus \Lambda_+$. Now consider $\widetilde{D}' \uplus_b D_-\subset R'_-$ with boundary $(\Lambda_- \uplus \Lambda_+) \uplus \Lambda_-\cong \Lambda_+$. By Lemma~\ref{lemma: trace equals boundary connect sum} (applied to $D_-$ instead of $D_+$), the contact $(n+1)$-handle is attached along the Legendrian sphere $(\widetilde{D}' \uplus_b D_-) \cup D_+$. Moreover the core $\Theta$-disk of the contact $(n+1)$-handle is foliated by a family of Legendrian disks $D_\tau, \tau \in [0,1]$, such that $D_0 = D_+$ and $D_1 = \widetilde D' \uplus_b D_-$.
	
To wrap up the proof, note that $\widetilde D' \uplus_b D_-$ is naturally identified with $\widetilde D \uplus_b D_-^{\ast}$ by a translation in the $[0,1]$-direction of the partial open book $(S^{\flat}, W', \phi')$. Hence the contact $(n+1)$-handle attachment produces $(M^{\flat}, \xi^{\flat},\Gamma^\flat)$, whose partial open book is given by $(S^{\flat}, W^{\flat}, \phi^{\flat})$ with $W^{\flat} = W \sqcup N_{\epsilon} (D_+)$ and $\phi^{\flat} (D_+) = \widetilde D \uplus_b D_-^{\ast}$ as desired.
\end{proof}

As an application of \autoref{lemma:bypass_to_OB}, we show that the trivial bypass attachment (cf.\ \autoref{defn:trivial_bypass}) is indeed trivial.

\begin{prop} \label{prop:trivial_bypass}
Let $(M, \xi,\Gamma)$ be a contact manifold with convex boundary and $(\Lambda_+, \Lambda_-; D_+, D_-)$ be a trivial bypass attachment data on $\p M$. The trivial bypass attachment to $M$ along $(\Lambda_+, \Lambda_-; D_+, D_-)$ yields a new contact structure $(M,\xi',\Gamma')$ which is contactomorphic to $(M,\xi,\Gamma)$. In other words, a trivial bypass attachment does not change the isotopy class of the contact structure.
\end{prop}

\begin{proof}
Suppose the trivial bypass attachment data $(\Lambda_+, \Lambda_-; D_+, D_-)$ is given such that $\Lambda_-$ is the standard Legendrian unknot, $D_-$ is the standard Lagrangian disk bounded by $\Lambda_-$, and $\Lambda_-$ is above $\Lambda_+$. The other case can be handled similarly. Let $\p M \setminus \Gamma = R_+ \cup R_-$ be the usual decomposition of the convex boundary.
	
It is not necessary to have a global partial open book decomposition of $(M,\xi)$. In fact, consider a $(4\epsilon)$-neighborhood $N_{4\epsilon} (D_+) \subset R_+$ of $D_+$ and a $(4\epsilon)$-neighborhood $N_{4\epsilon} (\Lambda_+) \subset \p M$ of $\Lambda_+$. Since $\Lambda_-$ is the standard Legendrian unknot bounding a standard Lagrangian disk $D_-$, we may assume that $D_- \subset N_{4\epsilon} (\Lambda_+)$. Let us consider the following local partial open book
\begin{gather*}
\left( \left( N_{4\epsilon} (\Lambda_+) \cap R_+ \right) \times [0,1] \right) \cup \left( N_{4\epsilon} (D_+) \times [\tfrac{1}{2},1] \right) / \sim,\\ (x,t) \sim (x,t') \text{ for any } x \in N_{4\epsilon} (\Lambda_+) \cap \Gamma, t,t' \in [0,1],
\end{gather*}
with page $S_0 = N_{4\epsilon} (\Lambda_+) \cap R_-, S_1 = N_{4\epsilon} (D_+)$, and trivial partial monodromy in $(M,\xi)$; see \autoref{fig:local_OB}.

		\begin{figure}[ht]
			\begin{overpic}[scale=.4]{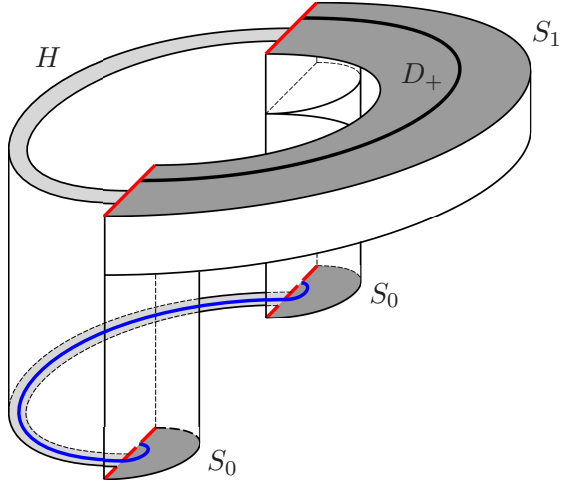}
				\put(100,84){$S_1$}
				\put(38,2){$S_0$}
				\put(69,34){$S_0$}
				\put(5,79){$H$}
				\put(75,76){$D_+$}
			\end{overpic}
			\caption{The local partial open book corresponding to the trivial bypass attachment.}
			\label{fig:local_OB}
		\end{figure}
	
By \autoref{lemma:bypass_to_OB}, the trivial bypass attachment along $(\Lambda_+, \Lambda_-; D_+, D_-)$ yields a new (local) partial open book $(S', W', \phi')$ defined as follows:
\begin{itemize}
\item For $i=0,1$, $S'_i$ is obtained from $S_i$ by attaching a Weinstein handle $H$ along $(\Lambda_- \uplus \Lambda_+)^{2\epsilon} \cong \Lambda_+^{2\epsilon}$ such that the attaching region is a standard $\epsilon$-neighborhood of $\Lambda_+^{2\epsilon}$.
\item $W'$ is a standard $\epsilon$-neighborhood of $D_+ \subset S \subset S'_1$.
\item Let $D_+^{2\epsilon} \subset S_1$ be a Hamiltonian isotopic copy of $D_+$ such that $\p D_+^{2\epsilon} = \Lambda_+^{2\epsilon}$. Then $L = D_+^{2\epsilon} \cup K$ is a Lagrangian sphere in $S'_1$, where $K$ is the core Lagrangian disk of the Weinstein handle $H$. Consider the Lagrangian disk $\tau_L (D_+)$ where $\tau_L$ denotes the positive Dehn twist along $L$. This is illustrated by the blue arc in \autoref{fig:local_OB}. Up to a Hamiltonian isotopy, we can assume that a standard $\epsilon$-neighborhood $N_{\epsilon} (\tau_L (D_+))$ is contained in $S'_0$. Define $\phi': W' \to N_{\epsilon} (\tau_L (D_+)) \subset S'_0$ to be a symplectomorphism onto its image such that $\phi' (D_+) = \tau_L (D_+)$.
\end{itemize}

By \autoref{lem:positive_stabilization}, the new contact manifold $(M, \xi',\Gamma')$ can be obtained from $(M, \xi,\Gamma)$ by boundary connected sum with the unit ball in $(\R^{2n+1}, \xi_{\std})$. Hence $\xi'$ is contactomorphic to $\xi$ as claimed.
\end{proof}

\begin{cor} \label{cor:trivial_bypass_exists}
Given a contact manifold $(M,\xi)$ and a convex hypersurface $\Sigma \subset M$, a trivial bypass exists on both sides of $\Sigma$.
\end{cor}

\begin{proof}
By \autoref{prop:trivial_bypass}, the trivial bypass exists in an $I$-invariant neighborhood of $\Sigma$.
\end{proof}


\section{Bypass rotation}

The notion of a bypass would be rather useless if there were no effective ways of finding them. In dimension $3$, this problem was solved in two different ways:
\be
\item Using the $C^\infty$-genericity of convex surfaces and the Giroux flexibility theorem for characteristic foliations, one can easily find bypasses, say, in a convex disk with Legendrian boundary. This is the strategy exploited in depth in \cite{Hon00} and subsequent works.
\item If one knows a priori the existence of one bypass, then there is an operation known as a \emph{bypass rotation}, defined in \cite{HKM05} and exploited extensively in \cite{HKM07}, which produces many new bypasses from the given one.
\ee
In dimensions greater than $3$, (1) fails essentially due to the nongenericity of convex hypersurfaces in higher dimensions. (2), however, can be generalized to higher dimensions with care, as we will explain in this section.


Let $(V, d\lambda, X)$ be a Liouville domain. A \emph{Liouville subdomain} $W \subset V$ is a compact subdomain such that $X$ is positively transverse to $\p W$; in particular $\lambda|_{\p W}$ is a contact $1$-form. The following key definition, roughly speaking, gives a partial ordering on (based) Legendrian submanifolds in $\p V$.

\begin{definition}[To the left/right] \label{defn:Legendrian_to_the_left}
Let $(V, d\lambda, X)$ be a Liouville domain and let $\Lambda, \Lambda' \subset (\p V, \lambda|_{\p V})$ be Legendrian submanifolds. Suppose there exists a point $p \in \Lambda \cap \Lambda'$ such that $\Lambda = \Lambda'$ on an open neighborhood $Op(p)\subset \bdry V$ of $p$. We say that {\em $\Lambda'$ is to the left of $\Lambda$ relative to $p$} if there exists a Liouville subdomain $W \subset V$ such that the following hold:
\be
\item There exist Lagrangian cylinders $\Lambda \times [0,1], \Lambda'\times[0,1]\subset V \setminus \mathring{W}$, where $\mathring{W}$ denotes the interior of $W$, such that
\be
\item $\Lambda = \Lambda \times \{0\} , \Lambda'=\Lambda'\times\{0\}\subset \p V$,
\item $\Lambda_1 := \Lambda \times \{1\}, \Lambda_1':=\Lambda'\times\{1\} \subset \p W$ are Legendrian, and
\item The restricted Liouville vector fields $X|_{\Lambda \times [0,1]_s}$ and $X|_{\Lambda' \times [0,1]_s}$ both coincide with $-\p_s$.\footnote{To explain the unusual choice of the direction of the $s$-coordinate, we note that in this section, bypasses are attached from the inside of a contact manifold with convex boundary, as opposed to the constructions in \autoref{subsec:bypass}, where the bypass is attached from the outside of a convex hypersurface.}
\ee
In particular, $\Lambda_1 = \Lambda'_1$ on an open neighborhood $Op(p_1)\subset \bdry W$ of the unique intersection point $p_1$ of the trajectory of $X$ passing through $p$ and $\p W$.
			
\item There exists a Legendrian isotopy $\Xi_\tau \subset (\p W, \lambda|_{\p W}), \tau \in [0,1]$, such that
\be
\item $\Xi_0 = \Lambda_1, \Xi_1 = \Lambda'_1$,
\item $\Xi_\tau \cap Op(p_1) \equiv \Lambda_1 \cap Op(p_1)$ for all $t$, and
\item $\p_\tau (\Xi_\tau (x)) > 0$ for any $x \in \Xi_\tau \setminus \overline{Op(p_1)}$.
\ee
\ee
	If $\Lambda'$ is to the left of $\Lambda$, then we also say $\Lambda$ is to the \emph{right} of $\Lambda'$.
\end{definition}

\begin{remark}
The notion of left/right is analogous to that of left/right in \cite{HKM07}. It is a subtle but important point in \autoref{defn:Legendrian_to_the_left} that the Legendrian isotopy defining the ``left rotation'' fixes not just one point $p$ but an open neighborhood of it. This, of course, was not a problem in $3$-dimensional contact topology because in that case $\Lambda$ is a $0$-dimensional sphere.
\end{remark}

\begin{definition}[Anchoring]
Let $D \subset (V, d\lambda, X)$ be a Lagrangian disk with a cylindrical end on $\Lambda = \p D \subset \p V$. If $W \subset V$ is a Liouville subdomain, then we say that $D$ is \emph{anchored in $W$} if $D \cap (V \setminus \mathring{W}) = \Lambda \times [0,1]$ with $\Lambda = \Lambda \times \{0\}$ and $X|_{\Lambda \times [0,1]_s} = c\p_s$ for some $c>0$. We write $D_W := D \cap W$.
\end{definition}

Suppose $\Lambda'$ is to the left (or right) of $\Lambda$ with Legendrian isotopy $\Xi_s \subset \p W$ as in \autoref{defn:Legendrian_to_the_left}. Then it follows from \autoref{cor:isotop_Lagrangian_disk} that there is a Lagrangian disk $D'_W \subset W$ with $\p D'_W = \Lambda'_1$ (unique up to exact Lagrangian isotopy rel boundary), which approximates the trace of $\Xi_s$. Hence there exists a Lagrangian disk $D' = D'_W \cup (\Lambda' \times [0,1]) \subset V$ with $\p D' = \Lambda'$.

\begin{lemma}[Bypass rotation] \label{lem:bypass_sliding}
Let $\Sigma \subset (M,\xi)$ be a convex hypersurface. Suppose a bypass $\Delta$ exists on one side of $\Sigma$ along $(\Lambda_+, \Lambda_-; D_+, D_-)$. If $\Lambda'_+$ is to the left of $\Lambda_+$ relative to $p\in \Lambda_+ \cap \Lambda_-$, then a bypass $\Delta'$ exists on the same side of $\Sigma$ along $(\Lambda'_+, \Lambda_-; D'_+, D_-)$, where $D'_+$ is the Lagrangian disk constructed above (up to completion). Similarly, if $\Lambda'_-$ is to the right of $\Lambda_-$, then a bypass exists along $(\Lambda_+, \Lambda'_-; D_+, D'_-)$.
\end{lemma}

\begin{remark}
Even in the special case where $D_+$ is anchored in $W \subset R_+$ with $R_+ \setminus W$ symplectomorphic to the (truncated) symplectization of $\Gamma$, 
the bypass rotation may still be nontrivial because the trace of the Legendrian isotopy $\Xi_s\subset \Gamma$ may intersect $\Lambda_-$.
\end{remark}


\begin{proof}[Proof of \autoref{lem:bypass_sliding}]
We decide to treat the case when $\Lambda'_-$ is to the right of $\Lambda_-$, for a reason that will be explained momentarily; the other case is proved in the same way. Let $W \subset R_-$ be the Liouville subdomain in which $D_-$ is anchored.

We continue to use the local model for the contact germ on $\Sigma$ as defined in \autoref{subsec:bypass}, with a twist. Namely, we assume without loss of generality that the contact manifold $M$ is (locally) identified with the half-open rectangular region depicted in \autoref{fig:convex_surface_model}. Moreover, the convex hypersurface $\Sigma = \p M$ is co-oriented by an inward-pointing contact vector field. In particular $R_+$ is identified with $\{s=0\}$ and $R_-$ is identified with $\{ s=\tfrac{1}{2} \}$. Roughly speaking, we will construct the new bypass $\Delta'$ in three steps. Firstly, we extend $M$ by attaching to it from the above (in the sense of \autoref{fig:convex_surface_model}) a large invariant neighborhood of $R_-$. The resulting contact manifold is of course canonically contactomorphic to the original $(M,\xi)$. Secondly, we use the assumption that $\Lambda'_-$ is to the right of $\Lambda_-$ to construct an isotopy of isotropic spheres in $R_-$. Finally the new bypass $\Delta'$ is constructed by extending $\Delta$ to the invariant neighborhood of $R_-$ from the first step using the isotropic isotopy from the second step. Here are the details.

\s\n
\textsc{Step 1.} \emph{Enlarging $(M,\xi)$.}

\s
Recall the standard model of the contact structure in a neighborhood of $\Sigma$ given by $Op(\Sigma)$ from \autoref{subsec:bypass}. Note that $Op(\Sigma)$ is obtained from an invariant contact structure on $\widetilde{Op(\Sigma)}$ by collapsing the dividing set. To enlarge $(M,\xi)$ as explained before, we attach to $\widetilde{Op(\Sigma)}$ an invariant layer $\overline{R}_- \times [\tfrac{1}{2},K]_t$ for $K \gg 0$ along $\overline{R}_- \times \{\tfrac{1}{2}\}$ and specify the contact form in coordinates as follows:
		\begin{itemize}
			\item On $R_- \times [\tfrac{1}{2},K]_t$, $\xi =\ker(dt+\lambda_-)$.
			\item On $\Gamma \times [0,\epsilon]_s \times [\tfrac{1}{2},K]_t$, $\xi = \ker (sdt + \lambda)$, where $\Gamma \times \{0\}$ is identified with $\p \overline{R}_-$ and $\lambda$ is the contact form on $\Gamma$.
			\item On $\p W \times [1-\epsilon,1+\epsilon]_{s'} \times [\tfrac{1}{2}, K]_t$, $\xi = \ker (s'dt + \lambda')$, where $\p W$ is identified with $\p W \times \{1\}$ and $\lambda'$ is the contact form on $\p W$.
			\item In a small tubular neighborhood of the Lagrangian cylinder $\Lambda_- \times [0,1] \times \{\tfrac{1}{2}\} \subset R_- \times \{\tfrac{1}{2}\}$, we have the unified contact form $\xi = \ker (sdt + \lambda)$. As the notation suggests, it implies in particular that
			\be
				\item the coordinates $s, s'$ can be extended to just one coordinate $s \in [0,1+\epsilon]$ on $Op(\Lambda_-) \times [1-\epsilon,1] \times [\tfrac{1}{2},K]$, and
				
				\item the contact forms $\lambda, \lambda'$ can be identified with each other on $Op(\Lambda_-)$, which, if one wishes, can be identified with the standard contact form on the 1-jet space $J^1 (\Lambda_-)$.
			\ee
			The same applies to a tubular neighborhood of the Lagrangian cylinder $\Lambda'_- \times [0,1] \times \{K\} \subset R_- \times \{K\}$.
		\end{itemize}

\s\n
\textsc{Step 2.} \emph{Constructing the isotropic isotopy in $\overline{R}_-$.}

\s
Let $Op(p)\subset Op'(p)\subset \bdry \overline{R}_-$ be an open neighborhood of $p$ and a slightly larger open neighborhood of $p$; let $Op(p_1)\subset Op'(p_1)\subset \bdry W$ be corresponding open neighborhoods of $p_1$.

We now describe a piecewise smooth ``based'' isotopy $\Xi_r, r \in [0,3]$, of isotropic spheres\footnote{Strictly speaking, $\overline{R}_-$ is not symplectic. So by isotropy we mean it is isotropic when restricted to $R_-$ and is Legendrian when restricted to the contact boundary $\p \overline{R}_-$.} in $\overline{R}_-$ such that $\Xi_0 = \Lambda_-$, $\Xi_3 = \Lambda'_-$, and {\em $\Xi_r$ is the constant isotopy when restricted to $Op(p)$}:
\be
\item For $r\in[0,1]$, $\Xi_r$ is contained in the cylinder $\Lambda_- \times [0,1] \subset \overline{R}_- \setminus \mathring{W}$ such that $\Xi_1=\Lambda_- \times \{1\}$ outside of $Op'(p_1)$.
\item For $r\in[1,2]$, $\Xi_r$ is the positive isotopy in $\p W$ from $\Xi_1 = \Lambda_- \times \{1\}$ to $\Xi_2 = \Lambda'_- \times \{1\}$ outside of $Op'(p_1)$ as in \autoref{defn:Legendrian_to_the_left}.
\item For $r\in[2,3]$, $\Xi_r$ is contained in the cylinder $\Lambda'_- \times [0,1] \subset \overline{R}_- \setminus \mathring{W}$. 
\ee
See \autoref{fig:bypass_slide}.

\begin{figure}[ht]
\centerline{
\begin{overpic}[scale=.45]{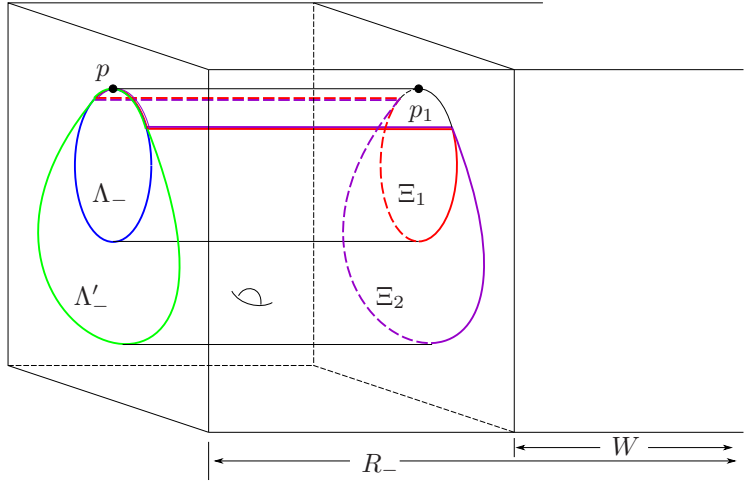}
			\put(82,3.5){\small{$W$}}
			\put(48,1.3){\small{$R_-$}}
			\put(11.5,38){\small{$\Lambda_-$}}
			\put(9,24){\small{$\Lambda'_-$}}
			\put(53,38){\small{$\Xi_1$}}
			\put(50,24){\small{$\Xi_2$}}
			\put(12,55){\small{$p$}}
            \put(54.5,49.5){\small{$p_1$}}
\end{overpic}
}
\caption{A schematic picture of the isotropic isotopy $\Xi_r, r \in [0,3]$ in $R_-$.}
\label{fig:bypass_slide}
\end{figure}

Observe that by construction $\lambda_- (\p_r \Xi_r) = 0$ for $r \in [0,1] \cup [2,3]$. Using the specific contact form on neighborhoods of the Lagrangian cylinders $\Lambda_- \times [0,1], \Lambda'_- \times [0,1] \subset R_-$ constructed in Step 1, one can slightly perturb $\Xi_r$ by a $C^\infty$-small perturbation such that, abusing notation,
		\begin{equation} \label{eqn:negative_isotopy}
			\lambda_- (\p _r \Xi_r (x)) < 0
		\end{equation}
for all $r \in [0,3]$ and $x \neq p$. From now on, we will refer to the perturbed (isotropic) isotopy by $\Xi_r$.

\s\n
\textsc{Step 3.} \emph{Constructing $\Delta'$.}

\s

Recall the function $\phi: [0,\tfrac{1}{2}] \to [0,\tfrac{1}{2}]$ from \autoref{subsec:bypass}. We define the function
$$\widetilde{\phi}: [0,K] \to [0,\tfrac{1}{2}]$$
such that $\widetilde{\phi} = \phi$ on $[0,\tfrac{1}{2}-\delta]$, $\widetilde{\phi}$ maps $[\tfrac{1}{2}-\delta, K)$ diffeomorphically onto $[\phi(\tfrac{1}{2}-\delta), \tfrac{1}{2})$ with small (positive) derivative, and $\widetilde{\phi} (K) = \tfrac{1}{2}$ with $\widetilde{\phi}'(K) = 0$. Here $\delta>0$ is small. Continuing to use the notation from \autoref{subsec:bypass}, we first truncate $\Delta$ to obtain
$$\Delta^{\delta} = \Delta \cap \{ 0 \leq t \leq \tfrac{1}{2}-\delta \}.$$
Then consider the (reparametrized) isotropic isotopy $\Xi^{\delta}_{\mu}, \mu \in [\phi(\tfrac{1}{2}-\delta), \tfrac{1}{2}]$ in $\overline{R}_-$ obtained by concatenating $\Lambda_t, \tau \in [\tfrac{1}{2}-\delta, \tfrac{1}{2}]$ and $\Xi_r, r \in [0,3]$ by identifying $\Lambda_{1/2} = \Lambda_- = \Xi_0 \subset \Gamma = \p \overline{R}_-$. We define the Legendrian cylinder $T \subset \overline{R}_- \times [\tfrac{1}{2}-\delta, K]$ by
		\begin{equation*}
			T = \cup_{\tau \in [1/2-\delta, K]} \left( \Xi^{\delta}_{\widetilde{\phi} (t)} \times \left\{ s = \lambda (\p_t \Xi^{\delta}_{\widetilde{\phi} (t)}) \right\} \right).
		\end{equation*}
	Then it follows from the construction that $T \cap (\overline{R}_- \times \{t\})$ bounds a Lagrangian disk $D_-^t \subset \overline{R}_- \times \{t\}$ for any $\tau \in [\tfrac{1}{2}-\delta, K]$. Moreover $\p D_-^t$ intersects $\Gamma \times \{0\} \times \{t\}$ precisely at the navel $p$ by (\ref{eqn:negative_isotopy}), and $\p D_-^{K} = \Lambda'$. Finally, the desired bypass is given by
		\begin{equation*}
			\Delta' = \Delta^{\delta} \cup \left( \cup_{\tau \in [1/2-\delta,K]} D^t_- \right).
		\end{equation*}
	This finishes the proof of the lemma.
\end{proof}

\section{Applications to overtwisted convex hypersurfaces} \label{sec:application}

In convex hypersurface theory, it is important to be able to determine when the contact germ on a convex hypersurface is (BEM)-overtwisted.  (A hypersurface with an overtwisted contact germ will be called {\em overtwisted} and a hypersurface which is not overtwisted will be called {\em tight}.)  In dimension $3$, this question is answered in a satisfactory way by Giroux's criterion, which states that a closed convex surface $\Sigma$ is overtwisted if and only if the dividing set $\Gamma_{\Sigma}$ has a homotopically trivial component if $\Sigma \neq S^2$ and has more than one component if $\Sigma=S^2$.

The answer to this problem in dimension greater than $3$ is much more delicate. In fact, to the best of the authors' knowledge, there did not exist a single example of a convex hypersurface which was known to be overtwisted. The goal of this section is to provide the first examples of closed overtwisted convex hypersurfaces in all dimensions and compare with the $3$-dimensional situation.

Let $\Sigma^0$ be a convex hypersurface. Suppose $(\Lambda_+, \Lambda_-; D_+, D_-)$ is a bypass attachment data on $\Sigma^0$ such that $\Lambda_+$ is the standard Legendrian unknot, $D_+$ is the standard Lagrangian disk bounded by $\Lambda_+$, and $\Lambda_+$ is above $\Lambda_-$ in the sense of \autoref{subsec:Legendrian_sum_property}. It is straightforward to check that $(\Lambda_+, \Lambda_-; D_+, D_-)$ is an overtwisted bypass attachment, i.e., that $\Lambda_- \uplus \Lambda_+$ is stabilized. We strengthen that result to show that the convex hypersurface $\Sigma^1$ obtained by attaching a bypass to $\Sigma^0$ along $(\Lambda_+, \Lambda_-; D_+, D_-)$ specified as above is overtwisted.

Using \autoref{thm:bypass_attachment}, we give an explicit description of $\Sigma^1$ as follows. Write $\Sigma^i \setminus \Gamma^i = R^i_+ \cup R^i_-, i=0,1$, as usual. Then $R^1_+$ is, up to completion, obtained by attaching a Weinstein handle to $R^0_+$ along $\Lambda_- \uplus \Lambda_+$ and removing a standard neighborhood of $D_+^{-\epsilon}$; $R^1_-$ is, up to completion, obtained by attaching a Weinstein handle to $R^0_-$ along $\Lambda_- \uplus \Lambda_+$ and removing a standard neighborhood of $D_-^{\epsilon}$. Moreover, the gluing map $\psi: \p \overline{R^1_+} \to \p \overline{R^1_-}$ is induced by Legendrian handlesliding $\Lambda_+^{-\epsilon}$ up across the contact $(-1)$-surgery along $\Lambda_- \uplus \Lambda_+$ to $\Lambda_-^{\epsilon}$. See \autoref{fig:OT_bypass}.

	\begin{figure}[ht]
	\centerline{
		\begin{overpic}[scale=.3]{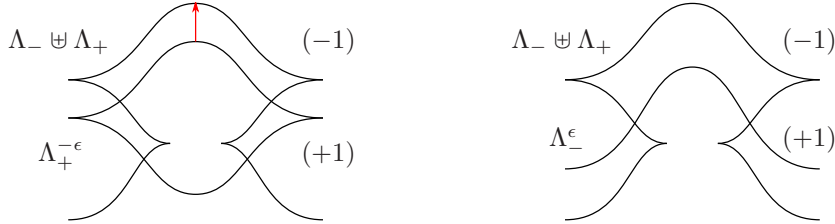}
			\put(-8,23){\small{$\Lambda_- \uplus \Lambda_+$}}
			\put(-4,8){\small{$\Lambda_+^{-\epsilon}$}}
			\put(31,8){\small{$(+1)$}}
			\put(31,23){\small{$(-1)$}}
			\put(59,23){\small{$\Lambda_- \uplus \Lambda_+$}}
			\put(64,10){\small{$\Lambda_-^{\epsilon}$}}
			\put(95,10){\small{$(+1)$}}
			\put(95,23){\small{$(-1)$}}
		\end{overpic}
		}
		\caption{Legendrian surgery diagrams of the overtwisted bypass attachment. The identification $\Gamma = \p \overline{R^1_+}$ is on the left, and $\Gamma = \p \overline{R^1_-}$ is on the right.} 
		\label{fig:OT_bypass}
	\end{figure}

\begin{theorem} \label{thm:OT_convex_hypersurface}
The convex hypersurface $\Sigma^1$ constructed above is overtwisted.
\end{theorem}

\begin{proof}
Consider the trivial bypass $\Delta^{\text{tri}}$ attached to $\Sigma^1$ along $(\Lambda_+^{\text{tri}}, \Lambda_-^{\text{tri}}; D_+^{\text{tri}}, D_-^{\text{tri}})$ as depicted in the left-hand side of \autoref{fig:OT_convex_hypersurface}.
\begin{figure}[ht]
	\centerline{
		\begin{overpic}[scale=.3]{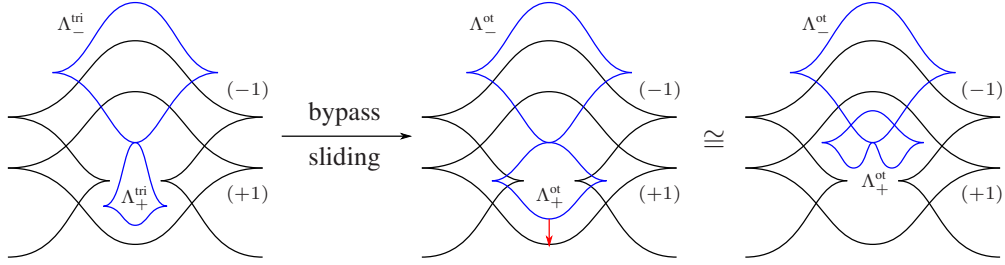}
			\put(70,11){$\cong$}
			\put(22,6){\tiny{$(+1)$}}
			\put(63.5,6){\tiny{$(+1)$}}
			\put(96,6){\tiny{$(+1)$}}
			\put(22,16.5){\tiny{$(-1)$}}
			\put(63.5,16.5){\tiny{$(-1)$}}
			\put(96,16.5){\tiny{$(-1)$}}
			\put(11.4,5.6){\tiny{$\Lambda_+^{\text{tri}}$}}
			\put(5,23){\tiny{$\Lambda_-^{\text{tri}}$}}
			\put(53.2,5.8){\tiny{$\Lambda_+^{\text{ot}}$}}
			\put(46.5,23){\tiny{$\Lambda_-^{\text{ot}}$}}
			\put(80,23){\tiny{$\Lambda_-^{\text{ot}}$}}
			\put(86,7.2){\tiny{$\Lambda_+^{\text{ot}}$}}
			\put(30.3,14){\small{bypass}}
			\put(30.3,9.4){\small{sliding}}
		\end{overpic}
	}
		\caption{The bypass rotation from $\Delta^{\text{tri}}$ to $\Delta^{\text{ot}}$ viewed in $R^1_+$. The red arrow indicates the short Reeb chord along which the Legendrian handleslide is performed. The $(-1)$-surgery is along $\Lambda_-\uplus \Lambda_+$ and the $(+1)$-surgery is along $\Lambda_+^{-\epsilon}$.}
		\label{fig:OT_convex_hypersurface}
	\end{figure}
Here $D_+^{\text{tri}}$ is the usual Lagrangian disk bounded by the standard Legendrian unknot $\Lambda_+^{\text{tri}}$, and the fact that $\Lambda_-^{\text{tri}}$ indeed bounds a Lagrangian disk $D_-^{\text{tri}}$ in $R_-^1$ can be seen by sliding $\Lambda_+^{-\epsilon}$ over $\Lambda_-\uplus \Lambda_+$ on the left-hand side of Figure~\ref{fig:OT_convex_hypersurface}. By \autoref{cor:trivial_bypass_exists}, $\Delta^{\text{tri}}$ exists in an invariant neighborhood of $\Sigma^1$.

Now observe that in $R^1_+$, $\Lambda_+^{\text{ot}}$, as given in the middle of \autoref{fig:OT_convex_hypersurface}, is to the left of $\Lambda_+^{\text{tri}}$ (in the sense of \autoref{defn:Legendrian_to_the_left}).  This can be seen by choosing $W \subset R^1_+$ to be the Liouville subdomain obtained by removing the Weinstein handle along $\Lambda_- \uplus \Lambda_+$ from $R^1_+$. Hence by \autoref{lem:bypass_sliding} the bypass $\Delta^{\text{ot}}$ along $(\Lambda_+^{\text{ot}}, \Lambda_-^{\text{ot}}; D_+^{\text{ot}}, D_-^{\text{ot}})$ exists in an invariant neighborhood of $\Sigma^1$.
\begin{figure}[ht]
	\centerline{
		\begin{overpic}[scale=.3]{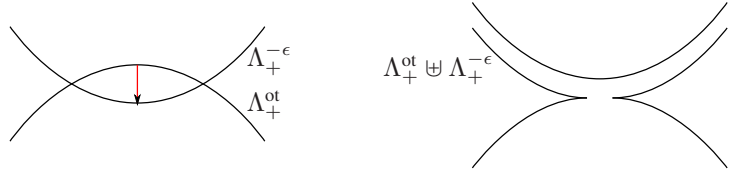}
			\put(33,14.3){\small $\Lambda_+^{-\epsilon}$}
			\put(33,8.1){\small $\Lambda_+^{\text{ot}}$}
            \put(52,13){\small $\Lambda_+^{\text{ot}}\uplus \Lambda_+^{-\epsilon}$}
		\end{overpic}
	}
		\caption{Sliding $\Lambda_+^{\text{ot}}$ below $\Lambda_+^{-\epsilon}$.}
		\label{fig: model_handleslide}
	\end{figure}

We now apply the Kirby calculus shown on the right-hand side of \autoref{fig:OT_convex_hypersurface}:  First we push the cusp edges of $\Lambda_+^{\text{ot}}$ below (the lower sheet of) $\Lambda_+^{-\epsilon}$.  Then we slide $\Lambda_+^{\text{ot}}$ below $\Lambda_+^{-\epsilon}$ using the model given in \autoref{fig: model_handleslide}; note that it looks slightly different from \autoref{fig:Legendrian_sum_front}(b) but is equivalent to it.  Finally we observe that $\Lambda_-^{\text{ot}}\uplus \Lambda_+^{\text{ot}}$ in the right-hand side of \autoref{fig:OT_convex_hypersurface} is loose. Hence $\Delta^{\text{ot}}$ is an overtwisted bypass and $\Sigma^1$ is overtwisted.
\end{proof}

We conclude this section by comparing our example in \autoref{thm:OT_convex_hypersurface} with the $3$-dimensional case. We claim that, unlike the situation in dimension $3$, the overtwistedness of a convex hypersurface in dimension at least $5$ cannot be determined solely by looking at either $R_+$ or $R_-$, or even both --- in fact how they are glued is important!

In the following we will apply the general discussions from the beginning of this subsection to a very special situation. Namely, let $\Sigma^0$ be the unit sphere in $(\R^{2n+1}, \xi_{\std})$, which is clearly convex since the radial vector field is contact and is transverse to $\Sigma^0$. Moreover the dividing set $\Gamma^0$ is contactomorphic to the standard contact structure on $S^{2n-1}$, and $R^0_\pm$ is symplectomorphic to the standard Liouville structure on $\R^{2n}$. Let $\Lambda_\pm$ both be Legendrian isotopic to the standard Legendrian unknot such that $\Lambda_+$ is above $\Lambda_-$. Let us write $\Sigma^1 = \overline{R^1_+} \cup_{\psi} \overline{R^1_-}$ for the resulting convex hypersurface of the bypass attachment, where $\psi: \p \overline{R^1_+} \to \p \overline{R^1_-}$ is the contactomorphism induced by a Legendrian handleslide as before.

Observe that in this case $R^1_+$ is in fact symplectomorphic to $R^1_-$, so we can construct another convex hypersurface $D(R^1_+) = \overline{R^1_+} \cup_{\text{id}} \overline{R^1_+}$ using the identity contactomorphism as the gluing map. Here the notation $D$ stands for ``double''. More generally, for any Liouville manifold $V$, one can form a convex hypersurface $D(V) = \overline{V} \cup_{\text{id}} \overline{V}$. The following result asserts that the contact germ on $D(V)$ is always tight. In particular it implies that $D(R^1_+)$ is tight, and hence together with \autoref{thm:OT_convex_hypersurface}, we have seen that the overtwistedness of convex hypersurfaces in dimension greater than $3$ crucially depends on the gluing contactomorphism.

\begin{lemma} \label{lemma: tight_contact_handlebody}
	For any Liouville manifold $V$, $D(V)$ is a tight convex hypersurface.
\end{lemma}

We first define the \emph{contact handlebody} $(H(V),\xi,\Gamma)$ as follows: Let $\lambda$ be the Liouville form on $V$ and let $\overline{V}$ be the ideal compactification of $V$. Define
\begin{equation*}
	H(V) = (\overline{V} \times [0,1]) / (x,t) \sim (x,t') \text{ for all } x \in \p \overline{V}, t,t' \in [0,1],
\end{equation*}
equipped with the contact structure $\xi = \ker (dt+\lambda)$ on $V \times [0,1]$ and $\xi = \ker(sdt+\eta)$ in a neighborhood of $\p \overline{V} \subset H(V)$, where $s$ is the coordinate on a collared neighborhood of $\p \overline{V} \subset \overline{V}$ and $\eta$ is the induced contact form on $\Gamma=\p \overline{V}$. It is not hard to see that $\p H(V) = D(V)$ is convex.

\begin{proof}
We will show the tightness of $(H(V), \xi,\Gamma)$, which implies the tightness of $D(V)$. To this end, we compute the sutured contact homology of $(H(V), \xi,\Gamma)$, which was defined \cite{CGHH11}. Following \cite{CGHH11}, we observe that the (sutured) completion of $(H(V),\xi,\Gamma)$ is contactomorphic to $(V \times \R, \ker(dt+\lambda))$. Hence the associated Reeb vector field $\p_t$ has no closed orbit. It follows that the sutured contact homology of $(H(V), \xi,\Gamma)$ is isomorphic to the ground field $k$, which in particular is nonzero. Hence $(H(V),\xi)$ is tight.
\end{proof}

\bibliography{mybib}
\bibliographystyle{amsalpha}

\end{document}